%% file: LLC-with-stability-main-current.tex
\documentclass{amsart}

\usepackage[utf8]{inputenc}
\usepackage{amsmath,amssymb,amsthm,amstext,mathtools}
\usepackage{geometry}
\usepackage{mathrsfs}
\usepackage{xr-hyper}
\usepackage[colorlinks=true,linkcolor=blue,citecolor=blue,urlcolor=blue]{hyperref}
\usepackage{comment}


\numberwithin{equation}{section}
\numberwithin{figure}{section}

\theoremstyle{plain}
\newtheorem{theorem}{Theorem}[section]
\newtheorem{lemma}[theorem]{Lemma}
\newtheorem{proposition}[theorem]{Proposition}

\theoremstyle{definition}
\newtheorem{definition}[theorem]{Definition}
\theoremstyle{remark}
\newtheorem{remark}[theorem]{Remark}
\newtheorem{hypothesis}[theorem]{Hypothesis}

\newcounter{thmA}

\newtheorem{theoremA}[thmA]{Theorem}

\newcommand{\C}{\mathbb{C}}

\newcommand{\J}{\mathbb{J}}

\newcommand{\Q}{\mathbb{Q}}

\newcommand{\cA}{\mathcal{A}}
\newcommand{\cB}{\mathcal{B}}

\newcommand{\cE}{\mathcal{E}}

\newcommand{\cS}{\mathcal{S}}

\newcommand{\ff}{\mathfrak{f}}

\newcommand{\fh}{\mathfrak{h}}

\newcommand{\rB}{\mathrm{B}}

\newcommand{\rG}{\mathrm{G}}
\newcommand{\rH}{\mathrm{H}}

\newcommand{\rL}{\mathrm{L}}

\newcommand{\rP}{\mathrm{P}}

\newcommand{\rS}{\mathrm{S}}
\newcommand{\rT}{\mathrm{T}}

\newcommand{\rX}{\mathrm{X}}
\newcommand{\rY}{\mathrm{Y}}

\DeclareMathOperator{\Hom}{Hom}

\DeclareMathOperator{\Fr}{Ext}

\DeclareMathOperator{\Ind}{Ind}
\DeclareMathOperator{\Irr}{Irr}
\DeclareMathOperator{\cInd}{c\text{-}Ind}
\DeclareMathOperator{\Res}{Res}

\DeclareMathOperator{\red}{red}

\DeclareMathOperator{\Stab}{Stab}

\DeclareMathOperator{\Uch}{Uch}

\DeclareMathOperator{\Ad}{Ad}

\DeclareMathOperator{\Frob}{Frob}

\DeclareMathOperator{\LLC}{LLC}
\DeclareMathOperator{\der}{der}
\DeclareMathOperator{\cusp}{cusp}
\DeclareMathOperator{\cind}{c-ind}
\DeclareMathOperator{\unip}{unip}
\DeclareMathOperator{\Lie}{Lie}
\DeclareMathOperator{\Rep}{Rep}

\newcommand{\GL}{\mathrm{GL}}
\newcommand{\SL}{\mathrm{SL}}

\title{A Pinned Local Langlands Correspondence for Depth-Zero Supercuspidal Representations}

\author{Manish Mishra}

\address{Department of Mathematics, IISER Pune,
Dr. Homi Bhabha Road, Pashan, Pune 411008, India}

\email{manish@iiserpune.ac.in}

\thanks{The author was partially supported by SERB Core Research Grant
(CRG/2022/000415) and ANRF MATRICS Grant ANRF/ARGM/2025/000365/MTR}
\date{April 2026}

\begin{document}

\begin{abstract}
We construct a pinning-normalized local Langlands correspondence for
depth-zero supercuspidal representations of a connected reductive group over a
non-archimedean local field.  After fixing a pinned splitting of the
quasi-split inner form, we obtain a canonical bijection between irreducible
depth-zero supercuspidal representations and relevant cuspidal enhanced
depth-zero Langlands parameters.

The construction is organized around the two pieces naturally present in a
depth-zero type: a tame toral part and a finite cuspidal representation of a
parahoric quotient.  The toral part is matched using the local Langlands
correspondence for maximally unramified elliptic tori and normalized
\(L\)-embeddings.  The finite cuspidal part is compared with the parameter
side by a pinned Jordan decomposition for the relevant finite reductive
quotients.  Since these quotients may be disconnected, the finite comparison
must retain the Clifford-theoretic data which records the possible extension
ambiguity.  On the connected unipotent part we use the correspondence of
Feng--Opdam--Solleveld for supercuspidal unipotent representations.  Combining
the toral, unipotent, and Clifford-theoretic pieces gives the enhanced
parameter attached to a depth-zero supercuspidal representation, and the inverse
map is obtained by reversing the same construction.

The correspondence is independent of auxiliary choices apart from the fixed
pinned normalization.  It is compatible with the tame inertial parameter
attached to the depth-zero character, with weakly unramified twists, and with
central characters via the torus correspondence.  Under the DeBacker--Reeder
logarithm hypothesis, the dimension-weighted packet distributions attached to
the resulting packets are stable.
\end{abstract}

\maketitle
\tableofcontents
\section{Introduction}
\label{sec:introduction}

The purpose of this paper is to construct a pinned form of the local
Langlands correspondence for depth-zero supercuspidal representations.  Let
\(F\) be a non-archimedean local field with residue field \(\ff\), and let
\(G\) be a connected reductive group over \(F\).  We fix the quasi-split inner
form \(G^\ast\) of \(G\), together with a pinned splitting of \(G^\ast\).  This
pinning fixes the \(W_F\)-action on the dual group \(\widehat G\), normalizes
the toral \(L\)-embeddings used in the depth-zero construction, and supplies
compatible pinnings for the finite reductive quotients which occur at vertices
of the Bruhat--Tits building.

Depth-zero supercuspidal representations are the first nontrivial testing
ground for a canonical construction of Langlands parameters beyond the torus
case.  On the representation-theoretic side, the Moy--Prasad filtration and
the theory of level-zero types reduce the construction to cuspidal
representations of reductive quotients of parahoric subgroups
\cite{MoyPrasad1994,MoyPrasad1996,Morris1999}.  If
\(x\in \mathcal B(G,F)\) is a vertex, then the full quotient
\[
        G(F)_x/G(F)_{x,0+}
\]
is the group of \(\ff\)-points of a possibly disconnected reductive group
\(\rG_x\), whose identity component is
\[
        G(F)_{x,0}/G(F)_{x,0+}.
\]
Thus the finite object attached to a depth-zero type is naturally a
representation of the full quotient, not only of its connected component.  The
component group of this full quotient is nevertheless abelian: it identifies
with \(G(F)_x/G(F)_{x,0}\), which injects into the target of the Kottwitz
homomorphism.  The disconnected Jordan-decomposition theorem used below also
requires a rational pinned-component condition: after fixing the pinning of the
connected quotient, every rational component must admit a representative whose
conjugation action preserves that pinning.  For full parahoric quotients this
is exactly the rational pinned-component condition formulated in
\cite[Hypothesis~\ref{JD-hyp:rational-pinned-component-condition}]{arotemishra}, and it is verified in
\cite[Lemma~\ref{JD-lem:parahoric-component-action-pinned}]{arotemishra}; the abelianity of the component group
is recorded in \cite[Lemma~\ref{JD-lem:parahoric-component-abelian}]{arotemishra}.  These two facts are the
structural reason why the disconnected finite Jordan decomposition used below
applies to all vertex quotients arising in the paper.

On the Galois side, a depth-zero parameter is trivial on wild inertia.  Its
restriction to tame inertia is encoded by a finite-order semisimple element of
\(\widehat G\), and the remaining unramified part is controlled by a unipotent
parameter for the centralizer of that semisimple element.  This separation of
the tame toral part from the unramified unipotent part is already visible in
DeBacker--Reeder's construction of depth-zero packets and in Kaletha's
construction of regular supercuspidal representations
\cite{DeBackerReeder2009,KalethaRegSC}.  In the present paper the toral part is
normalized by the local Langlands correspondence for tori, in a form suited to
comparison with finite tori \cite{CB2020,ImaiFiniteLC}, while the unramified
unipotent part is supplied by the correspondence of Feng--Opdam--Solleveld for
supercuspidal unipotent representations \cite{FOS2020}.  The notion of
cuspidality for enhanced parameters is the one developed by
Aubert--Moussaoui--Solleveld through Lusztig's generalized Springer theory
\cite{AMS18}.

The bridge between the two sides is finite Jordan decomposition.  For connected
finite reductive groups, Lusztig's Jordan decomposition relates characters in a
rational Lusztig series \(\cE(\rX(\ff),s)\) to unipotent characters of the dual
centralizer \(C_{\rX^\ast}(s)(\ff)\); it is rooted in Deligne--Lusztig theory
\cite{DL76,Lusztig1984Characters}, and we use the standard terminology recalled
in \cite{GM20}.  Since the finite quotients \(\rG_x(\ff)\) occurring in depth
zero are generally disconnected, the connected theory is not sufficient for our
purposes.  We use instead the pinned Jordan decomposition for disconnected
finite reductive groups with rationally pinned abelian component group from
\cite[Theorem~\ref{JD-thm:disc-JD-bijection}]{arotemishra}.
For a pinned finite reductive group \(\rX\) and a semisimple element \(s\) on
the dual side, it gives a pinning-dependent canonical bijection
\[
        J^{\mathbb P}_{\rX,s}:
        \cE(\rX(\ff),s)
        \xrightarrow{\;\sim\;}
        \Uch\bigl(C_{\rX^{\!*}}(s)(\ff)\bigr),
\]
compatible with Harish--Chandra series and with the Clifford-theoretic data
carried by the full disconnected quotient.  This finite input is what allows us
to convert an arbitrary finite cuspidal label in a depth-zero type first into a
cuspidal unipotent label on the finite dual centralizer and then, by pinned
unipotent duality, into the corresponding unipotent label on the finite
\(H\)-side.

The dependence on a pinning is part of the statement, not a cosmetic choice.
Without such a normalization, Jordan decomposition is canonical only up to the
usual choices of representatives and extensions.  The pinning fixes preferred
finite labels.  On the \(p\)-adic side, the same pinned splitting of the
quasi-split inner form fixes the Whittaker normalization of the
Langlands--Shelstad \(L\)-embeddings for the maximally unramified elliptic tori
which occur in depth zero \cite{LanglandsShelstad1987,KalethaRegSC}.  Since
ramified symmetric roots do not occur for these tori, the minimally ramified
\(\chi\)-data are canonical.  The correspondence constructed here is therefore
canonical relative to this pinned normalization.

Let us also make explicit how inner forms enter the formulation.  The paper is
written for a fixed connected reductive \(F\)-group \(G\).  The parameter set
\(\Phi^{\mathrm e}_{0,\cusp}(G)\) is the set of \(G\)-relevant cuspidal enhanced
parameters in the sense of Aubert--Moussaoui--Solleveld; it is not the
Adams--Vogan, DeBacker--Reeder, or rigid-inner-form formulation in which one
simultaneously distributes one abstract parameter over all pure or rigid inner
forms.  Thus the inner form on the representation side is already fixed.  The
auxiliary group \(H_\varphi\) which appears after removing the tame inertial
semisimple part is likewise not selected from a family of inner forms.  It is
the unramified connected group determined by the pinned Frobenius action on
\(C_{\widehat G}(\varphi(I_F))^\circ\), together with the adapted
\(L\)-embedding into \({}^LG\).  At the finite level the group
\(\mathbf H_x\) is defined as the pinned finite reductive group whose pinned dual
is
\[
        \mathbf H_x^\vee=C_{\rG_x^\vee}(s_x).
\]
This finite dual centralizer is related to the complex tame centralizer by the
depth-zero specialization \(s\mapsto s_x\); we do not require, and do not assert,
a literal isomorphism between \(\mathbf H_x^\vee\) and \(Z_{\widehat G}(s)\).  The
vertex \(x\) on the \(G\)-side is the vertex attached to the toral realization of
the same \(G\)-relevant inertial datum.  Hence the construction does not involve
an additional choice of an inner form of \(H_\varphi\), nor an independent matching
of unrelated vertices.

We now state the main results in introductory form.  Let
\[
        \Irr_{0,\cusp}(G(F))
\]
denote the set of isomorphism classes of irreducible depth-zero supercuspidal
representations of \(G(F)\), and let
\[
        \Phi^{\mathrm e}_{0,\cusp}(G)
\]
denote the set of \(\widehat G\)-conjugacy classes of relevant cuspidal enhanced
depth-zero Langlands parameters for \(G\).

\begin{theoremA}[Main results]
\label{thm:intro-main}
The fixed pinned normalization determines a canonical bijection
\[
        \LLC^{0,\cusp}_{G}:
        \Irr_{0,\cusp}(G(F))
        \xrightarrow{\;\sim\;}
        \Phi^{\mathrm e}_{0,\cusp}(G),
        \qquad
        \pi\longmapsto (\varphi_\pi,\rho_\pi).
\]
Its inverse is the map
\[
        \Pi^G_{0,\cusp}:
        \Phi^{\mathrm e}_{0,\cusp}(G)
        \xrightarrow{\;\sim\;}
        \Irr_{0,\cusp}(G(F))
\]
constructed from enhanced depth-zero cuspidal parameters.

More explicitly, if
\[
        \pi=\pi(S,\theta;\tau)
        =
        \cInd_{G(F)_x}^{G(F)}\tau
\]
is represented by a depth-zero datum, with \(x=x_S\), finite quotient
representation
\[
        \bar\tau\in \Irr(\rG_x(\ff)),
\]
and toral parameter \(\varphi_\theta:W_F\to {}^LS\), then
\[
        \varphi_\pi
        =
        \varphi_\theta\star \lambda_{x,\tau}.
\]
Here
\[
        \mathfrak u^{\mathrm{enh}}_{x,\tau}
        =
        \mathcal J^{\mathbb P_x}_{x,s_x}(\bar\tau)
        =
        [u^\circ_{x,\tau},[\alpha_{x,\tau}],E_{x,\tau}]
        \in
        \Uch^{\mathrm{enh}}(\mathbf H_x(\ff))_{\cusp}
\]
is the enriched cuspidal unipotent finite datum obtained by enriched pinned
Jordan decomposition followed by enriched pinned unipotent duality.  The
connected shadow gives the unramified FOS parameter \(\lambda_{x,\tau}\), and
the projective Clifford label \(([\alpha_{x,\tau}],E_{x,\tau})\) determines the
remaining part of the enhancement \(\rho_\pi\).  The bijection is compatible with
tame inertial restriction, weakly unramified twists, and central characters.

\end{theoremA}

The proof of the bijection is organized around two inverse constructions.  In
one direction, one starts from a relevant cuspidal enhanced depth-zero
parameter \((\varphi,\rho)\).  The tame inertial part of \(\varphi\) determines
a maximally unramified elliptic torus and a depth-zero character \(\theta\) of
that torus.  The enhancement \(\rho\), through the
Aubert--Moussaoui--Solleveld description of cuspidal enhanced parameters and
the Feng--Opdam--Solleveld unipotent correspondence, determines an enriched
cuspidal unipotent finite datum on the relevant \(H\)-side group.  Applying
inverse enriched pinned unipotent duality and then the inverse enriched pinned
Jordan decomposition gives a cuspidal representation \(\bar\tau\) of the
full finite quotient \(\rG_x(\ff)\).  Inflation to \(G(F)_x\), followed by
compact induction, gives the representation
\[
        \Pi^G_{0,\cusp}(\varphi,\rho).
\]

In the opposite direction, one starts with a depth-zero supercuspidal
representation \(\pi=\pi(S,\theta;\tau)\).  The character \(\theta\) gives the
toral parameter \(\varphi_\theta\).  Its tame inertial value determines a
semisimple element \(s_x\) in the dual of the finite reductive quotient
\(\rG_x\).  The finite representation \(\bar\tau\) belongs to the Lusztig
series labelled by \(s_x\).  Applying the enriched pinned disconnected Jordan
decomposition converts \(\bar\tau\) into an enriched cuspidal unipotent datum on
the finite dual-centralizer side; pinned unipotent duality then transports this
datum to the corresponding finite \(H\)-side group.  Its connected constituent
gives, through the Feng--Opdam--Solleveld correspondence, the unramified factor
\(\lambda_{x,\tau}\), while its Clifford cohomology class and projective label
supply the enhancement.  The two constructions are inverse because both the
toral part and the full enriched finite unipotent datum are recovered
separately.

This construction is related to earlier depth-zero parametrizations, but its
emphasis is different.  DeBacker--Reeder construct depth-zero packets and prove
stability in their setting \cite{DeBackerReeder2009}.  Kaletha's regular
supercuspidal construction gives a refined parametrization for regular data and
fixes the toral normalizations used there \cite{KalethaRegSC}.  Feng--Opdam--
Solleveld construct the unipotent supercuspidal part of the correspondence and
prove formal-degree compatibilities \cite{FOS2020}.  The present paper combines
these ingredients with the pinned Jordan decomposition for disconnected finite
reductive quotients.  The outcome is a pinned construction of the depth-zero
supercuspidal correspondence, with compatibility for tame inertial restriction,
weakly unramified twists, and central characters; under the DeBacker--Reeder
logarithm hypothesis we also prove stability of the associated
packet distributions.  Related explicit parametrizations in some simple adjoint
cases appear in \cite{fujii2025}.
More recently, Fujii constructed the local Langlands correspondence for
essentially unipotent supercuspidal representations in the framework of rigid
inner forms, and extended this correspondence to certain disconnected
reductive groups \cite{Fujii2026EUCDisconnected}.  The overlap with the present
paper lies in the unipotent centralizer input.  The emphasis here is different:
we treat depth-zero supercuspidal representations of connected groups with
arbitrary tame semisimple inertial class, and handle the resulting disconnected
centralizers at the finite level through the finite dual centralizer
\(C_{G_x^\vee}(s_x)\), pinned Jordan decomposition, and pinned unipotent duality,
rather than by constructing an LLC for a disconnected \(p\)-adic centralizer
group.

We finish by outlining the organization of the paper.  The first section fixes the toral normalizations used later: the usual local Langlands correspondence for \(F\)-tori and its depth-zero specialization to
finite tori.  Section~\ref{sec:parahoric-preliminaries} records the parahoric and
maximally unramified torus facts needed to pass between depth-zero toral
characters and finite tori.  Section~\ref{sec:lusztig-jordan-decomposition}
recalls the pinned Jordan decomposition for disconnected finite reductive
groups with abelian component group.  Section~\ref{sec:Whittaker-canonical-chi-data}
fixes the Whittaker-normalized \(\chi\)-data and \(L\)-embeddings for the
maximally unramified elliptic tori used in depth zero.  The next sections set up enhanced parameters and depth-zero data on the
representation side.  Section~\ref{sec:FOS-unipotent} recalls the part of the
Feng--Opdam--Solleveld correspondence used for the unipotent factor and isolates
the finite \(H\)-side character which will be needed in the forward construction.
Section~\ref{sec:LLC-enhanced-depth-zero-cuspidal} constructs the representation
attached to an enhanced depth-zero cuspidal parameter.
Section~\ref{sec:construction-langlands-parameter} constructs the enhanced
parameter attached to a depth-zero supercuspidal representation.  Section~\ref{sec:pinned-depth-zero-supercuspidal-LLC}
combines the two constructions and proves the pinned depth-zero
supercuspidal local Langlands correspondence.  The final section proves
stability of the resulting dimension-weighted packet distributions under the
DeBacker--Reeder logarithm hypothesis.

%----------------
\section{Toral normalizations}
\label{sec:toral-normalizations}

We recall the form of the local Langlands correspondence for tori used in this
paper.  Let \(S\) be an \(F\)-torus.  We write
\[
        \widehat S = X_*(S)\otimes_{\mathbb Z}\C^\times
\]
with its natural \(W_F\)-action.  The local Langlands correspondence for tori
gives a canonical isomorphism
\[
        \Hom_{\mathrm{cts}}(S(F),\C^\times)
        \xrightarrow{\;\sim\;}
        H^1(W_F,\widehat S).
\]
We shall use this correspondence in the following direction: a depth-zero
character
\[
        \theta:S(F)\longrightarrow \C^\times
\]
determines a Langlands parameter
\[
        \varphi_\theta:W_F\longrightarrow {}^LS .
\]

Only the depth-zero part of this correspondence is needed below.  If \(S\) is a
maximally unramified elliptic maximal \(F\)-torus of \(G\), and if \(x\) is the
corresponding vertex of the Bruhat--Tits building, then a depth-zero character
\(\theta\) is trivial on \(S(F)_{0+}\).  Hence it factors through the finite
torus
\[
        \rS_x(\ff)=S(F)_0/S(F)_{0+}.
\]
On the dual side, the tame finite-order part of \(\varphi_\theta\) determines,
via Imai's specialization map \cite{ImaiFiniteLC}, a semisimple element
\[
        s_{\theta,x}\in \rS_x^\vee(\ff).
\]
Equivalently, \(s_{\theta,x}\) is characterized by the equality
\[
        \theta_x(t)=\langle t,s_{\theta,x}\rangle,
        \qquad
        t\in \rS_x(\ff),
\]
under the perfect duality between the finite torus \(\rS_x(\ff)\) and its dual
finite torus \(\rS_x^\vee(\ff)\).  This is the finite toral datum which enters
the Lusztig series of the parahoric quotient \(\rG_x(\ff)\).

Thus, throughout the paper, the phrase ``the toral part of the parameter''
means the depth-zero character \(\theta\), its torus parameter
\(\varphi_\theta\), and the associated finite semisimple element
\(s_{\theta,x}\in \rS_x^\vee(\ff)\), all identified through the above
normalizations.
%--------------------
\section{Preliminaries about parahorics}\label{sec:parahoric-preliminaries}
\begin{definition}\cite[Fact 3.4.1 and Lemma 3.4.2]{KalethaRegSC}
Let \(G\) be a connected reductive group over a non-archimedean local field \(F\).
A maximal torus \(S\subset G\) is said to be \emph{maximally unramified} if, writing
\(S'\subset S\) for its maximal unramified subtorus, any (hence all) of the
following equivalent conditions hold:
\begin{enumerate}
  \item \(S'\) has maximal dimension among the unramified subtori of \(G\);
  \item \(S'\) is not properly contained in a larger unramified subtorus of \(G\);
  \item \(S=\operatorname{Cent}_G(S')\);
  \item \(S\times_F F^{\mathrm u}\) is a minimal Levi subgroup of \(G\times_F F^{\mathrm u}\);
  \item the inertia group \(I_F\) acts on the root system \(R(S,G)\) preserving
        some set of positive roots.
\end{enumerate}

\end{definition}

We now recall a correspondence between vertices and maximally unramified elliptic maximal tori. We will denote by \(S(F)_b\) the maximal bounded subgroup of \(S\) .
\subsection{Vertices and maximally unramified tori}
\label{subsec:maxunramified}

Let \(F^{\mathrm u}\) denote the maximal unramified extension of \(F\).  If
\(S\subset G\) is a maximal \(F\)-torus, let
\[
        S^{\mathrm{ur}}\subset S
\]
be its maximal unramified subtorus.

\begin{definition}[Maximally unramified tori]
A maximal \(F\)-torus \(S\subset G\) is called \emph{maximally unramified} if
\(S^{\mathrm{ur}}\) is maximal among the unramified subtori of \(G\).
Equivalently,
\[
        S=C_G(S^{\mathrm{ur}}),
\]
or, equivalently, \(S_{F^{\mathrm u}}\) is a minimal Levi subgroup of
\(G_{F^{\mathrm u}}\).  These conditions are also equivalent to the inertia
group \(I_F\) preserving a set of positive roots in \(R(S,G)\);
see \cite[Fact~3.4.1 and Def.~3.4.2]{KalethaRegSC}.
\end{definition}

We shall use this notion only for elliptic maximal tori.  Let
\(S\subset G\) be a maximally unramified elliptic maximal \(F\)-torus.  Since
\(S^{\mathrm{ur}}_{F^{\mathrm u}}\) is a maximal split torus of
\(G_{F^{\mathrm u}}\), it determines an apartment
\[
        \mathcal A^{\mathrm{red}}(S^{\mathrm{ur}},F^{\mathrm u})
        \subset
        \cB^{\mathrm{red}}(G,F^{\mathrm u}).
\]
This apartment is Frobenius-stable, and ellipticity of \(S\) implies that the
Frobenius action has a unique fixed point.  Following
\cite[\S3.4.1]{KalethaRegSC}, we denote this point by
\[
        x_S\in \cB^{\mathrm{red}}(G,F).
\]

For \(x=x_S\), let
\[
        \rG_x^\circ(\ff)=G(F)_{x,0}/G(F)_{x,0+}
\]
be the connected reductive quotient of the parahoric at \(x\).  Kaletha's
construction gives an elliptic maximal \(\ff\)-torus
\[
        \rS_x\subset \rG_x^\circ
\]
obtained by reducing \(S\) at \(x\); concretely,
\(\rS_x(\ff)\) is the image of \(S(F)\cap G(F)_{x,0}\) in
\(G(F)_{x,0}/G(F)_{x,0+}\), as in
\cite[Lem.~3.4.4(1)]{KalethaRegSC}.  In
Lemma~\ref{lem:supp-bounded-torus} below this group will be identified with
\(S(F)_0/S(F)_{0+}\).

\begin{lemma}[Vertices and maximally unramified elliptic tori]
\label{lem:vertex-torus-bijection}
Let \(G\) be a connected reductive group over \(F\).
\begin{enumerate}
\item If \(S\subset G\) is a maximally unramified elliptic maximal \(F\)-torus,
then the associated point
\[
        x_S\in \cB^{\mathrm{red}}(G,F)
\]
is a vertex.

\item Conversely, if \(x\in \cB^{\mathrm{red}}(G,F)\) is a vertex, then there
exists a maximally unramified elliptic maximal \(F\)-torus
\(S\subset G\) such that
\[
        x_S=x .
\]
\end{enumerate}
\end{lemma}

\begin{proof}
The first assertion is precisely \cite[Lem.~3.4.3]{KalethaRegSC}, applied to
the point associated in \cite[\S3.4.1]{KalethaRegSC} to a maximally
unramified elliptic maximal torus.

For the converse, let \(x\) be a vertex.  Choose an elliptic maximal
\(\ff\)-torus
\[
        \rS\subset \rG_x^\circ .
\]
Such tori exist in every connected reductive group over a finite field; for
instance, one may take a rational maximal torus corresponding to an elliptic
element of the absolute Weyl group.  By \cite[Lem.~3.4.4(2)]{KalethaRegSC},
every elliptic maximal torus of \(\rG_x^\circ\) arises by reduction from a
maximally unramified elliptic maximal \(F\)-torus \(S\subset G\) whose
associated point is \(x\).  Hence \(x_S=x\), as required.
\end{proof}

For later use, we write \(S(F)_b\) for the maximal bounded subgroup of
\(S(F)\), and \(S(F)_0\), \(S(F)_{0+}\) for the parahoric subgroup of \(S(F)\)
and its pro-unipotent radical.

\begin{lemma}[Bounded and parahoric subgroups of an elliptic torus]\label{lem:supp-bounded-torus}
Let $S\subset G$ be a maximally unramified elliptic maximal $F$--torus and $x=x_S$.
Then
\[
S(F)_b \;=\; S(F)_0 \;=\; S(F)\cap G(F)_{x,0}
\qquad\text{and}\qquad
S(F)_{0+}\;=\; S(F)\cap G(F)_{x,0+}.
\]
In particular $S(F)_b\subset G(F)_x$, and reduction gives a canonical identification
$S(F)_b/S(F)_{0+}\cong \rS_x(\ff)$.
\end{lemma}
\begin{proof}
Let \(\mathcal O=\mathcal O_F\).  For a torus \(T/F\), we write
\(T(F)_0\) for the parahoric subgroup of \(T(F)\), and \(T(F)_{0+}\) for its
pro-unipotent radical.  Equivalently, if \(\mathcal T^0\) is the connected
Neron--Bruhat--Tits model of \(T\), then
\[
        T(F)_0=\mathcal T^0(\mathcal O),
\]
and \(T(F)_{0+}\) is the kernel of reduction to the maximal reductive quotient
of the special fibre.  The valuation map
\[
        \nu_T:T(F)\longrightarrow
        X_*(T)_{I_F}^{\Fr}\otimes_{\mathbb Z}\mathbb R
\]
has kernel \(T(F)_0\), and its image is a discrete lattice.  Hence
\(T(F)_0\) is the maximal bounded subgroup of \(T(F)\).  Applying this to
\(T=S\), we obtain
\[
        S(F)_b=S(F)_0 .
\]

It remains to compare this subgroup with the parahoric filtration of \(G(F)\)
at \(x=x_S\).  Let \(\mathcal G_x^0\) be the connected parahoric
\(\mathcal O\)-group scheme attached to \(x\), so that
\[
        \mathcal G_x^0(\mathcal O)=G(F)_{x,0}.
\]
By the construction of the point \(x_S\) for a maximally unramified torus, the
schematic closure of \(S\) in \(\mathcal G_x^0\) is the connected parahoric
model \(\mathcal S^0\) of \(S\).  Equivalently, the inclusion \(S\subset G\)
extends to a closed immersion
\[
        \mathcal S^0\hookrightarrow \mathcal G_x^0 .
\]
Taking \(\mathcal O\)-points gives
\[
        S(F)_0
        =
        \mathcal S^0(\mathcal O)
        =
        S(F)\cap \mathcal G_x^0(\mathcal O)
        =
        S(F)\cap G(F)_{x,0}.
\]
Together with the first paragraph this proves
\[
        S(F)_b=S(F)_0=S(F)\cap G(F)_{x,0}.
\]

We next compare the \(0+\)-subgroups.  Let
\[
        \red_x:G(F)_{x,0}\longrightarrow \rG_x^\circ(\ff)
\]
be the reduction map to the reductive quotient of the special fibre of
\(\mathcal G_x^0\); its kernel is \(G(F)_{x,0+}\).  The closed immersion
\(\mathcal S^0\hookrightarrow \mathcal G_x^0\) induces a closed immersion of
reductive quotients
\[
        \rS_x\hookrightarrow \rG_x^\circ .
\]
Under this immersion, the restriction of \(\red_x\) to \(S(F)_0\) is precisely
the reduction map
\[
        S(F)_0\longrightarrow \rS_x(\ff).
\]
Therefore
\[
\begin{aligned}
        S(F)\cap G(F)_{x,0+}
        &=
        \{s\in S(F)\cap G(F)_{x,0}:\red_x(s)=1\}  \\
        &=
        \{s\in S(F)_0:\red_{\rS_x}(s)=1\}          \\
        &=
        S(F)_{0+}.
\end{aligned}
\]

Finally, since \(S(F)_0\subset G(F)_{x,0}\subset G(F)_x\), the equality
\(S(F)_b=S(F)_0\) gives \(S(F)_b\subset G(F)_x\).  The reduction map for
\(\mathcal S^0\) is surjective onto the reductive quotient, with kernel
\(S(F)_{0+}\).  Hence it induces the canonical identification
\[
        S(F)_b/S(F)_{0+}
        =
        S(F)_0/S(F)_{0+}
        \cong
        \rS_x(\ff).
\]
\end{proof}
\smallskip
\begin{lemma}[Abelianity of the component group at a vertex]
\label{lem:vertex-component-abelian}
Let \(x\in \mathcal B(G,F)\) be a vertex, and put
\[
        K_x=G(F)_x,\qquad
        K_x^0=G(F)_{x,0},\qquad
        K_x^+=G(F)_{x,0+}.
\]
Then the component group of the full reductive quotient
\[
        \rG_x(\ff)=K_x/K_x^+
\]
is abelian. More precisely, the Kottwitz homomorphism induces an injective
homomorphism
\[
        K_x/K_x^0
        \hookrightarrow
        X^\ast(Z(\widehat G)^I)^\Phi.
\]
Consequently
\[
        \rG_x(\ff)/\rG_x^\circ(\ff)
        \cong
        K_x/K_x^0
\]
is abelian.
\end{lemma}

\begin{proof}
Let
\[
        \kappa_G:G(F)\longrightarrow X^\ast(Z(\widehat G)^I)^\Phi
\]
be the Kottwitz homomorphism, and let \(G(F)_1=\ker(\kappa_G)\).  The
Bruhat--Tits--Kottwitz description of parahoric subgroups gives
\[
        G(F)_{x,0}=G(F)_x\cap G(F)_1.
\]
Hence the kernel of \(\kappa_G|_{K_x}\) is exactly \(K_x^0\), and therefore
\(\kappa_G|_{K_x}\) descends to an injective homomorphism
\[
        K_x/K_x^0\hookrightarrow X^\ast(Z(\widehat G)^I)^\Phi.
\]
The target is abelian, so \(K_x/K_x^0\) is abelian.  Since
\[
        \rG_x^\circ(\ff)=K_x^0/K_x^+,
        \qquad
        \rG_x(\ff)=K_x/K_x^+,
\]
we have
\[
        \rG_x(\ff)/\rG_x^\circ(\ff)
        \cong
        K_x/K_x^0,
\]
and the result follows.
\end{proof}

\begin{remark}[The rational pinned-component condition at vertices]
\label{rem:parahoric-rational-pinned-condition}
The disconnected Jordan decomposition used below requires more than the
abelianity in Lemma~\ref{lem:vertex-component-abelian}.  It also requires the
rational pinned-component condition of \cite[Hypothesis~\ref{JD-hyp:rational-pinned-component-condition}]{arotemishra}.
For full parahoric quotients this extra condition is verified in
\cite[Lemma~\ref{JD-lem:parahoric-component-action-pinned}]{arotemishra}: after choosing the pinning of the connected
quotient induced from the fixed global pinning and an alcove adjacent to the
vertex, every component of
\[
        G(F)_x/G(F)_{x,0+}
\]
has a representative whose conjugation action preserves that pinning.
Equivalently, the component action on the connected quotient is
pinning-preserving modulo inner conjugacy by
\[
        G(F)_{x,0}/G(F)_{x,0+}.
\]
Together with Lemma~\ref{lem:vertex-component-abelian}--or equivalently
\cite[Lemma~\ref{JD-lem:parahoric-component-abelian}]{arotemishra}--this verifies the hypotheses of the disconnected
Jordan decomposition theorem for the full vertex quotients used in the paper.
\end{remark}

\section{Lusztig's Jordan Decomposition}
\label{sec:lusztig-jordan-decomposition}

This section records the finite-field Jordan decomposition input used below.
The point of the formulation is that the reductive quotients which occur at
vertices of the building need not be connected: if
\[
K=G(F)_x,\qquad K^+=G(F)_{x,0+},
\]
then the finite quotient $K/K^+$ is usually a possibly disconnected reductive
group over the residue field $\ff$.  We therefore use the pinned Jordan
decomposition for disconnected finite reductive groups with rationally pinned
abelian component group from \cite{arotemishra}.

Throughout the section $\rX$ denotes a reductive group over $\ff$ with Frobenius
$\sigma$, not necessarily connected.  We write
\[
X:=\rX(\ff)=\rX^\sigma,\qquad
X^\circ:=(\rX^\circ)^\sigma,
\]
for the group of rational points and the rational points of the identity
component.  When $\rX$ is disconnected, the finite component group is
\[
\Omega_X:=X/X^\circ.
\]
A fixed $\sigma$--stable pinning
\[
\mathbb P=(\rX^\circ,\rB^\circ,\rT^\circ,\{x_\alpha\}_{\alpha\in\Delta})
\]
of $\rX^\circ$ determines a dual pinned datum for $(\rX^\circ)^{\!*}$.  We shall
say that $(\rX,\mathbb P)$ satisfies the \emph{rational pinned-component
condition} if
\[
X=X^\circ\cdot X_{\mathbb P},
\qquad
X_{\mathbb P}:=
\{x\in X: \Ad(x)|_{\rX^\circ}\text{ preserves }\mathbb P\}.
\]
This is the condition imposed in \cite[Hypothesis~\ref{JD-hyp:rational-pinned-component-condition}]{arotemishra}.  Under
this condition the pinned component action gives the semidirect-product dual
model
\[
\rX^{\!*}:=(\rX^\circ)^{\!*}\rtimes \Omega_X.
\]
The dual Frobenius will be denoted by $\sigma^\ast$.  Thus
$(\rX^{\!*})^\circ=(\rX^\circ)^{\!*}$.

\subsection{Lusztig and unipotent series}

We first recall the connected case, because it fixes the normalization.  If $\rX$ is connected and $s\in \rX^{\!*}(\ff)$ is semisimple, the Lusztig series $\cE(X,s)$ consists of all $\rho\in\Irr(X)$ such that
\[
\bigl\langle R^{\rX}_{\rT^{\!*}}(s),\rho\bigr\rangle_X\neq 0
\]
for some $\sigma^\ast$--stable maximal torus $\rT^{\!*}\subset \rX^{\!*}$ containing $s$.  The unipotent characters are
\[
\Uch(X):=\cE(X,1).
\]
The pinned Jordan decomposition of \cite[Theorem~\ref{JD-thm:JD-connected-groups}]{arotemishra} refines Lusztig's
Jordan decomposition and gives, for a connected $\rX$ and possibly
disconnected centralizer
\[
\rH:=C_{\rX^{\!*}}(s),
\]
a canonical bijection
\[
J^{\mathbb P}_{\rX,s}:\cE(X,s)\xrightarrow{\sim}\Uch(\rH(\ff)).
\]
With the normalization fixed in \cite[Theorem~\ref{JD-thm:JD-connected-groups}]{arotemishra}, this bijection satisfies
the Deligne--Lusztig scalar-product identity
\begin{equation}\label{eq:JD-connected-DL-formula}
\bigl\langle R^{\rX}_{\rT^{\!*}}(s),\rho\bigr\rangle_X
=
\varepsilon_{\rX}\varepsilon_{\rH}
\bigl\langle R^{\rH}_{\rT^{\!*}}(1),J^{\mathbb P}_{\rX,s}(\rho)\bigr\rangle_{\rH(\ff)}
\end{equation}
for every $\rho\in\cE(X,s)$ and every $\sigma^\ast$--stable maximal torus
$\rT^{\!*}\subset \rH$.

We now pass to the disconnected source group.  Let $s\in ((\rX^{\!*})^\circ)^{\sigma^\ast}$ be semisimple.  The disconnected Lusztig series attached to $s$ is defined by restriction to the identity component:
\begin{equation}\label{eq:disc-lusztig-series}
\cE(X,s):=
\left\{
\rho\in\Irr(X)\;\middle|\;
\Res^{X}_{X^\circ}\rho\text{ has an irreducible constituent in }
\cE(X^\circ,s)
\right\}.
\end{equation}
Equivalently,
\[
\cE(X,s)=\bigcup_{\rho^\circ\in\cE(X^\circ,s)}\Irr(X\mid \rho^\circ),
\]
where $\Irr(X\mid \rho^\circ)$ denotes irreducible characters of $X$ lying above $\rho^\circ$.  This definition depends only on the $\rX^{\!*}(\ff)$--conjugacy class of $s$, and the sets $\cE(X,s)$ partition $\Irr(X)$ as $s$ runs through semisimple $\rX^{\!*}(\ff)$--classes in $((\rX^{\!*})^\circ)^{\sigma^\ast}$.

For a possibly disconnected reductive group $\rY/\ff$, an irreducible character $\gamma\in\Irr(\rY(\ff))$ is called unipotent if its restriction to $(\rY^\circ)(\ff)$ has a unipotent irreducible constituent.  Equivalently, $\gamma$ occurs in
\[
\Ind_{(\rY^\circ)(\ff)}^{\rY(\ff)}(\gamma^\circ)
\]
for some $\gamma^\circ\in\Uch((\rY^\circ)(\ff))$.  We denote the set of such characters by $\Uch(\rY(\ff))$.

\subsection{Harish--Chandra series in the disconnected setting}

We use the standard notion of regular Levi subgroup for a disconnected reductive group.  Thus a parabolic subgroup $\rP\subset\rX$ is a closed subgroup containing a Borel subgroup of $\rX^\circ$; if $\rL^\circ$ is a Levi subgroup of $\rP^\circ$, then
\[
\rL:=N_{\rP}(\rL^\circ)
\]
is a regular Levi subgroup of $\rP$.  When $\rP$ and $\rL$ are $\sigma$--stable, Lusztig induction gives an additive homomorphism on Grothendieck groups
\[
R^{\rX}_{\rL\subset\rP}:
K_0\bigl(\Rep_{\C}(\rL(\ff))\bigr)
\longrightarrow
K_0\bigl(\Rep_{\C}(X)\bigr).
\]
Equivalently, after identifying these Grothendieck groups with groups of
virtual characters, it gives a map
\[
R^{\rX}_{\rL\subset\rP}:
\mathbb Z\,\Irr(\rL(\ff))
\longrightarrow
\mathbb Z\,\Irr(X).
\]

For a cuspidal character $\tau\in\Irr(\rL(\ff))$, we write
\[
\Irr\bigl(X,(\rL(\ff),\tau)\bigr)
\]
for the set of irreducible constituents of $R^{\rX}_{\rL\subset\rP}(\tau)$.  In the situations considered below, these Harish--Chandra series are independent of the auxiliary parabolic up to the usual canonical identifications.

Let $s\in ((\rX^{\!*})^\circ)^{\sigma^\ast}$ be semisimple.  We denote by $\Sigma_X(s)$ the set of $X$--conjugacy classes of cuspidal pairs $(\rL,\tau)$ such that $\rL$ is a $\sigma$--stable regular Levi subgroup and
\[
\tau\in\cE(\rL(\ff),s),
\]
where the Lusztig series on $\rL$ is defined by the same restriction-to-the-identity-component convention.

\subsection{Pinned enriched Jordan decomposition for disconnected groups}

We now state the finite-field input from
\cite[Theorem~\ref{JD-thm:disc-JD-bijection}]{arotemishra} in the form used in
this paper.  The rational pinned-component hypothesis appearing in the statement
is \cite[Hypothesis~\ref{JD-hyp:rational-pinned-component-condition}]{arotemishra}.
The point of the revised form is that the disconnected theorem is used in its
enriched form.  The target keeps the connected unipotent Jordan datum together
with the source Clifford cohomology class and the corresponding projective
Clifford label.  It becomes an ordinary set of unipotent characters only under
the additional ordinary-recovery condition of
\cite[Remark~\ref{JD-rmk:ordinary-disc-JD-recovery}]{arotemishra}.

\begin{theorem}[Enriched pinned Jordan decomposition for disconnected finite reductive groups]
\label{thm:disc-JD}
Let $\rG$ be a possibly disconnected reductive group over $\ff$ with Frobenius
$\sigma$.  Fix a $\sigma$--stable pinning $\mathbb P$ of $\rG^\circ$, and assume
that the finite component group
\[
\Omega_G:=\rG(\ff)/\rG^\circ(\ff)
\]
is abelian.  Assume moreover that $(\rG,\mathbb P)$ satisfies the rational
pinned-component condition
\[
\rG(\ff)=\rG^\circ(\ff)\cdot \rG_{\mathbb P}(\ff),
\qquad
\rG_{\mathbb P}(\ff):=
\{g\in \rG(\ff):\Ad(g)|_{\rG^\circ}\text{ preserves }\mathbb P\}.
\]
Whenever a regular $\sigma$--stable Levi subgroup occurs in the
Harish--Chandra decomposition below, we impose the same condition on it with
the induced pinning.  Form the pinned dual semidirect product
\[
\rG^{\!*}=(\rG^\circ)^{\!*}\rtimes \Omega_G.
\]
Let
\[
s\in ((\rG^{\!*})^\circ)^{\sigma^\ast}
\]
be semisimple, and put
\[
\rH:=C_{\rG^{\!*}}(s).
\]
Then there is a canonical bijection, depending only on the pinning $\mathbb P$,
\[
J^{\mathbb P,\mathrm{enh}}_{\rG,s}:
\cE(\rG(\ff),s)
\xrightarrow{\ \sim\ }
\Uch^{\mathrm{enh}}_{\rG,s}(\rH(\ff)).
\]
Here the target is the full enriched unipotent target of
\cite[Definition~\ref{JD-def:full-enriched-target}]{arotemishra}.  It is
obtained by decomposing the disconnected Lusztig series into Harish--Chandra
series and, on each cuspidal support, retaining the connected Jordan image, the
source Clifford class, and the projective Clifford label.

On the cuspidal part this bijection restricts to
\[
J^{\mathbb P,\mathrm{enh},\cusp}_{\rG,s}:
\cE(\rG(\ff),s)^{\cusp}
\xrightarrow{\ \sim\ }
\Uch^{\mathrm{enh}}_{\rG,s}(\rH(\ff))^{\cusp},
\]
the cuspidal enriched Jordan decomposition of
\cite[Theorem~\ref{JD-thm:disc-cusp-JD}]{arotemishra}.  Thus, if
$\rho\in\cE(\rG(\ff),s)^{\cusp}$ and
$\rho^\circ\prec\Res^{\rG(\ff)}_{\rG^\circ(\ff)}\rho$, then
\[
        J^{\mathbb P,\mathrm{enh},\cusp}_{\rG,s}(\rho)
        =
        \left[
        J^{\rG^\circ}_{s}(\rho^\circ),
        [\alpha_{\rho^\circ}],
        E_\rho
        \right].
\]
Here $[\alpha_{\rho^\circ}]$ is the Clifford cohomology class for the normal
inclusion
\[
        \rG^\circ(\ff)\triangleleft I_{\rG(\ff)}(\rho^\circ),
\]
and $E_\rho$ is the corresponding irreducible projective representation of
$I_{\rG(\ff)}(\rho^\circ)/\rG^\circ(\ff)$.  Replacing $\rho^\circ$ by a conjugate
transports all three entries, and the enriched target is the resulting
conjugacy class.

The disconnected Lusztig series is a disjoint union of Harish--Chandra series:
\begin{equation}\label{eq:disc-lusztig-HC-union}
\cE(\rG(\ff),s)
=
\bigsqcup_{[(\rL,\tau)]\in\Sigma_G(s)}
\Irr\bigl(\rG(\ff),(\rL(\ff),\tau)\bigr).
\end{equation}
For $[(\rL,\tau)]\in\Sigma_G(s)$ put
\[
\rH_{\rL}:=C_{\rL^{\!*}}(s).
\]
If $\tau^\circ$ is a connected constituent of
$\Res^{\rL(\ff)}_{\rL^\circ(\ff)}\tau$, define the enriched cuspidal datum
\[
        \mathfrak u^{\mathrm{enh}}_\tau
        :=
        \left[
        J^{\rL^\circ}_{s}(\tau^\circ),
        [\alpha_{\tau^\circ}],
        E_\tau
        \right]
        \in
        \Uch^{\mathrm{enh}}_{\rL,s}(\rH_{\rL}(\ff))^{\cusp}.
\]
Then $J^{\mathbb P,\mathrm{enh}}_{\rG,s}$ carries the Harish--Chandra series
attached to $(\rL,\tau)$ to the corresponding enriched Harish--Chandra target:
\begin{equation}\label{eq:disc-JD-HC-compatibility}
J^{\mathbb P,\mathrm{enh}}_{\rG,s}
\left(
\Irr\bigl(\rG(\ff),(\rL(\ff),\tau)\bigr)
\right)
=
\left\{
        [\rH_{\rL},\mathfrak u^{\mathrm{enh}}_\tau,\varphi]
        :
        \varphi\in\Irr(W_\tau)
\right\},
\end{equation}
where $W_\tau$ is the relative Harish--Chandra stabilizer group matched on the
two sides.  More precisely, the bijection on this Harish--Chandra series is
obtained by identifying the stabilizer relative Weyl groups and preserving the
same irreducible Weyl-group label.
\end{theorem}
\begin{remark}\label{rem:disc-JD-not-DL-formula}
Theorem~\ref{thm:disc-JD} is the form of the Jordan decomposition needed for the
full reductive quotients $G(F)_x/G(F)_{x,0+}$.  Its essential content is the
replacement of the non-canonical orbit-valued Jordan decomposition by a
pinning-dependent canonical enriched bijection.  Unlike the ordinary connected
Jordan decomposition recalled above, this theorem is not a parametrization by
ordinary unipotent characters of the full disconnected group
$C_{\rG^{\!*}}(s)(\ff)$.  Instead, the connected unipotent character, the source
Clifford class, and the projective Clifford label are retained as part of the
finite datum.  The Deligne--Lusztig scalar-product formula
\eqref{eq:JD-connected-DL-formula} is part of the connected-source normalization
recalled above; the disconnected-source theorem is used below through the
enriched canonical bijection and the Harish--Chandra compatibility
\eqref{eq:disc-JD-HC-compatibility}.  For full parahoric quotients, the abelian
component-group hypothesis and the rational pinned-component condition are
verified in
\cite[Lemmas~\ref{JD-lem:parahoric-component-abelian} and~\ref{JD-lem:parahoric-component-action-pinned}]{arotemishra}.
\end{remark}

\begin{remark}[Dependence on pinning]
The notation $J^{\mathbb P,\mathrm{enh}}_{\rG,s}$ records the only external
choice in the construction.  Once the pinning of $\rG^\circ$ is fixed and the
rational pinned-component condition is satisfied, the pinned representatives of
the component action, Lusztig preferred extensions, Clifford classes, projective
Clifford labels, and relative Weyl-group identifications are all fixed
compatibly.  In particular, for the reductive quotients attached to parahoric
normalizers, the same global pinning fixed for the quasi-split form of $G$
induces the pinnings used in Theorem~\ref{thm:disc-JD}; the required
factorization by pinning-preserving representatives is precisely
\cite[Lemma~\ref{JD-lem:parahoric-component-action-pinned}]{arotemishra}.
\end{remark}

\section{Whittaker--canonical \(\chi\)-data}
\label{sec:Whittaker-canonical-chi-data}

We fix the normalization of toral \(L\)-embeddings used throughout the
depth-zero construction.  The only case needed below is that of maximally
unramified elliptic maximal tori.  In this case the usual ambiguity in
minimally ramified \(\chi\)-data disappears: ramified symmetric roots do not
occur.

Let \(G^\ast\) be the quasi-split inner form of \(G\), equipped with the fixed
pinned splitting and compatible Whittaker datum.  Let
\[
        S\subset G
\]
be a maximally unramified elliptic maximal \(F\)-torus, and let
\(S^\ast\subset G^\ast\) be the corresponding stable conjugate.  We write
\(R(S,G)\) for the absolute root system, transported to \(G^\ast\) when needed.
For \(\alpha\in R(S,G)\), let \(F_\alpha\) be the field of definition of
\(\alpha\), and let \(F_{\pm\alpha}\) be the field of definition of the unordered
pair \(\{\alpha,-\alpha\}\).  Thus \(F_\alpha/F_{\pm\alpha}\) is either trivial
or quadratic.  In the quadratic case let
\[
        \kappa_\alpha:F_{\pm\alpha}^{\times}\longrightarrow \{\pm 1\}
\]
be the character attached to \(F_\alpha/F_{\pm\alpha}\) by local class field
theory.

Recall that a root \(\alpha\) is called \emph{asymmetric} if
\(F_\alpha=F_{\pm\alpha}\), and \emph{symmetric} otherwise.  Since \(S\) is
maximally unramified, inertia preserves a system of positive roots in
\(R(S,G)\); equivalently, \(S_{F^{\mathrm u}}\) is a minimal Levi subgroup of
\(G_{F^{\mathrm u}}\).  Hence no element of inertia sends a root to its negative.
It follows that, if \(\alpha\) is symmetric, then
\[
        F_\alpha/F_{\pm\alpha}
\]
is unramified.  Thus, in the present depth-zero setting, every symmetric root
is unramified symmetric.  We shall use this observation to choose the
minimally ramified \(\chi\)-data canonically; compare
\cite[Fact~3.4.1 and Def.~4.6.1]{KalethaRegSC}.

\begin{definition}[Canonical minimally ramified \(\chi\)-data]
\label{def:Wcan-chi}
For a maximally unramified elliptic maximal \(F\)-torus \(S\subset G\), define
\[
        \chi_S=\{\chi_\alpha\}_{\alpha\in R(S,G)}
\]
as follows.
\begin{enumerate}
\item If \(\alpha\) is asymmetric, set
\[
        \chi_\alpha=1.
\]

\item If \(\alpha\) is symmetric, then \(F_\alpha/F_{\pm\alpha}\) is unramified.
We take \(\chi_\alpha\) to be the unramified quadratic character of
\(F_\alpha^\times\).  Equivalently, for any uniformizer
\(\varpi_\alpha\in F_\alpha\),
\[
        \chi_\alpha(\varpi_\alpha)=-1 .
\]
\end{enumerate}
The characters are extended over Galois orbits by
\[
        \chi_{\sigma\alpha}=\chi_\alpha\circ \sigma^{-1},
        \qquad
        \chi_{-\alpha}=\chi_\alpha^{-1}.
\]
\end{definition}

\begin{lemma}
\label{lem:Wcan-chi-is-LS}
The collection \(\chi_S\) of Definition~\ref{def:Wcan-chi} is
Langlands--Shelstad \(\chi\)-data for \(R(S,G)\).
\end{lemma}

\begin{proof}
For asymmetric roots there is no restriction in the definition of
Langlands--Shelstad \(\chi\)-data, and the trivial character is compatible with
Galois conjugacy and with \(\alpha\mapsto -\alpha\).  If \(\alpha\) is symmetric,
then \(F_\alpha/F_{\pm\alpha}\) is unramified, and the restriction of the
unramified quadratic character of \(F_\alpha^\times\) to
\(F_{\pm\alpha}^\times\) is precisely the unramified quadratic character
\(\kappa_\alpha\) attached to \(F_\alpha/F_{\pm\alpha}\).  Thus the
Langlands--Shelstad restriction condition
\[
        \chi_\alpha|_{F_{\pm\alpha}^\times}=\kappa_\alpha
\]
is satisfied.  The stated Galois equivariance and the relation
\(\chi_{-\alpha}=\chi_\alpha^{-1}\) are built into the definition.  Hence
\(\chi_S\) is valid \(\chi\)-data in the sense of Langlands--Shelstad.
\end{proof}

Let
\[
        \widehat j:\widehat S\hookrightarrow \widehat G
\]
be an admissible dual embedding in the \(\widehat G\)-conjugacy class determined
by the stable class of \(S^\ast\subset G^\ast\) and by the fixed pinned
splitting of \(G^\ast\).  Applying the Langlands--Shelstad construction to
\(\widehat j\) and to the \(\chi\)-data \(\chi_S\) gives a
\(\widehat G\)-conjugacy class of \(L\)-embeddings
\[
        {}^LS\longrightarrow {}^LG;
\]
see \cite[\S2.6]{LanglandsShelstad1987} and
\cite[\S6.1]{KalethaRegSC}.  We fix one representative of this conjugacy class
and denote it by
\[
        \iota_S^{\mathrm{can}}:{}^LS\hookrightarrow {}^LG .
\]
Only its \(\widehat G\)-conjugacy class will enter the construction of
parameters, so this choice of representative is harmless.

\begin{proposition}[Inertial normalization of the toral \(L\)-embedding]
\label{prop:canonical-LS}
Let \(S\subset G\) be a maximally unramified elliptic maximal \(F\)-torus.  The
embedding
\[
        \iota_S^{\mathrm{can}}:{}^LS\hookrightarrow {}^LG
\]
has the following properties.

\begin{enumerate}
\item Its restriction to the dual torus is the chosen admissible embedding:
\[
        \iota_S^{\mathrm{can}}|_{\widehat S}=\widehat j.
\]

\item For every \(\sigma\in I_F\) and every \(s\in \widehat S\),
\[
        \iota_S^{\mathrm{can}}(s\rtimes \sigma)
        =
        \widehat j(s)\rtimes \sigma .
\]
Equivalently, the Langlands--Shelstad correction term for
\(\iota_S^{\mathrm{can}}\) is trivial on inertia.

\item If
\[
        \varphi_\theta:W_F\longrightarrow {}^LS
\]
is the toral parameter attached to a depth-zero character
\(\theta:S(F)\to \C^\times\), and if
\[
        \varphi_\theta(w)=a_\theta(w)\rtimes w,
        \qquad a_\theta(w)\in \widehat S,
\]
then, for \(\sigma\in I_F\),
\[
        \bigl(\iota_S^{\mathrm{can}}\circ\varphi_\theta\bigr)(\sigma)
        =
        \widehat j(a_\theta(\sigma))\rtimes \sigma .
\]
Thus the tame inertial element used below may be regarded literally as an
element
\[
        a_\theta(\iota)\in \widehat S\subset \widehat G
\]
after the fixed admissible embedding \(\widehat j\).

\item The construction is compatible with rational conjugacy of tori.  If
\(g\in G(F)\) and \(S'=gSg^{-1}\), then the two embeddings
\[
        \iota_{S'}^{\mathrm{can}}\circ {}^L(\operatorname{Ad}(g))
        \quad\text{and}\quad
        \iota_S^{\mathrm{can}}
\]
are \(\widehat G\)-conjugate.
\end{enumerate}
\end{proposition}

\begin{proof}
The first assertion is part of the Langlands--Shelstad construction of the
embedding attached to \(\widehat j\) and \(\chi_S\).

For the inertial statement, choose the positive system in \(R(S,G)\) preserved
by \(I_F\), whose existence is equivalent to \(S\) being maximally unramified.
In the Langlands--Shelstad formula, the value of the \(L\)-embedding on
\(1\rtimes w\) is written as a product of two terms: the Weyl representative
attached to the action of \(w\) on the chosen chamber, and the correction
cochain determined by the \(\chi\)-data.  When \(w=\sigma\in I_F\), the chosen
positive system is fixed by \(\sigma\).  Hence the Weyl element is trivial, and
no root is carried from the positive system to the negative system.  Therefore
the correction cochain is also trivial on \(I_F\).  This gives
\[
        \iota_S^{\mathrm{can}}(s\rtimes \sigma)
        =
        \widehat j(s)\rtimes \sigma,
        \qquad \sigma\in I_F,
\]
which proves (2).  Assertion (3) is immediate after composing this equality
with the toral parameter \(\varphi_\theta\).

Finally, rational conjugacy of tori transports the root system, the fields
\(F_\alpha\), \(F_{\pm\alpha}\), and the canonical \(\chi\)-data.  The
Langlands--Shelstad construction is functorial for this transport, up to
\(\widehat G\)-conjugacy.  This proves (4).
\end{proof}

\begin{remark}
For general tame elliptic tori one must also make choices at ramified symmetric
roots; minimally ramified \(\chi\)-data are then not unique.  Those choices are
irrelevant in the present paper because all tori used in the depth-zero
construction are maximally unramified.  The role of the fixed pinning and
Whittaker datum is therefore to keep the toral \(L\)-embeddings normalized
compatibly with the rest of the pinned construction, not to introduce additional
ramified \(\chi\)-data.
\end{remark}
\medskip
\section{Langlands parameters}
\subsection{Enhanced parameters}

An \emph{enhanced \(L\)-parameter} for \({}^LH\) is a pair
\((\phi,\rho)\), where
\[
        \phi:W_F\times \SL_2(\C)\longrightarrow {}^LH
\]
is an \(L\)-parameter, continuous on \(W_F\), algebraic on \(\SL_2(\C)\),
and Frobenius-semisimple, and where
\[
        \rho\in\Irr(\mathcal S_\phi).
\]
We use the AMS component group
\[
        \mathcal S_\phi
        :=
        \pi_0\!\left(
        Z^{1}_{\widehat H_{\mathrm{sc}}}(\phi)
        \right).
\]

Equivalently, let
\[
        G_\phi
        :=
        Z^{1}_{\widehat H_{\mathrm{sc}}}
        \bigl(\phi|_{W_F}\bigr)
\]
and let
\[
        u_\phi
        :=
        \mathrm{proj}_{\widehat H}\,
        \phi\!\left(
        1,
        \begin{psmallmatrix}1&1\\0&1\end{psmallmatrix}
        \right)
        \in \widehat H
\]
be the unipotent element attached to the \(\SL_2\)-factor.  Then
\[
        \mathcal S_\phi
        \cong
        \pi_0\!\left(
        Z_{G_\phi}(u_\phi)
        \right).
\]

\subsection{Cuspidal pair and cuspidal enhanced $L$-parameter}
Let $\mathcal G$ be a (possibly disconnected) complex reductive group, $u\in \mathcal G^\circ$ unipotent, and let $A_{\mathcal G}(u):=\pi_0(Z_{\mathcal G}(u))$. An irreducible representation $\varpi\in \mathrm{Irr}\,A_{\mathcal G}(u)$ is \emph{cuspidal} (in Lusztig's sense) if, upon restriction to $A_{\mathcal G^\circ}(u)$, it is a sum of cuspidal representations and it does not occur in generalized parabolic induction from any proper Levi of $\mathcal G^\circ$.

The following result was proved in \cite{AMS18}. For a connected reductive $H/F$, an enhanced parameter $(\phi,\rho)$ for ${}^LH$ is \emph{cuspidal}  if:
\begin{enumerate}
  \item $\phi$ is \emph{discrete}: its image is not contained in any proper Levi $L$-subgroup of ${}^LH$;
  \item via $\mathcal S_\phi\cong \pi_0(Z_{G_\phi}(u_\phi))$, the pair $(u_\phi,\rho)$ is cuspidal for $G_\phi$ in Lusztig's sense.
\end{enumerate}

\section{Set up and notations for depth-zero LLC}
\paragraph{Reduction of a depth-zero toral character.}
Let \(S\subset G\) be a maximally unramified elliptic maximal \(F\)-torus,
and let \(x=x_S\) be the associated vertex.  Let
\[
        \theta:S(F)\to \overline{\Q}_\ell^\times
\]
be a depth-zero character. Write $\rS_x/\ff$ for the reductive quotient of the connected parahoric
model of $S$ at $x$, so that
\[
        \rS_x(\ff)\simeq S(F)_0/S(F)_{0+}.
\]
Let
\[
        \theta:S(F)\longrightarrow \overline{\Q}_\ell^\times
\]
be a depth-zero character, and let
\[
        \underline\theta:\rS_x(\ff)\longrightarrow \overline{\Q}_\ell^\times
\]
be its reduction. Let
\[
        \varphi_\theta:W_F\longrightarrow {}^LS=\widehat S\rtimes W_F
\]
be the toral Langlands parameter attached to $\theta$, and write
\[
        \varphi_\theta(w)=a_\theta(w)\rtimes w,\qquad a_\theta(w)\in \widehat S,
\]
for its $\widehat S$-component. When this toral parameter is later regarded as a parameter for \(G\), we always
compose it with the Whittaker-canonical \(L\)-embedding
\[
        \iota_S^{\mathrm{can}}:{}^LS\hookrightarrow {}^LG
\]
constructed in Section~\ref{sec:Whittaker-canonical-chi-data}; its inertial
normalization is Proposition~\ref{prop:canonical-LS}.

Since $\theta$ has depth zero, $a_\theta$ is
trivial on the wild inertia group $P_F=I_F^{0+}$.

Fix a lift $\sigma_{\mathrm{eq}}\in W_F$ of arithmetic Frobenius
$\sigma_q\in W_\ff$, and let
\[
        p_{\sigma_{\mathrm{eq}}}:WD_F\longrightarrow WD_\ff
\]
be the morphism of \cite[\S 6, (6.1)]{ImaiFiniteLC}. We shall use only its
restriction to tame inertia. Let $\iota\in I_F$ be an element whose image
topologically generates $I_F/P_F$, and put
\[
        s:=a_\theta(\iota)\in \widehat S .
\]
Let \(\sigma_x\) denote the Frobenius of \(\rS_x/\ff\), and let \(\rS_x^\vee\) denote the \(\overline{\Q}_\ell\)-dual torus of the
finite reductive quotient torus \(\rS_x/\ff\). Let
\(\sigma_x^\ast\) be the dual Frobenius on \(\rS_x^\vee\).  Let
\[
        s_x\in (\rS_x^\vee)^{\sigma_x^\ast}
\]
be the semisimple element corresponding to \(\underline\theta\) under the usual
duality for finite tori. The passage from the
connected parahoric model of \(S\) to its special fibre identifies
\(\rS_x^\vee\) with the finite-field dual torus naturally associated with
the depth-zero quotient \(S(F)_0/S(F)_{0+}\).  Finally, we shall write
\[
        \operatorname{sp}_x(s)\in \rS_x^\vee
\]
for the finite-dual element obtained from the depth-zero character
\(\underline\theta\) by reduction.  More precisely, if
\[
        s=a_\theta(\iota)\in \widehat S
\]
is the inertia value of the toral parameter attached to \(\theta\), then
\(\operatorname{sp}_x(s)\) denotes the semisimple element of \(\rS_x^\vee\)
corresponding, under finite-torus duality, to the reduced character
\[
        \underline\theta:\rS_x(\ff)\to \overline{\Q}_\ell^\times .
\]

\begin{lemma}[Tame inertia element reduces to the finite dual element]
\label{lem:tame-reduces-finite-dual}
With the notation above,
\[
        s_x=\operatorname{sp}_x(s).
\]
Consequently every irreducible constituent of
\[
        R^{\rG_x}_{\rS_x}(\underline\theta)
\]
belongs to the Lusztig series
\[
        \mathcal E(\rG_x,s_x).
\]
\end{lemma}

\begin{proof}
The depth-zero condition says precisely that $\theta$ is trivial on
$S(F)_{0+}$, hence factors through
\[
        S(F)_0/S(F)_{0+}\simeq \rS_x(\ff).
\]
On the Galois side, the corresponding statement for the toral LLC is that the
$\widehat S$-valued cocycle $a_\theta$ is trivial on $P_F$. Thus
$a_\theta|_{I_F}$ factors through the tame quotient $I_F/P_F$.

The choice of $\sigma_{\mathrm{eq}}$ gives, via
\cite[\S 6, (6.1)]{ImaiFiniteLC}, a morphism
\[
        p_{\sigma_{\mathrm{eq}}}:WD_F\to WD_\ff
\]
which identifies the tame inertial quotient of $W_F$ with the inertial part of
the finite-field Weil--Deligne group. Therefore the restriction of the local
toral parameter to tame inertia may be compared directly with the finite-field
toral parameter attached to $\underline\theta$.

Let
\[
        \phi_{\underline\theta}:WD_\ff\longrightarrow {}^L\rS_x
\]
be the finite-field toral parameter corresponding to $\underline\theta$, and
write $\phi_{\underline\theta,0}$ for its $\rS_x^\vee$-component. Compatibility
of the torus LLC with reduction gives
\[
        \operatorname{sp}_x\bigl(a_\theta(\gamma)\bigr)
        =
        \phi_{\underline\theta,0}\bigl(p_{\sigma_{\mathrm{eq}}}(\gamma)\bigr),
        \qquad \gamma\in I_F .
\]
Evaluating this identity at $\gamma=\iota$ gives
\[
        \operatorname{sp}_x(s)
        =
        \phi_{\underline\theta,0}\bigl(p_{\sigma_{\mathrm{eq}}}(\iota)\bigr).
\]
By the standard duality between characters of the finite torus $\rS_x(\ff)$
and semisimple elements of the dual finite torus, the right-hand side is exactly
the element $s_x$ attached to $\underline\theta$. Hence
\[
        s_x=\operatorname{sp}_x(s).
\]

The final assertion is then the defining property of Lusztig series for finite
reductive groups: the Deligne--Lusztig virtual character
$R^{\rG_x}_{\rS_x}(\underline\theta)$ has all its irreducible constituents in
the rational series labelled by the dual semisimple element corresponding to
$\underline\theta$, namely $s_x$.
\end{proof}

\subsection{The setup}
\label{subsec:depth-zero-LLC-setup}

We fix the following notation for the depth-zero construction.  Let $F$ be a
non-archimedean local field with residue field $\ff$ of characteristic $p$, and
fix a prime $\ell\neq p$.  All finite-group characters will be regarded as
$\overline{\Q}_{\ell}$-valued.

Let $G$ be a connected reductive group over $F$, and let $G^\ast$ denote its
quasi-split inner form.  We fix a pinned splitting of $G^\ast$, hence a dual
group $\widehat G$ with its induced $W_F$-action.  We also fix, when needed, an
inner twist
\[
        \xi:G_{\overline F}\xrightarrow{\sim}G^\ast_{\overline F}
\]
in order to speak about stable conjugacy classes of maximal tori.  In this
subsection $G^\ast$ always denotes the quasi-split $F$-group, while
$\widehat G$ denotes the complex, or $\ell$-adic, dual group.

Let $S\subset G$ be a maximally unramified elliptic maximal $F$-torus, and let
\[
        x=x_S\in \cB^{\mathrm{red}}(G,F)
\]
be the vertex attached to $S$.  Put
\[
        K_x^0:=G(F)_{x,0},\qquad
        K_x^+:=G(F)_{x,0+},\qquad
        K_x:=G(F)_x .
\]
Thus $K_x^0/K_x^+$ is the group of $\ff$-points of the connected reductive
quotient at $x$, and $K_x/K_x^+$ is the group of $\ff$-points of the full,
possibly disconnected, reductive quotient.  We write
\[
        \rG_x^\circ(\ff):=K_x^0/K_x^+,\qquad
        \rG_x(\ff):=K_x/K_x^+ .
\]
The identity component of $\rG_x$ is $\rG_x^\circ$.  By Lemma~\ref{lem:vertex-component-abelian} the component group
\[
        \Omega_x:=\rG_x(\ff)/\rG_x^\circ(\ff)
\]
is abelian.  Moreover, with the pinning of $\rG_x^\circ$ induced from the fixed pinned splitting and an alcove adjacent to $x$, the full quotient satisfies the rational pinned-component condition by \cite[Lemma~\ref{JD-lem:parahoric-component-action-pinned}]{arotemishra}; see also Remark~\ref{rem:parahoric-rational-pinned-condition}.  Hence the disconnected pinned Jordan decomposition of Theorem~\ref{thm:disc-JD} applies to $\rG_x$.

The torus $S$ specializes to a maximal torus of $\rG_x^\circ$.  More precisely,
\[
        \rS_x(\ff)
        \;:=\;
        S(F)_0/S(F)_{0+}
        \;\subset\;
        \rG_x^\circ(\ff).
\]
Let $S^\ast\subset G^\ast$ be a representative of the stable conjugacy class
corresponding to $S$ under the chosen inner twist, and let
\[
        x^\ast=x_{S^\ast}\in \cB^{\mathrm{red}}(G^\ast,F)
\]
be its associated vertex.  The fixed pinned data identify the finite dual torus
$\rS_x^\vee$ with the reduction $\rS^\ast_{x^\ast}$ of $S^\ast$, and identify
the finite dual group of $\rG_x^\circ$ with the corresponding reductive quotient
on the $G^\ast$-side.  When $\rG_x$ is disconnected, we write $\rG_x^\vee$ for
the pinned semidirect-product dual used in Theorem~\ref{thm:disc-JD}; its
identity component is $(\rG_x^\circ)^\vee$.

Let
\[
        \theta:S(F)\longrightarrow \overline{\Q}_\ell^\times
\]
be a depth-zero character.  Thus $\theta$ is trivial on $S(F)_{0+}$.  We denote
by
\[
        \underline\theta:\rS_x(\ff)=S(F)_0/S(F)_{0+}
        \longrightarrow \overline{\Q}_\ell^\times
\]
the character induced by the restriction of $\theta$ to $S(F)_0$.  Notice that
we do not require $\theta$ to factor through $\rS_x(\ff)$ on all of $S(F)$:
the possible unramified part of $\theta$ is recorded by the Frobenius value of
the toral Langlands parameter, whereas the finite Deligne--Lusztig character
uses only $\underline\theta$.

Let
\[
        \varphi_\theta:W_F\longrightarrow {}^LS
\]
be the toral Langlands parameter attached to $\theta$.  Write
\[
        \varphi_\theta(w)=a_\theta(w)\rtimes w,\qquad a_\theta(w)\in \widehat S,
\]
for its $\widehat S$-component.  Since $\theta$ has depth zero, $a_\theta$ is
trivial on the wild inertia group $P_F=I_F^{0+}$.  Composing with the
Whittaker-normalized $L$-embedding
\[
        \iota_S^{\mathrm{can}}:{}^LS\hookrightarrow {}^LG
\]
of Proposition~\ref{prop:canonical-LS}, we view $\varphi_\theta$ as an
$L$-parameter for $G$.

Fix once and for all an element $\iota\in I_F$ whose image topologically
generates $I_F/P_F$, compatibly with the choice of
$\sigma_{\mathrm{eq}}$ used in Lemma~\ref{lem:tame-reduces-finite-dual}.  Set
\[
        s:=a_\theta(\iota)\in \widehat S\subset \widehat G .
\]
This element has finite order prime to $p$.  Let
\[
        s_x\in \rS_x^\vee(\overline{\Q}_\ell)
             \subset (\rG_x^\vee)^\circ
\]
be the semisimple element corresponding to $\underline\theta$ under the usual
duality for finite tori.  By Lemma~\ref{lem:tame-reduces-finite-dual},
\[
        s_x=\operatorname{sp}_x(s),
\]
where \(\operatorname{sp}_x(s)\) denotes the finite-dual semisimple element
associated, under finite-torus duality, to the reduced character
\[
        \underline\theta:
        S(F)_0/S(F)_{0+}\simeq \rS_x(\ff)
        \longrightarrow \overline{\Q}_\ell^\times .
\]
Equivalently, \(\operatorname{sp}_x(s)\) is the reduction of the inertial
value \(s=a_\theta(\iota)\) of the toral Langlands parameter to the dual torus
\(\rS_x^\vee\) of the reductive quotient. Thus the Lusztig series containing the constituents of
$R^{\rG_x}_{\rS_x}(\underline\theta)$ is
\[
        \cE(\rG_x,s_x).
\]

We shall use the following notation for the centralizers.  On the complex dual
side put
\[
        \widehat H:=Z_{\widehat G}(s),\qquad
        \widehat H^\circ:=Z_{\widehat G}(s)^\circ .
\]
This complex centralizer records the tame inertial part of the eventual
Langlands parameter.

On the finite dual side we distinguish the dual centralizer from the finite
group on which the unipotent supercuspidal representation will live.  Set
\[
        \mathbf H_x^\vee:=C_{\rG_x^\vee}(s_x),\qquad
        \mathbf H_x^{\vee,\circ}:=C_{(\rG_x^\vee)^\circ}(s_x)^\circ .
\]
Thus \(\mathbf H_x^\vee\) is the possibly disconnected finite dual centralizer
which appears as the target of Lusztig--Jordan decomposition for the Lusztig
series \(\cE(\rG_x,s_x)\), and \(s_x\in Z(\mathbf H_x^\vee)\).

Let \(\mathbf H_x\) denote the pinned finite reductive group, in the same
pinned-dual sense as in Theorem~\ref{thm:disc-JD}, whose pinned dual is
\(\mathbf H_x^\vee\).  In particular,
\[
        (\mathbf H_x^\circ)^\vee
        \cong
        \mathbf H_x^{\vee,\circ}
        =
        C_{(\rG_x^\vee)^\circ}(s_x)^\circ .
\]
The maximal torus of \(\mathbf H_x\) dual to
\(\rS_x^\vee\subset \mathbf H_x^{\vee,\circ}\) is identified with
\(\rS_x\).  Thus the torus on the \(H\)-side is \(\rS_x\), whereas the torus
inside the finite dual centralizer is \(\rS_x^\vee\).

We define \(H^{\ast,\circ}\) to be the unramified connected reductive
\(F\)-group whose hyperspecial reductive quotient is \(\mathbf H_x^\circ\).
Equivalently, \(H^{\ast,\circ}\) is the unramified lift of the pinned finite
\(\ff\)-root datum of \(\mathbf H_x^\circ\).  We denote the corresponding
hyperspecial vertex by \(x_H^\ast\), so that
\[
        \rH^{\ast,\circ}_{x_H^\ast}
        \cong
        \mathbf H_x^\circ,
        \qquad
        (\rH^{\ast,\circ}_{x_H^\ast})^\vee
        \cong
        C_{(\rG_x^\vee)^\circ}(s_x)^\circ .
\]
Under this identification the reductive quotient of the maximal torus of
\(H^{\ast,\circ}\) is
\[
        \rS^{\ast}_{H,x_H^\ast}\cong \rS_x .
\]
This is a finite-level construction.  We do not regard
\(H^{\ast,\circ}\) as a schematic centralizer of \(s\) in \(G^\ast\), since
\(s\) lies in \(\widehat G\), not in \(G^\ast\).  Nor do we assert in this
paragraph that the complex dual of \(H^{\ast,\circ}\) is literally
\(Z_{\widehat G}(s)^\circ\).  The relationship between the complex tame
centralizer and the finite group above is mediated by the depth-zero
specialization
\[
        \widehat S\longrightarrow \rS_x^\vee,
\]
which reflects
\[
        X^\ast(\rS_{x,\overline\ff})
        \simeq
        X^\ast(S_{\overline F})_{I_F,\mathrm{free}} .
\]

The hypotheses needed for the enriched disconnected Jordan decomposition are
satisfied for the full quotient
\(\rG_x(\ff)=G(F)_x/G(F)_{x,0+}\): abelianity is
Lemma~\ref{lem:vertex-component-abelian}, and the rational pinned-component
condition for the induced pinning is
\cite[Lemma~\ref{JD-lem:parahoric-component-action-pinned}]{arotemishra}.  The
finite passage used below has two layers.  First the enriched pinned Jordan
decomposition for \(\rG_x\) gives
\[
        J^{\mathbb P_x,\mathrm{enh}}_{\rG_x,s_x}:
        \cE(\rG_x(\ff),s_x)
        \xrightarrow{\ \sim\ }
        \Uch^{\mathrm{enh}}_{\rG_x,s_x}
        (\mathbf H_x^\vee(\ff)).
\]
An element of the target is represented, on the cuspidal part, by a triple
\[
        [u^\circ,[\alpha],E],
\]
where \(u^\circ\) is a connected cuspidal unipotent character of the identity
component of the finite dual centralizer, \([\alpha]\) is the source Clifford
class, and \(E\) is the corresponding projective Clifford label.  Thus the
source-side Clifford datum is not discarded.

Second we use pinned unipotent duality for the dual pair
\((\mathbf H_x,\mathbf H_x^\vee)\).  In the enriched setting this means that the
usual pinned unipotent duality is applied to the connected unipotent character,
while the stabilizer quotient, the cohomology class and the projective Clifford
label are transported along the pinned duality identification.  We write
\[
        \mathfrak D_x^{\mathrm{unip},\mathrm{enh}}:
        \Uch^{\mathrm{enh}}_{\rG_x,s_x}(\mathbf H_x^\vee(\ff))
        \xrightarrow{\ \sim\ }
        \Uch^{\mathrm{enh}}(\mathbf H_x(\ff))
\]
for this enriched unipotent duality.  Thus the finite transform relevant for the
LLC construction is
\[
        \mathcal J^{\mathbb P_x}_{x,s_x}
        :=
        \mathfrak D_x^{\mathrm{unip},\mathrm{enh}}
        \circ
        J^{\mathbb P_x,\mathrm{enh}}_{\rG_x,s_x}:
        \cE(\rG_x(\ff),s_x)
        \xrightarrow{\ \sim\ }
        \Uch^{\mathrm{enh}}(\mathbf H_x(\ff)).
\]
By definition, throughout the sequel
\[
        \mathcal J^{\mathbb P_x}_{x,s_x}
\]
denotes the enriched finite transform displayed above.  Its connected shadow is
the usual pinned finite Jordan decomposition followed by pinned unipotent
duality: it carries the toral Harish--Chandra datum
\((\rS_x,\underline\theta)\) on the \(G\)-side to the unipotent toral datum
\((\rS_x,\mathbf 1)\) on the \(H\)-side.  The enrichment records, in addition,
the relevant stabilizer quotient, the Clifford cohomology class and the
projective Clifford label.

%----------------------------------------------

\begin{lemma}[The enriched finite unipotent side attached to a depth-zero toral datum]
\label{lem:Hstar-finite}
Keep the notation of Subsection~\ref{subsec:depth-zero-LLC-setup}.  The finite
dual centralizer and the finite group on the unipotent representation side are
related by
\[
        \mathbf H_x^\vee=C_{\rG_x^\vee}(s_x),
        \qquad
        (\mathbf H_x^\circ)^\vee
        \cong
        C_{(\rG_x^\vee)^\circ}(s_x)^\circ .
\]
The unramified connected group \(H^{\ast,\circ}\) has hyperspecial reductive
quotient
\[
        \rH^{\ast,\circ}_{x_H^\ast}
        \cong
        \mathbf H_x^\circ,
        \qquad
        \rS^{\ast}_{H,x_H^\ast}
        \cong
        \rS_x .
\]
Moreover, the finite operation relevant for the LLC is the enriched composite
\[
        \mathcal J^{\mathbb P_x}_{x,s_x}
        =
        \mathfrak D_x^{\mathrm{unip},\mathrm{enh}}
        \circ
        J^{\mathbb P_x,\mathrm{enh}}_{\rG_x,s_x}:
        \cE(\rG_x(\ff),s_x)
        \xrightarrow{\ \sim\ }
        \Uch^{\mathrm{enh}}(\mathbf H_x(\ff)).
\]
If \(\rho\) is an irreducible constituent of
\[
        R^{\rG_x}_{\rS_x}(\underline\theta),
\]
then the connected shadow of
\(\mathcal J^{\mathbb P_x}_{x,s_x}(\rho)\) belongs to the unipotent toral series
of \(\mathbf H_x^\circ(\ff)\) generated by \((\rS_x(\ff),\mathbf 1)\).  In other
words, if
\[
        \mathcal J^{\mathbb P_x}_{x,s_x}(\rho)
        =
        [u^\circ_\rho,[\alpha_\rho],E_\rho],
\]
then
\[
        u^\circ_\rho
        \in
        \Irr\bigl(\mathbf H_x^\circ(\ff),(\rS_x(\ff),\mathbf 1)\bigr).
\]
\end{lemma}

\begin{proof}
The first assertions are part of the construction in
Subsection~\ref{subsec:depth-zero-LLC-setup}.  The group
\[
        \mathbf H_x^\vee=C_{\rG_x^\vee}(s_x)
\]
is the finite dual centralizer attached to the Lusztig series
\(\cE(\rG_x(\ff),s_x)\), and \(\mathbf H_x^\circ\) is the connected finite
reductive group whose pinned dual is
\[
        (\mathbf H_x^\circ)^\vee
        \cong
        C_{(\rG_x^\vee)^\circ}(s_x)^\circ .
\]
The group \(H^{\ast,\circ}\) was defined as the unramified lift of
\(\mathbf H_x^\circ\).  Therefore its hyperspecial reductive quotient is
canonically identified with \(\mathbf H_x^\circ\), and the corresponding finite
torus is identified with \(\rS_x\):
\[
        \rH^{\ast,\circ}_{x_H^\ast}\cong \mathbf H_x^\circ,
        \qquad
        \rS^{\ast}_{H,x_H^\ast}\cong \rS_x .
\]

By Lemma~\ref{lem:tame-reduces-finite-dual}, the reduced character
\[
        \underline\theta:\rS_x(\ff)\longrightarrow \C^\times
\]
corresponds under finite-torus duality to the element
\[
        s_x\in \rS_x^\vee\subset \rG_x^\vee .
\]
Hence every irreducible constituent of
\[
        R^{\rG_x}_{\rS_x}(\underline\theta)
\]
belongs to the Lusztig series \(\cE(\rG_x(\ff),s_x)\).  Applying the enriched
pinned Jordan decomposition gives an enriched unipotent datum on
\(\mathbf H_x^\vee\).  Its connected shadow is the connected pinned Jordan
image of a connected constituent of \(\rho\).  The connected Deligne--Lusztig
normalization \eqref{eq:JD-connected-DL-formula} therefore sends the
Deligne--Lusztig series attached to \((\rS_x,\underline\theta)\) to the
unipotent toral series attached to \((\rS_x^\vee,\mathbf 1)\) on the finite dual
centralizer side.  Applying enriched pinned unipotent duality transports this
connected unipotent character to the unipotent toral series attached to
\((\rS_x,\mathbf 1)\) on the \(\mathbf H_x\)-side.  The remaining entries,
\([\alpha_\rho]\) and \(E_\rho\), are simply the source Clifford class and the
projective Clifford label retained by the enriched target.
\end{proof}

\begin{definition}\label{def:depth-zero-yu-datum}
A depth-zero datum is a triple $(S,\theta;\tau)$, where:
\begin{itemize}
    \item $S\subset G$ is a maximally unramified elliptic maximal $F$--torus, and
    $x=x_S\in \mathcal B^{\mathrm{red}}(G,F)$ is the associated vertex;

    \item $\theta:S(F)\to \overline{\mathbb Q}_{\ell}^{\times}$ is a depth-zero character.
    We write
    \[
        \underline\theta:\rS_x(\mathfrak f)=S(F)_0/S(F)_{0+}
        \longrightarrow \overline{\mathbb Q}_{\ell}^{\times}
    \]
    for its reduction;

    \item $\tau$ is an irreducible representation of
    \[
        K_x:=G(F)_x
    \]
    which is trivial on $K_x^+:=G(F)_{x,0+}$ such that its restriction to the connected quotient
    \[
        \rG_x^\circ(\mathfrak f):=G(F)_{x,0}/G(F)_{x,0+}
    \]
    contains an irreducible constituent of
    \[
        \pm R_{\rS_x}^{\rG_x^\circ}(\underline\theta).
    \]
\end{itemize}
\end{definition}

The associated depth-zero supercuspidal representation is
\[
    \pi(S,\theta;\tau)
    :=
    \cInd_{G(F)_x}^{G(F)} \tau .
\]
\medskip
The next proposition packages the enriched finite operation which will be used
in the construction of the depth-zero LLC.  We first fix the notation attached
to the depth-zero datum just defined.

Let
\[
        \mathfrak d=(S,\theta;\tau)
\]
be a depth-zero datum in the sense of Definition~\ref{def:depth-zero-yu-datum},
and put \(x=x_S\).  Since \(\tau\) is trivial on \(G(F)_{x,0+}\), it factors
through the full reductive quotient
\[
        \rG_x(\ff)=G(F)_x/G(F)_{x,0+}.
\]
We denote the resulting representation by
\[
        \bar\tau\in\Irr(\rG_x(\ff)).
\]
By the definition of a depth-zero datum, the restriction of \(\bar\tau\) to
\(\rG_x^\circ(\ff)\) contains an irreducible constituent of
\[
        \pm R_{\rS_x}^{\rG_x^\circ}(\underline\theta).
\]

Let
\[
        s=a_\theta(\iota)\in\widehat S\subset\widehat G,
        \qquad
        s_x\in \rS_x^\vee\subset(\rG_x^\vee)^\circ
\]
be the tame inertial element attached to \(\theta\) and its finite-level avatar.
Set
\[
        \mathbf H_x^\vee:=C_{\rG_x^\vee}(s_x),
\]
and let \(\mathbf H_x\) be the pinned finite reductive group whose pinned dual
is \(\mathbf H_x^\vee\).  We shall use the enriched composite finite transform
\[
        \mathcal J^{\mathbb P_x}_{x,s_x}:
        \cE(\rG_x(\ff),s_x)
        \xrightarrow{\ \sim\ }
        \Uch^{\mathrm{enh}}(\mathbf H_x(\ff)).
\]

\begin{proposition}[Enriched finite Jordan transform at the vertex]
\label{prop:JD-at-vertex}
With the notation fixed above, the following assertions hold.

\begin{enumerate}
\item The finite representation \(\bar\tau\) belongs to the disconnected
Lusztig series
\[
        \bar\tau\in\cE(\rG_x(\ff),s_x).
\]
Hence the enriched finite datum
\[
        \mathfrak u_{x,\tau}^{\mathrm{enh}}
        :=
        \mathcal J^{\mathbb P_x}_{x,s_x}(\bar\tau)
        \in
        \Uch^{\mathrm{enh}}(\mathbf H_x(\ff))
\]
is well defined.  We write it, on the cuspidal part, as
\[
        \mathfrak u_{x,\tau}^{\mathrm{enh}}
        =
        [u^\circ_{x,\tau},[\alpha_{x,\tau}],E_{x,\tau}].
\]

\item The connected shadow \(u^\circ_{x,\tau}\) lies in the unipotent series of
\(\mathbf H_x^\circ(\ff)\) generated by \((\rS_x(\ff),\mathbf 1)\); equivalently,
\[
        u^\circ_{x,\tau}
        \in
        \Irr\bigl(\mathbf H_x^\circ(\ff),(\rS_x(\ff),\mathbf 1)\bigr).
\]

\item Conversely, if
\[
        \mathfrak u=[u^\circ,[\alpha],E]
        \in
        \Uch^{\mathrm{enh}}(\mathbf H_x(\ff))
\]
has connected shadow in the above unipotent toral series, then
\[
        \bar\tau_{\mathfrak u}
        :=
        \bigl(\mathcal J^{\mathbb P_x}_{x,s_x}\bigr)^{-1}(\mathfrak u)
        \in
        \cE(\rG_x(\ff),s_x)
\]
is an irreducible representation of \(\rG_x(\ff)\) whose restriction to
\(\rG_x^\circ(\ff)\) lies in the Deligne--Lusztig series generated by
\((\rS_x,\underline\theta)\).  Its inflation to \(G(F)_x\) may therefore be used
as the \(\tau\)-term of a depth-zero datum with fixed toral part \((S,\theta)\).

\item If \(\bar\tau\) is cuspidal for \(\rG_x(\ff)\), then
\(\mathfrak u_{x,\tau}^{\mathrm{enh}}\) is a cuspidal enriched unipotent datum.
\end{enumerate}

The construction is natural under \(G(F)\)-conjugacy of the datum
\((S,\theta;\tau)\).
\end{proposition}
\begin{proof}
Let \(\rho^\circ\) be an irreducible constituent of
\[
        \bar\tau|_{\rG_x^\circ(\ff)}
\]
which occurs in
\[
        \pm R_{\rS_x}^{\rG_x^\circ}(\underline\theta).
\]
By finite-torus duality, the character \(\underline\theta\) corresponds to the
semisimple element
\[
        s_x\in (\rS_x^\vee)^{\sigma_x^\ast}
        \subset ((\rG_x^\vee)^\circ)^{\sigma_x^\ast}.
\]
Lemma~\ref{lem:tame-reduces-finite-dual} identifies this element with the
finite-level specialization of \(s=a_\theta(\iota)\).  Hence the usual
Deligne--Lusztig theory for the connected quotient gives
\[
        \rho^\circ\in
        \cE(\rG_x^\circ(\ff),s_x).
\]
By the definition of the disconnected Lusztig series used here, namely by
restriction to the identity component, this implies
\[
        \bar\tau\in \cE(\rG_x(\ff),s_x).
\]
Thus the enriched pinned Jordan decomposition, followed by enriched pinned
unipotent duality, is applicable to \(\bar\tau\).  This gives the enriched datum
\(\mathfrak u_{x,\tau}^{\mathrm{enh}}\).

It remains to identify the finite toral support of its connected shadow.  On
identity components the enriched transform is exactly the connected pinned
Jordan decomposition, followed by the usual pinned unipotent duality.  Hence the
toral compatibility used in Lemma~\ref{lem:Hstar-finite} sends the
Deligne--Lusztig series attached to \((\rS_x,\underline\theta)\) on the
\(\rG_x\)-side to the unipotent toral series attached to
\((\rS_x,\mathbf 1)\) on the \(\mathbf H_x\)-side.  This proves the assertion
about \(u^\circ_{x,\tau}\).

Conversely, start from an enriched datum
\(\mathfrak u=[u^\circ,[\alpha],E]\) whose connected shadow lies in the
unipotent toral series.  Applying inverse enriched unipotent duality and then
inverse enriched pinned Jordan decomposition gives
\[
        \bar\tau_{\mathfrak u}
        =
        \bigl(\mathcal J^{\mathbb P_x}_{x,s_x}\bigr)^{-1}(\mathfrak u)
        \in
        \cE(\rG_x(\ff),s_x).
\]
The same toral compatibility, now in reverse, shows that the restriction of
\(\bar\tau_{\mathfrak u}\) to \(\rG_x^\circ(\ff)\) contains an irreducible
constituent of
\[
        \pm R_{\rS_x}^{\rG_x^\circ}(\underline\theta).
\]
Inflating \(\bar\tau_{\mathfrak u}\) along
\[
        G(F)_x\longrightarrow G(F)_x/G(F)_{x,0+}=\rG_x(\ff)
\]
therefore gives the required finite part of a depth-zero datum.

Finally suppose that \(\bar\tau\) is cuspidal for \(\rG_x(\ff)\).  The cuspidal
part of Theorem~\ref{thm:disc-JD} gives a cuspidal enriched datum on the finite
dual centralizer side; enriched unipotent duality preserves cuspidality of the
connected shadow and transports the Clifford entries unchanged.  Hence
\(\mathfrak u_{x,\tau}^{\mathrm{enh}}\) is cuspidal.  Naturality follows from
functoriality under conjugacy: conjugating the datum \((S,\theta;\tau)\)
identifies the vertices, finite reductive quotients, finite tori, reduced
characters, dual semisimple elements, pinnings, and Clifford stabilizer data.
\end{proof}
\medskip

\section{Enhanced $L$-parameters for unipotent supercuspidals}
\label{sec:FOS-unipotent}

This section recalls the part of the construction of
Feng--Opdam--Solleveld which will be used in the next section.  We use
\(F\) for the non-archimedean local field and \(\ff\) for its residue
field.  Feng--Opdam--Solleveld use \(K\) for this local field; apart
from this change of notation, the conventions below follow
\cite[\S 1]{FOS2020}.

Let \(M^\ast\) be a quasi-split connected reductive \(F\)-group which
splits over an unramified extension, and let \(M_\omega\) be an inner
form of \(M^\ast\).  Write \(\widehat M\) for the complex dual group of
\(M^\ast\).  The \(W_F\)-action on \(\widehat M\) is unramified, and we
write
\[
        {}^L M=\widehat M\rtimes W_F
\]
for the corresponding \(L\)-group.  The inner twist \(\omega\)
determines, via the Kottwitz isomorphism, the relevance character
\[
        \zeta_\omega:Z(\widehat M_{\mathrm{sc}})\longrightarrow \C^\times
\]
used in \cite[(1.7)]{FOS2020}.  Equivalently, the restriction of
\(\zeta_\omega\) to the \(W_F\)-fixed part of \(Z(\widehat M_{\mathrm{sc}})\)
is the character corresponding to the inner form \(M_\omega\).

\subsection{Unramified cuspidal enhanced parameters}

An \(L\)-parameter for \(M_\omega\) is an admissible homomorphism
\[
        \lambda:W_F\times \SL_2(\C)\longrightarrow {}^L M
\]
whose composition with the projection \({}^L M\to W_F\) is the projection
onto the first factor.  It is \emph{unramified} if
\[
        \lambda(w)=(1,w)\qquad \text{for all }w\in I_F,
\]
and it is \emph{discrete} if its image is not contained in the \(L\)-group
of any proper \(F\)-Levi subgroup relevant to \(M_\omega\).  We denote by
\(\Phi_{\mathrm{nr}}(M_\omega)\) the set of \(\widehat M\)-conjugacy classes
of unramified parameters.

Let
\[
        C_\lambda:=Z_{\widehat M}(\operatorname{im}\lambda).
\]
Following \cite[(1.5)--(1.6)]{FOS2020}, define \(\cA_\lambda\) to be the
component group of the full inverse image of
\[
        C_\lambda/Z(\widehat M)^{W_F}
        \subset \widehat M_{\mathrm{ad}}
\]
in \(\widehat M_{\mathrm{sc}}\).  Equivalently,
\[
        \cA_\lambda
        =
        \pi_0\!\left(
        Z^1_{\widehat M_{\mathrm{sc}}}(\lambda)
        \right),
\]
where
\[
        Z^1_{\widehat M_{\mathrm{sc}}}(\lambda)
        :=
        \left\{
        g\in \widehat M_{\mathrm{sc}}
        :
        g\lambda g^{-1}=\lambda b
        \text{ for some }
        b\in B^1(W_F,Z(\widehat M))
        \right\}.
\]
Here \(B^1(W_F,Z(\widehat M))\) denotes the group of \(1\)-coboundaries.

An enhancement of \(\lambda\), relevant for the inner form \(M_\omega\),
is an irreducible representation
\[
        \rho\in \Irr(\cA_\lambda,\zeta_\omega),
\]
meaning that the restriction of \(\rho\) to
\(Z(\widehat M_{\mathrm{sc}})\subset \cA_\lambda\) is
\(\zeta_\omega\)-isotypic.  This is the relevance condition for the inner
form \(M_\omega\); see \cite[(1.7)]{FOS2020}.

Let
\[
        u_\lambda
        :=
        \operatorname{pr}_{\widehat M}
        \lambda\!\left(
        1,
        \begin{psmallmatrix}
        1&1\\
        0&1
        \end{psmallmatrix}
        \right).
\]
Let \(Z^1_{\widehat M_{\mathrm{sc}}}(\lambda(W_F))\) be the inverse image of
\[
        Z_{\widehat M}(\lambda(W_F))/Z(\widehat M)^{W_F}
\]
in \(\widehat M_{\mathrm{sc}}\).  Then \cite[(1.8)]{FOS2020} identifies
\(\cA_\lambda\) with
\[
        \cA_\lambda
        =
        \pi_0\!\left(
        Z_{\,Z^1_{\widehat M_{\mathrm{sc}}}(\lambda(W_F))}(u_\lambda)
        \right).
\]
We say that \((\lambda,\rho)\) is \emph{cuspidal} if
\((u_\lambda,\rho)\) is a cuspidal pair for
\(Z^1_{\widehat M_{\mathrm{sc}}}(\lambda(W_F))\), in the sense of
\cite[Def.~6.9]{AMS18}.  Equivalently, \(\rho\) determines a
\(Z^1_{\widehat M_{\mathrm{sc}}}(\lambda(W_F))\)-equivariant cuspidal local
system on the unipotent class of \(u_\lambda\).  This condition implies
that \(\lambda\) is discrete, although not every discrete parameter
admits a cuspidal enhancement.  We denote the set of \(\widehat M\)-conjugacy
classes of unramified cuspidal enhanced parameters relevant to
\(M_\omega\) by
\[
        \Phi^{\mathrm e}_{\mathrm{nr}}(M_\omega)_{\cusp}.
\]

We shall also use the weakly unramified twisting action.  Let
\(M_\omega(F)_1\) be the kernel of the Kottwitz homomorphism.  A smooth
character of \(M_\omega(F)\) is called weakly unramified if it is trivial
on \(M_\omega(F)_1\).  The group of such characters is naturally
identified with the group denoted \(Z(\widehat M)_\theta\) in
\cite[(1.3)]{FOS2020}, where \(\theta\) is the Frobenius action on
\(\widehat M\).  Under this identification, weakly unramified twisting on the
representation side corresponds to the action
\[
        z\cdot(\lambda,\rho)=(z\lambda,\rho)
\]
on enhanced parameters; see \cite[(1.9)]{FOS2020}.

\subsection{Unipotent supercuspidals on the group side}

Let \(x\) be a vertex in the reduced Bruhat--Tits building of
\(M_\omega\).  Put
\[
        P:=M_\omega(F)_{x,0},
        \qquad
        P^+:=M_\omega(F)_{x,0+}.
\]
Thus \(P\) is the connected parahoric subgroup attached to \(x\), and
\[
        \underline M_x(\ff):=P/P^+
\]
is the connected reductive quotient.  The full stabilizer
\[
        N_x:=N_{M_\omega(F)}(P)=M_\omega(F)_x
\]
normalizes \(P^+\), and \(N_x/P^+\) is the corresponding full, possibly
disconnected, finite quotient.

Following \cite[(1.14)]{FOS2020}, an irreducible representation of \(P\)
is called unipotent, respectively cuspidal unipotent, if it is inflated
from an irreducible unipotent, respectively cuspidal unipotent,
representation of \(\underline M_x(\ff)\).  An irreducible smooth
representation of \(M_\omega(F)\) is unipotent if its restriction to some
parahoric subgroup contains such a unipotent representation.  The
supercuspidal unipotent representations are precisely those obtained
from cuspidal unipotent representations of maximal parahoric subgroups;
in the present notation this means from vertices \(x\) as above; see
\cite[(1.16)--(1.18)]{FOS2020}.

\begin{proposition}[Extending a cuspidal unipotent type and compact induction]
\label{prop:supp-depthzero-type}
Let \(M_\omega/F\) be as above, and let \(x\) be a vertex in the reduced
building of \(M_\omega\).  Put
\[
        P:=M_\omega(F)_{x,0},
        \qquad
        P^+:=M_\omega(F)_{x,0+},
        \qquad
        \underline M_x(\ff):=P/P^+,
        \qquad
        N_x:=N_{M_\omega(F)}(P).
\]
Let
\[
        \bar\sigma\in \Irr(\underline M_x(\ff))
\]
be an irreducible cuspidal unipotent representation, and let
\(\sigma\) denote its inflation to \(P\).  Then the following hold.

\begin{enumerate}
\item The representation \(\sigma\) extends to a representation
\(\sigma^N\) of \(N_x\).  Equivalently, \(\bar\sigma\) extends to a
representation of the full finite quotient \(N_x/P^+\).

\item For every smooth character
\[
        \chi:N_x/P\longrightarrow \C^\times,
\]
the compact induction
\[
        \pi_\chi
        :=
        \cInd_{N_x}^{\,M_\omega(F)}
        \bigl(\chi\otimes \sigma^N\bigr)
\]
is an irreducible depth-zero supercuspidal unipotent representation of
\(M_\omega(F)\).

\item The representations \(\pi_\chi\), as \(\chi\) varies, are precisely
the irreducible supercuspidal unipotent representations \(\pi\) of
\(M_\omega(F)\) such that
\[
        \Hom_P(\sigma,\pi)\neq 0.
\]
Equivalently, they are precisely the irreducible quotients of
\(\cInd_P^{\,M_\omega(F)}\sigma\).  The resulting set is a torsor under
the appropriate group of weakly unramified characters.  The choice of the
extension \(\sigma^N\) only chooses an origin for this torsor.
\end{enumerate}
\end{proposition}

\begin{proof}
In the notation of \cite[\S 1]{FOS2020}, \(P\) is a maximal parahoric
subgroup and \(\underline M_x(\ff)=P/P^+\) is the associated connected
finite reductive quotient.  Since \(\bar\sigma\) is cuspidal unipotent,
its inflation \(\sigma\) is one of the cuspidal unipotent parahoric
representations considered there.

The extension assertion is exactly \cite[Lem.~15.7]{FOS2020}, which
applies to reductive groups with unramified splitting field and to their
inner forms.  Thus \(\sigma\) extends to \(N_x=N_{M_\omega(F)}(P)\).  Since
\(P^+\subset P\) acts trivially on \(\sigma\), the extension also factors
through \(N_x/P^+\).  Moreover, \(N_x/P\) is the stabilizer of the
corresponding facet in the Kottwitz group, as in \cite[(1.15)]{FOS2020},
and is abelian.  Hence any two extensions of \(\sigma\) to \(N_x\) differ
by a character of \(N_x/P\).

Fix one extension \(\sigma^N\).  The construction and irreducibility of
\[
        \cInd_{N_x}^{\,M_\omega(F)}(\chi\otimes\sigma^N)
\]
are the content of \cite[(1.16)--(1.18)]{FOS2020}, with the reductive
case supplied by the reduction in \cite[\S 15]{FOS2020}.  These results
show that the representations so obtained are irreducible and
supercuspidal, and that every supercuspidal unipotent representation
containing \(\sigma\) on restriction to \(P\) arises in this way.

For depth, observe directly that \(\sigma^N\) is trivial on \(P^+\).
Therefore the compactly induced representation has nonzero
\(P^+\)-fixed vectors, namely the functions supported on \(N_x\).  Thus
\(\pi_\chi\) has depth zero.  It is unipotent because its restriction to
\(P\) contains the inflated cuspidal unipotent representation
\(\sigma\).  Finally, by Frobenius reciprocity,
\(\Hom_P(\sigma,\pi)\neq 0\) is equivalent to \(\pi\) occurring as an
irreducible quotient of \(\cInd_P^{\,M_\omega(F)}\sigma\), and the
exhaustion and simple transitivity by weakly unramified characters are
precisely \cite[(1.17)--(1.18)]{FOS2020}.  Replacing \(\sigma^N\) by
\(\eta\otimes\sigma^N\), with \(\eta\in\Hom(N_x/P,\C^\times)\), replaces
the parametrization \(\chi\mapsto\pi_\chi\) by
\(\chi\mapsto\pi_{\chi\eta}\).  Thus the packet is canonical, while the
chosen extension only fixes its labelling.
\end{proof}
\smallskip
The \emph{packet} attached to the cuspidal unipotent type \((P,\sigma)\) is
\[
        \Irr(M_\omega(F))[P,\sigma]
        :=
        \left\{
        \pi\in \Irr(M_\omega(F))_{\cusp,\unip}
        :
        \Hom_P(\sigma,\pi)\neq 0
        \right\}.
\]
Equivalently, it is the set of irreducible quotients of
\(\cInd_P^{M_\omega(F)}\sigma\).  After fixing an extension
\(\sigma^N\) of \(\sigma\) to \(N_x=N_{M_\omega(F)}(P)\), the preceding
proposition identifies this packet with
\[
        \left\{
        \cInd_{N_x}^{M_\omega(F)}(\chi\otimes\sigma^N)
        :
        \chi\in \Hom(N_x/P,\C^\times)
        \right\}/\cong .
\]
In the notation of \cite[(1.15)]{FOS2020}, when the relevant quotient is
finite one has
\[
        N_x/P \;\cong\; \Omega_{\theta,x},
\]
where \(\Omega_{\theta,x}\) is the stabilizer of the corresponding
parahoric in \(\Omega^\theta\).  Thus the choice of \(\sigma^N\) gives an
equivariant parametrization
\[
        \Omega_{\theta,x}^{\ast}
        \;\xrightarrow{\;\sim\;}\;
        \Irr(M_\omega(F))[P,\sigma],
        \qquad
        \chi\longmapsto
        \cInd_{N_x}^{M_\omega(F)}(\chi\otimes\sigma^N),
\]
as in \cite[(1.18)]{FOS2020}.  Changing the extension \(\sigma^N\) merely
translates this parametrization by a character of \(N_x/P\).  Hence the
packet itself is intrinsic, whereas the displayed labelling depends on
the chosen extension.

With the Haar-measure normalization used in
\cite{FOS2020}, the formal degree of any member of this packet is
\[
        \operatorname{fdeg}(\pi)
        =
        \frac{\dim(\sigma)}
             {|\Omega_{\theta,x}|\,\operatorname{vol}(P)}
        =
        \frac{\dim(\sigma^N)}
             {\operatorname{vol}(N_x)}
        \qquad
        \bigl(\pi\in \Irr(M_\omega(F))[P,\sigma]\bigr),
\]
whenever \(\Omega_{\theta,x}\) is finite; this is
\cite[(1.19)]{FOS2020}.  For a reductive group with non-compact split
centre, the same construction should be understood after separating off
the weakly unramified character torus, as in \cite[\S 15]{FOS2020}.  In
the applications below we use the formula only in the finite, or
equivalently compact-centre modulo the harmless split-centre reduction,
situation.

\subsection{The Feng--Opdam--Solleveld correspondence}

The main result of Feng--Opdam--Solleveld gives a canonical supply of
enhanced parameters for the supercuspidal unipotent representations just
described.  In the notation of this paper, their correspondence is a
bijection
\[
        \LLC_{\mathrm{FOS}}:
        \Irr(M_\omega(F))_{\cusp,\unip}
        \xrightarrow{\;\sim\;}
        \Phi^{\mathrm e}_{\mathrm{nr}}(M_\omega)_{\cusp},
        \qquad
        \pi\longmapsto (\lambda_\pi,\rho_\pi),
\]
where the right hand side denotes the set of \(\widehat M\)-conjugacy classes
of unramified cuspidal enhanced parameters relevant to the inner form
\(M_\omega\).  The characterization we shall need is the following form of
\cite[Thm.~2]{FOS2020}.

\begin{theorem}[Feng--Opdam--Solleveld]
\label{thm:FOS-unipotent-LLC}
Let \(M_\omega\) be an inner form of an unramified connected reductive
\(F\)-group.  There is a bijection
\[
        \LLC_{\mathrm{FOS}}:
        \Irr(M_\omega(F))_{\cusp,\unip}
        \xrightarrow{\;\sim\;}
        \Phi^{\mathrm e}_{\mathrm{nr}}(M_\omega)_{\cusp}
\]
with the following properties.

\begin{enumerate}
\item If \(G\) is semisimple, then the formal degree of \(\pi\) agrees
with the adjoint \(\gamma\)-factor of \(\lambda_\pi\), up to a rational
factor depending only on the enhancement \(\rho_\pi\).

\item The bijection is equivariant for weakly unramified twisting.  More
precisely, if \(\eta\) is a weakly unramified character of \(M_\omega(F)\)
and \(z_\eta\in Z(\widehat M)^\theta\) is the corresponding element under
\cite[(1.3)]{FOS2020}, then
\[
        \LLC_{\mathrm{FOS}}(\eta\otimes\pi)
        =
        \bigl(z_\eta\lambda_\pi,\rho_\pi\bigr).
\]

\item The bijection is equivariant for \(W_F\)-automorphisms of the based
root datum.  In the semisimple case this is equivariance for diagram
automorphisms.

\item The bijection is compatible with almost direct products of
reductive groups.

\item Let \(Z_s\) be the maximal \(F\)-split central torus of \(M_\omega\),
and put
\[
        M_{\der,Z}:=\bigl(M_\omega/Z_s\bigr)_{\der}.
\]
If \(Z_s(F)\) acts trivially on \(\pi\), then \(\pi\) may be viewed as a
representation of \((M_\omega/Z_s)(F)\), and its restriction to
\(M_{\der,Z}(F)\) corresponds to the image of \(\lambda_\pi\) under the
natural map of \(L\)-groups
\[
        {}^L(M_\omega/Z_s)\longrightarrow {}^L M_{\der,Z} .
\]

\item The preceding compatibility induces bijections between the
supercuspidal unipotent representations lying in the relevant
\(L\)-packet on the \(M_\omega\)-side and those lying in the corresponding
\(L\)-packet on the \(M_{\der,Z}\)-side.
\end{enumerate}

Moreover, for a fixed \(\pi\), properties \((1)\), \((2)\), \((4)\), and
\((5)\) determine the ordinary parameter \(\lambda_\pi\) uniquely up to
weakly unramified twist.
\end{theorem}

For later reference we record the consequence for a fixed type
\((P,\sigma)\).  Choose an extension \(\sigma^N\) of \(\sigma\) to \(N_x\)
and put
\[
        \pi_1:=\cInd_{N_x}^{M_\omega(F)}\sigma^N .
\]
For a character \(\chi\in \Hom(N_x/P,\C^\times)\), choose a weakly
unramified character \(\eta_\chi\) of \(M_\omega(F)\) whose restriction to
\(N_x/P\) is \(\chi\).  Then
\[
        \cInd_{N_x}^{M_\omega(F)}(\chi\otimes\sigma^N)
        \;\cong\;
        \eta_\chi\otimes \pi_1,
\]
and Theorem~\ref{thm:FOS-unipotent-LLC} gives
\[
        \LLC_{\mathrm{FOS}}
        \!\left(
        \cInd_{N_x}^{M_\omega(F)}(\chi\otimes\sigma^N)
        \right)
        =
        \bigl(z_{\eta_\chi}\lambda_{\pi_1},\rho_{\pi_1}\bigr).
\]
Thus the image of the packet \(\Irr(M_\omega(F))[P,\sigma]\) is a weakly
unramified orbit of unramified cuspidal enhanced parameters.  The
enhancement is unchanged under the weakly unramified twisting action,
whereas the ordinary parameter is multiplied by the corresponding central
element of \(\widehat M\).

\subsection{Parameter-side unipotentization and the enriched finite \(H\)-side datum}

We now isolate the input from FOS which will be used in
Section~\ref{sec:LLC-enhanced-depth-zero-cuspidal}.  The relevant operation has
two logically separate parts.  First one removes the tame inertial semisimple
part of a depth-zero parameter and views the remaining Frobenius--\(\SL_2\)-datum
as an ordinary unramified parameter for an auxiliary unramified connected group.
This is the part to which FOS applies.  Second, the original enhancement
contains a residual Clifford datum measuring how the full component group acts
on the connected unramified datum.  In the enriched version this residual datum
is kept as a projective Clifford label; it is not forced to become an ordinary
character of a disconnected finite group.

\begin{proposition}[Ordinary parameter-side unipotentization]
\label{prop:parameter-side-unipotentization}
Let \((\varphi,\rho)\) be a relevant cuspidal depth-zero enhanced parameter for
\(G\).  Write
\[
        \varphi(w)=g_w\rtimes w,
        \qquad w\in W_F,
\]
and put
\[
        \widehat H_\varphi
        :=
        C_{\widehat G}\bigl(\varphi(I_F)\bigr)^\circ
        =
        \{h\in\widehat G: h\varphi(i)=\varphi(i)h
          \text{ for all }i\in I_F\}^{\circ}.
\]
Then \(\widehat H_\varphi\) is a connected reductive subgroup of
\(\widehat G\), and \(\varphi(\Frob)\) normalizes \(\widehat H_\varphi\).
After the fixed pinned normalization, the outer class of the automorphism of
\(\widehat H_\varphi\) induced by \(\varphi(\Frob)\) determines an unramified
\(F\)-group \(H_\varphi\) with dual group \(\widehat H_\varphi\).  Moreover
\(\varphi\) determines an unramified \(L\)-parameter
\[
        \lambda_\varphi:W_F\times\SL_2(\C)\longrightarrow {}^LH_\varphi
\]
characterized, up to \(\widehat H_\varphi\)-conjugacy, by
\[
        \lambda_\varphi|_{I_F}=1,
        \qquad
        \lambda_\varphi|_{\SL_2}=\varphi|_{\SL_2}
        \ \text{viewed inside }\widehat H_\varphi .
\]
The inclusion \(\widehat H_\varphi\subset\widehat G\) does not, by
itself, give a literal inclusion of \(L\)-groups
\({}^LH_\varphi\subset{}^LG\).  The same pinned normalization determines a
\(\varphi\)-adapted \(L\)-embedding
\[
        \xi_\varphi:{}^LH_\varphi\longrightarrow{}^LG,
\]
characterized by
\[
        \xi_\varphi|_{\widehat H_\varphi}
        =
        \widehat H_\varphi\hookrightarrow\widehat G,
        \qquad
        \xi_\varphi\circ\lambda_\varphi=\varphi .
\]
Equivalently, if
\[
        \lambda_\varphi(w,1)=\ell_\varphi(w)\rtimes w,
        \qquad
        \varphi(w,1)=g_w\rtimes w,
\]
then
\[
        \xi_\varphi(h\rtimes w)
        =
        h\,\ell_\varphi(w)^{-1}g_w\rtimes w .
\]
The construction gives
\[
        Z_{\widehat H_\varphi}(\lambda_\varphi)
        =
        Z_{\widehat G}(\varphi)\cap\widehat H_\varphi
\]
and weakly unramified twists of \(\varphi\) correspond to weakly unramified
twists of \(\lambda_\varphi\).
\end{proposition}

\begin{proof}
We first isolate the inertial centralizer.  Since \(\varphi\) has depth zero,
wild inertia acts trivially and the tame inertial image is a finite group of
semisimple elements.  Hence
\[
        \widehat H_\varphi
        =
        C_{\widehat G}\bigl(\varphi(I_F)\bigr)^\circ
\]
is connected reductive.  The notation means the centralizer inside
\(\widehat G\) of the subgroup \(\varphi(I_F)\subset{}^LG\), in the semidirect
product sense.  Thus \(h\in\widehat G\) belongs to this centralizer if and only
if
\[
        h\varphi(i)=\varphi(i)h
        \qquad\text{for every }i\in I_F.
\]
When the inertia action on \(\widehat G\) is trivial, this is the ordinary
centralizer of the projected inertial image in \(\widehat G\).

Let \(\Frob\) be a Frobenius element.  Conjugating the equality
\(h\varphi(i)=\varphi(i)h\) by \(\varphi(\Frob)\) shows that
\(\varphi(\Frob)h\varphi(\Frob)^{-1}\) centralizes
\(\varphi(\Frob i\Frob^{-1})\), and hence centralizes \(\varphi(I_F)\), since
\(\Frob I_F\Frob^{-1}=I_F\).  Therefore \(\varphi(\Frob)\) normalizes
\(\widehat H_\varphi\).

Let
\[
        \alpha_\varphi:=\Ad(\varphi(\Frob))|_{\widehat H_\varphi}
\]
denote the induced automorphism of \(\widehat H_\varphi\), with the original
\(W_F\)-action included.  The fixed pinning chooses a pinned representative
\(\theta_\varphi\) of the outer class of \(\alpha_\varphi\).  Thus there is an
\(h_\varphi\in\widehat H_\varphi\), unique up to \(\theta_\varphi\)-twisted
conjugacy, such that
\[
        \alpha_\varphi=\Ad(h_\varphi)\circ\theta_\varphi .
\]
The pinned automorphism \(\theta_\varphi\) defines an unramified \(F\)-group
\(H_\varphi\) with dual group \(\widehat H_\varphi\).  Its \(L\)-group is
\[
        {}^LH_\varphi=\widehat H_\varphi\rtimes W_F,
\]
where inertia acts trivially and \(\Frob\) acts through \(\theta_\varphi\).

We define an unramified parameter by setting
\[
        \lambda_\varphi(i)=1\rtimes i
        \qquad (i\in I_F),
        \qquad
        \lambda_\varphi(\Frob)=h_\varphi\rtimes \Frob .
\]
The \(\SL_2\)-image of \(\varphi\) centralizes \(\varphi(I_F)\), because the two
factors in \(W_F\times\SL_2(\C)\) commute.  Since \(\SL_2(\C)\) is connected,
this image lies in \(\widehat H_\varphi\), and we put
\[
        \lambda_\varphi|_{\SL_2}=\varphi|_{\SL_2}.
\]
Frobenius compatibility follows from the equality
\(\alpha_\varphi=\Ad(h_\varphi)\circ\theta_\varphi\) and from the fact that
\(\varphi(\Frob)\) commutes with the \(\SL_2\)-image in the domain.  Hence
\(\lambda_\varphi\) is an unramified parameter for \(H_\varphi\).

The associated adapted \(L\)-embedding is obtained as follows.  Write
\[
        \lambda_\varphi(w,1)=\ell_\varphi(w)\rtimes w,
        \qquad
        \varphi(w,1)=g_w\rtimes w .
\]
Then
\[
        \xi_\varphi(h\rtimes w)
        =
        h\,\ell_\varphi(w)^{-1}g_w\rtimes w
\]
defines an \(L\)-embedding
\[
        \xi_\varphi:{}^LH_\varphi\longrightarrow{}^LG .
\]
It restricts to the inclusion on \(\widehat H_\varphi\) and satisfies
\[
        \xi_\varphi\circ\lambda_\varphi=\varphi .
\]
Thus the unramified centralizer model is viewed inside \({}^LG\) through
\(\xi_\varphi\).

The centralizer identity is direct.  An element \(h\in\widehat H_\varphi\)
centralizes \(\lambda_\varphi\) if and only if it centralizes the \(\SL_2\)-image
and satisfies
\[
        h h_\varphi=h_\varphi\theta_\varphi(h).
\]
This latter condition is equivalent to \(\alpha_\varphi(h)=h\), namely to the
condition that \(h\) centralizes \(\varphi(\Frob)\) in \({}^LG\).  Since
\(h\in\widehat H_\varphi\) already centralizes \(\varphi(I_F)\), the displayed
identity follows.  Finally, a weakly unramified twist does not change
\(\varphi|_{I_F}\); it only multiplies the Frobenius element by a central
unramified element.  This is exactly weakly unramified twisting of
\(\lambda_\varphi\).
\end{proof}
The preceding proposition constructs only the ordinary unramified parameter
\(\lambda_\varphi\).  We now fix the notation for the enhancement.  The point is
that the FOS correspondence applies to an ordinary enhanced unramified parameter
for the connected unramified group \(H_\varphi\), while the original enhancement
\(\rho\) may still contain residual component-group information.  We keep that
residual information as Clifford data.

We use throughout the Arthur--Kaletha--AMS convention for component groups,
recalled in Section~\ref{sec:FOS-unipotent}.  Thus, for an \(L\)-parameter
\(\psi\) of an \(F\)-group \(M\), the enhancement group is
\[
        \mathcal S^M_\psi
        :=
        \pi_0\!\left(Z^1_{\widehat M_{\mathrm{sc}}}(\psi)\right),
\]
where \(Z^1_{\widehat M_{\mathrm{sc}}}(\psi)\) denotes the full inverse image in
\(\widehat M_{\mathrm{sc}}\) of the appropriate centralizer in
\(\widehat M_{\mathrm{ad}}\), with the central coboundary convention used in
\cite[(1.5)--(1.8)]{FOS2020} and \cite[Def.~6.9]{AMS18}.  Equivalently, if
\[
        G^M_\psi
        :=
        Z^1_{\widehat M_{\mathrm{sc}}}\bigl(\psi(W_F)\bigr),
        \qquad
        u_\psi
        :=
        \operatorname{pr}_{\widehat M}
        \psi\!\left(1,
        \begin{psmallmatrix}1&1\\0&1\end{psmallmatrix}
        \right),
\]
then
\[
        \mathcal S^M_\psi
        \simeq
        \pi_0\!\left(Z_{G^M_\psi}(u_\psi)\right).
\]

Apply this convention to the unramified parameter \(\lambda_\varphi\) of
\(H_\varphi\).  Let \(\widehat H_{\varphi,\mathrm{sc}}\) be the simply connected
cover of the derived group of \(\widehat H_\varphi\), and put
\[
        u_\varphi=u_{\lambda_\varphi}
        :=
        \operatorname{pr}_{\widehat H_\varphi}
        \lambda_\varphi\!\left(
        1,
        \begin{psmallmatrix}1&1\\0&1\end{psmallmatrix}
        \right).
\]
Since \(\lambda_\varphi|_{\SL_2}=\varphi|_{\SL_2}\) inside
\(\widehat H_\varphi\), this is the same unipotent element as the projection of
\(\varphi(1,\begin{psmallmatrix}1&1\\0&1\end{psmallmatrix})\) to
\(\widehat G\), now viewed in \(\widehat H_\varphi\).  We write
\[
        \cA_{\lambda_\varphi}
        :=
        \pi_0\!\left(
        Z^1_{\widehat H_{\varphi,\mathrm{sc}}}(\lambda_\varphi)
        \right).
\]
Equivalently, with
\[
        G_{\lambda_\varphi}^{H_\varphi}
        :=
        Z^1_{\widehat H_{\varphi,\mathrm{sc}}}
        \bigl(\lambda_\varphi(W_F)\bigr),
\]
one has
\[
        \cA_{\lambda_\varphi}
        \simeq
        \pi_0\!\left(
        Z_{G_{\lambda_\varphi}^{H_\varphi}}(u_\varphi)
        \right).
\]
Thus \(\cA_{\lambda_\varphi}\) is precisely the ordinary FOS enhancement group
for the unramified parameter \(\lambda_\varphi\) of \(H_\varphi\).  If one uses
the notation \(\mathcal S_{\lambda_\varphi}^{H_\varphi}\) for the AMS component
group of this unramified parameter, then
\[
        \mathcal S_{\lambda_\varphi}^{H_\varphi}
        =
        \cA_{\lambda_\varphi}.
\]

We now compare this FOS group with the component group carrying the original
enhancement.  The original enhancement is a representation of
\[
        \mathcal S_\varphi
        =
        \pi_0\!\left(Z^1_{\widehat G_{\mathrm{sc}}}(\varphi)\right),
\]
as in the convention fixed in Section~\ref{sec:FOS-unipotent}.  The inclusion
\(\widehat H_\varphi\subset \widehat G\), together with the centralizer identity
of Proposition~\ref{prop:parameter-side-unipotentization}, gives a component
comparison in which the FOS group \(\cA_{\lambda_\varphi}\) is the subgroup
coming from the identity component of the inertial centralizer, while the full
enhancement group still acts on it by conjugation.  More explicitly, let
\[
        \xi_\varphi:{}^LH_\varphi\longrightarrow{}^LG
\]
be the adapted \(L\)-embedding from
Proposition~\ref{prop:parameter-side-unipotentization}.  Since
\[
        \xi_\varphi\bigl(\lambda_\varphi(W_F)\bigr)=\varphi(W_F),
\]
the full transported group may be written directly in the original dual group as
\[
        G_\varphi^{\mathrm{un}}
        :=
        Z^1_{\widehat G_{\mathrm{sc}}}\bigl(\varphi(W_F)\bigr).
\]
Equivalently,
\[
        G_\varphi^{\mathrm{un}}
        =
        Z^1_{\widehat G_{\mathrm{sc}}}
        \bigl(\varphi(I_F),\xi_\varphi(\lambda_\varphi(W_F))\bigr),
\]
where the right hand side denotes the inverse image in \(\widehat G_{\mathrm{sc}}\)
of the centralizer in \(\widehat G_{\mathrm{ad}}\) of the subgroup generated by
\(\varphi(I_F)\) and \(\xi_\varphi(\lambda_\varphi(W_F))\).  We define
\[
        \mathcal S_\varphi^{\mathrm{un}}
        :=
        \pi_0\!\left(Z_{G_\varphi^{\mathrm{un}}}(u_\varphi)\right).
\]
Thus the centralizer comparison identifies
\[
        \mathcal S_\varphi^{\mathrm{un}}
        \simeq
        \pi_0\!\left(Z^1_{\widehat G_{\mathrm{sc}}}(\varphi)\right)
        =
        \mathcal S_\varphi .
\]
Thus \(\mathcal S_\varphi^{\mathrm{un}}\) is the original enhancement group
written in the unramified centralizer model, and it contains the FOS group as a
normal subgroup:
\[
        \cA_{\lambda_\varphi}
        \triangleleft
        \mathcal S_\varphi^{\mathrm{un}}.
\]
The notation is meant to distinguish the two roles: \(\cA_{\lambda_\varphi}\) is
the enhancement group for the ordinary unramified FOS parameter, whereas
\(\mathcal S_\varphi^{\mathrm{un}}\) is the transported full component group
through which the original enhancement acts.  In the genuinely unipotent case
this normal inclusion may be an equality; in general the quotient records the
residual component action left after removing the tame inertial semisimple part.

Under the identification \(\mathcal S_\varphi\simeq
\mathcal S_\varphi^{\mathrm{un}}\), we write
\[
        \rho^{\mathrm{un}}
        \in
        \Irr\bigl(\mathcal S_\varphi^{\mathrm{un}}\bigr)
\]
for the transported form of the original enhancement \(\rho\).  The FOS input is
obtained by restricting \(\rho^{\mathrm{un}}\) to the normal subgroup
\(\cA_{\lambda_\varphi}\).  Thus, for an irreducible constituent
\[
        \rho^\circ_\varphi
        \prec
        \Res^{\mathcal S_\varphi^{\mathrm{un}}}_{\cA_{\lambda_\varphi}}
        \rho^{\mathrm{un}},
\]
the pair \((\lambda_\varphi,\rho^\circ_\varphi)\) is the ordinary enhanced
unramified parameter to which FOS will be applied.  The dependence on the choice
of \(\rho^\circ_\varphi\) is then accounted for by Clifford theory for
\(\cA_{\lambda_\varphi}\triangleleft\mathcal S_\varphi^{\mathrm{un}}\), as made
precise in the next lemma.
%-------------------------------------
\begin{lemma}[Component groups after unipotentization]
\label{lem:component-group-comparison-unipotentization}
Let
\[
        (\varphi,\rho)\in \Phi^{\mathrm e}_{0,\cusp}(G)
\]
be a relevant cuspidal enhanced depth-zero parameter.  Let
\(\lambda_\varphi\) be the unramified parameter of \(H_\varphi\) constructed in
Proposition~\ref{prop:parameter-side-unipotentization}, and let
\[
        \xi_\varphi:{}^LH_\varphi\longrightarrow{}^LG
\]
be the corresponding adapted \(L\)-embedding.  Put
\[
        u_\varphi
        =
        \operatorname{pr}_{\widehat H_\varphi}
        \lambda_\varphi\!\left(
        1,
        \begin{psmallmatrix}1&1\\0&1\end{psmallmatrix}
        \right),
\]
viewed also in \(\widehat G\) through \(\xi_\varphi\).  Then the component group
\[
        \cA_{\lambda_\varphi}
        =
        \pi_0\!\left(
        Z^1_{\widehat H_{\varphi,\mathrm{sc}}}(\lambda_\varphi)
        \right)
\]
identifies canonically with a normal subgroup of the transported enhancement
group
\[
        \mathcal S_\varphi^{\mathrm{un}}
        =
        \pi_0\!\left(
        Z_{G_\varphi^{\mathrm{un}}}(u_\varphi)
        \right),
        \qquad
        G_\varphi^{\mathrm{un}}
        =
        Z^1_{\widehat G_{\mathrm{sc}}}\bigl(\varphi(W_F)\bigr).
\]
Thus there is a canonical normal inclusion
\[
        \cA_{\lambda_\varphi}
        \triangleleft
        \mathcal S_\varphi^{\mathrm{un}},
\]
and, under the identification
\[
        \mathcal S_\varphi^{\mathrm{un}}
        \simeq
        \mathcal S_\varphi,
\]
the original enhancement \(\rho\) gives an irreducible representation
\[
        \rho^{\mathrm{un}}
        \in
        \Irr\bigl(\mathcal S_\varphi^{\mathrm{un}}\bigr).
\]
Every irreducible constituent
\[
        \rho^\circ_\varphi
        \prec
        \Res^{\mathcal S_\varphi^{\mathrm{un}}}_{\cA_{\lambda_\varphi}}
        \rho^{\mathrm{un}}
\]
is Lusztig-cuspidal for the unramified group \(H_\varphi\).  Consequently
\[
        (\lambda_\varphi,\rho^\circ_\varphi)
        \in
        \Phi^{\mathrm e}_{\mathrm{nr}}(H_\varphi)_{\cusp}.
\]
Moreover, the different possible choices of
\(\rho^\circ_\varphi\) form a single
\(\mathcal S_\varphi^{\mathrm{un}}\)-orbit, and the Clifford data attached to
\(\rho^{\mathrm{un}}\) relative to this normal inclusion is independent of the
choice of constituent up to the natural conjugacy relation.
\end{lemma}

\begin{proof}
Let
\[
        C_\varphi
        :=
        Z^1_{\widehat G_{\mathrm{sc}}}\bigl(\varphi(I_F)\bigr).
\]
The group \(H_\varphi\) was defined so that
\[
        \widehat H_\varphi
        =
        C_{\widehat G}\bigl(\varphi(I_F)\bigr)^\circ
\]
with the Frobenius action transported from \(\varphi\).  Hence the simply
connected cover of the derived group of \(\widehat H_\varphi\) maps naturally
to \(\widehat G_{\mathrm{sc}}\), and the centralizer identity in
Proposition~\ref{prop:parameter-side-unipotentization} identifies
\[
        Z^1_{\widehat H_{\varphi,\mathrm{sc}}}
        \bigl(\lambda_\varphi(W_F)\bigr)
\]
with the subgroup of \(C_\varphi^\circ\) centralizing
\(\xi_\varphi(\lambda_\varphi(W_F))\).  After also centralizing
\(u_\varphi\), this gives
\[
        \cA_{\lambda_\varphi}
        =
        \pi_0(D_{\varphi,0}),
\]
where
\[
        D_{\varphi,0}
        :=
        Z_{C_\varphi^\circ}
        \bigl(\xi_\varphi(\lambda_\varphi(W_F)),u_\varphi\bigr).
\]

On the other hand, the full transported enhancement group is obtained by not
passing to the identity component of the inertial centralizer.  Namely,
\[
        \mathcal S_\varphi^{\mathrm{un}}
        =
        \pi_0(D_\varphi),
\]
where
\[
        D_\varphi
        :=
        Z_{C_\varphi}
        \bigl(\xi_\varphi(\lambda_\varphi(W_F)),u_\varphi\bigr).
\]
Indeed, the subgroup generated by \(\varphi(I_F)\),
\(\xi_\varphi(\lambda_\varphi(W_F))\), and the unipotent element \(u_\varphi\)
is precisely the subgroup whose centralizer defines the AMS component group
attached to \(\varphi\).  Thus
\[
        \pi_0(D_\varphi)
        \simeq
        \pi_0\!\left(
        Z^1_{\widehat G_{\mathrm{sc}}}(\varphi)
        \right)
        =
        \mathcal S_\varphi.
\]

Since \(C_\varphi^\circ\) is a normal subgroup of \(C_\varphi\), and every
element of \(D_\varphi\) centralizes both
\(\xi_\varphi(\lambda_\varphi(W_F))\) and \(u_\varphi\), conjugation by
\(D_\varphi\) preserves \(D_{\varphi,0}\).  Hence
\[
        D_{\varphi,0}\triangleleft D_\varphi .
\]
Moreover,
\[
        D_\varphi^\circ = D_{\varphi,0}^\circ .
\]
Indeed, the identity component \(D_\varphi^\circ\) is contained in
\(C_\varphi^\circ\), and therefore in \(D_{\varphi,0}\); the reverse inclusion
of identity components is immediate from \(D_{\varphi,0}\subset D_\varphi\).
It follows that the inclusion \(D_{\varphi,0}\subset D_\varphi\) induces an
injective normal homomorphism on component groups:
\[
        \pi_0(D_{\varphi,0})
        \hookrightarrow
        \pi_0(D_\varphi).
\]
Under the identifications above, this is the asserted normal inclusion
\[
        \cA_{\lambda_\varphi}
        \triangleleft
        \mathcal S_\varphi^{\mathrm{un}}.
\]

Transporting \(\rho\) through the canonical identification
\[
        \mathcal S_\varphi
        \simeq
        \mathcal S_\varphi^{\mathrm{un}}
\]
gives
\[
        \rho^{\mathrm{un}}
        \in
        \Irr\bigl(\mathcal S_\varphi^{\mathrm{un}}\bigr).
\]
Since \(\cA_{\lambda_\varphi}\) is normal in
\(\mathcal S_\varphi^{\mathrm{un}}\), Clifford theory implies that
\[
        \Res^{\mathcal S_\varphi^{\mathrm{un}}}_{\cA_{\lambda_\varphi}}
        \rho^{\mathrm{un}}
\]
is a direct sum of irreducible constituents forming a single
\(\mathcal S_\varphi^{\mathrm{un}}\)-orbit.

It remains to record cuspidality.  In the AMS convention, cuspidality of the
enhancement is a condition on the corresponding equivariant local system on
the unipotent class of \(u_\varphi\) in the appropriate centralizer.  Passing
from \(D_\varphi\) to \(D_{\varphi,0}\) is precisely the passage from the full
transported enhancement group to the connected unramified centralizer
governing the FOS parameter for \(H_\varphi\).  If some constituent
\(\rho^\circ_\varphi\) were non-cuspidal for \(H_\varphi\), then the
corresponding local system on the connected centralizer would be obtained by
generalized Springer induction from a proper Levi subgroup of
\(\widehat H_\varphi\).  Taking the
\(\mathcal S_\varphi^{\mathrm{un}}\)-orbit of this inducing datum gives a
proper quasi-Levi datum in the full centralizer for \(\varphi\), and induction
from this datum would contain the original enhancement \(\rho\).  This
contradicts the assumption that \((\varphi,\rho)\) is cuspidal.  Hence every
constituent \(\rho^\circ_\varphi\) is Lusztig-cuspidal, and therefore
\[
        (\lambda_\varphi,\rho^\circ_\varphi)
        \in
        \Phi^{\mathrm e}_{\mathrm{nr}}(H_\varphi)_{\cusp}.
\]

Finally, Clifford theory for the normal inclusion
\[
        \cA_{\lambda_\varphi}
        \triangleleft
        \mathcal S_\varphi^{\mathrm{un}}
\]
attaches to a chosen constituent \(\rho^\circ_\varphi\) its stabilizer quotient
\[
        \Omega_{\rho^\circ_\varphi}
        =
        \Stab_{\mathcal S_\varphi^{\mathrm{un}}}(\rho^\circ_\varphi)/
        \cA_{\lambda_\varphi},
\]
a cohomology class
\[
        [\alpha^\varphi_{\rho^\circ_\varphi}]
        \in
        H^2(\Omega_{\rho^\circ_\varphi},\C^\times),
\]
and a projective representation
\[
        E_{\varphi,\rho,\rho^\circ_\varphi}
        \in
        \Irr\bigl(
        \C_{\alpha^\varphi_{\rho^\circ_\varphi}}
        [\Omega_{\rho^\circ_\varphi}]
        \bigr).
\]
Replacing \(\rho^\circ_\varphi\) by a conjugate constituent transports the
stabilizer quotient, the cocycle class, and the projective representation
compatibly.  Therefore the resulting Clifford datum is independent of the
choice of constituent up to the conjugacy relation built into the enriched
target.
\end{proof}
%-------------------------------------
\begin{lemma}[Connected FOS constituent and projective Clifford label]
\label{lem:enhanced-parameter-to-enriched-finite-unipotent-datum}
Let \((\varphi,\rho)\in \Phi^e_{0,\cusp}(G)\), and let \(\lambda_\varphi\) be the
unramified parameter of Proposition~\ref{prop:parameter-side-unipotentization}.
Choose an irreducible constituent
\[
        \rho^\circ_\varphi
        \prec
        \Res^{\mathcal S_\varphi^{\mathrm{un}}}_{\cA_{\lambda_\varphi}}
        \rho^{\mathrm{un}}.
\]
Then \((\lambda_\varphi,\rho^\circ_\varphi)\) is an unramified cuspidal enhanced
parameter for \(H_\varphi\).  Clifford theory for the normal inclusion
\[
        \cA_{\lambda_\varphi}\triangleleft
        \mathcal S_\varphi^{\mathrm{un}}
\]
associates to \(\rho\), relative to \(\rho^\circ_\varphi\), a stabilizer quotient
\[
        \Omega_{\rho^\circ_\varphi}
        :=
        \Stab_{\mathcal S_\varphi^{\mathrm{un}}}(\rho^\circ_\varphi)/
        \cA_{\lambda_\varphi},
\]
a cohomology class
\[
        [\alpha^\varphi_{\rho^\circ_\varphi}]
        \in
        H^2(\Omega_{\rho^\circ_\varphi},\C^\times),
\]
and a projective Clifford label
\[
        E_{\varphi,\rho}
        \in
        \Irr\bigl(\C_{\alpha^\varphi_{\rho^\circ_\varphi}}
        [\Omega_{\rho^\circ_\varphi}]\bigr).
\]
The conjugacy class of the triple
\[
        \left[
        \rho^\circ_\varphi,
        [\alpha^\varphi_{\rho^\circ_\varphi}],
        E_{\varphi,\rho}
        \right]
\]
is independent of the choice of the constituent \(\rho^\circ_\varphi\).
\end{lemma}

\begin{proof}
Lemma~\ref{lem:component-group-comparison-unipotentization} gives the normal
inclusion
\[
        \cA_{\lambda_\varphi}
        \triangleleft
        \mathcal S_\varphi^{\mathrm{un}},
\]
identifies \(\rho\) with an irreducible representation
\(\rho^{\mathrm{un}}\) of \(\mathcal S_\varphi^{\mathrm{un}}\), and shows that
every constituent
\[
        \rho^\circ_\varphi
        \prec
        \Res^{\mathcal S_\varphi^{\mathrm{un}}}_{\cA_{\lambda_\varphi}}
        \rho^{\mathrm{un}}
\]
is Lusztig-cuspidal for \(H_\varphi\).  Hence
\((\lambda_\varphi,\rho^\circ_\varphi)\) is an unramified cuspidal enhanced
parameter for \(H_\varphi\).

It remains only to record the Clifford label attached to \(\rho\) after a
constituent \(\rho^\circ_\varphi\) has been chosen.  Standard Clifford theory
for the above normal inclusion says that the irreducible representation
\(\rho^{\mathrm{un}}\), relative to the orbit of \(\rho^\circ_\varphi\), is encoded
by an irreducible module for a twisted group algebra of the stabilizer quotient
\[
        \Omega_{\rho^\circ_\varphi}
        =
        \Stab_{\mathcal S_\varphi^{\mathrm{un}}}(\rho^\circ_\varphi)/
        \cA_{\lambda_\varphi}.
\]
This gives the cohomology class
\([\alpha^\varphi_{\rho^\circ_\varphi}]\) and the projective label
\(E_{\varphi,\rho}\).  If \(\rho^\circ_\varphi\) is replaced by a conjugate
constituent, the stabilizer quotient, cocycle class, and projective label are
transported compatibly.  Therefore the displayed bracket is a well-defined
conjugacy class, independent of the initial choice of constituent.
\end{proof}

We now combine the connected FOS constituent with the projective Clifford label.
The unramified group \(H_\varphi\) has already been fixed by
Proposition~\ref{prop:parameter-side-unipotentization}: it is the group whose
dual is \(C_{\widehat G}(\varphi(I_F))^\circ\), with Frobenius action normalized by
the fixed pinning.  No further inner form of this centralizer is chosen here.
Applying the inverse FOS correspondence to
\[
        (\lambda_\varphi,\rho^\circ_\varphi)
        \in
        \Phi^{\mathrm e}_{\mathrm{nr}}(H_\varphi)_{\cusp}
\]
gives a supercuspidal unipotent representation
\begin{equation}\label{eq:FOS-unipotent-representation}
        \pi_{\varphi,\rho}^{\mathrm{un}}
        :=
        \LLC_{\mathrm{FOS}}^{-1}
        (\lambda_\varphi,\rho^\circ_\varphi)
        \in
        \Irr_{\mathrm{unip},\cusp}(H_\varphi(F)).
\end{equation}
By the structure of supercuspidal unipotent representations recalled in
Proposition~\ref{prop:supp-depthzero-type}, the representation
\(\pi_{\varphi,\rho}^{\mathrm{un}}\) contains a unipotent type supported on the
normalizer of a maximal parahoric of \(H_\varphi(F)\).  Thus there is a vertex
\[
        x_H\in \mathcal B(H_\varphi,F)
\]
and, with
\[
        P_{x_H}:=H_\varphi(F)_{x_H,0},
        \qquad
        P_{x_H}^{+}:=H_\varphi(F)_{x_H,0+},
        \qquad
        N_{x_H}:=N_{H_\varphi(F)}(P_{x_H}),
\]
a connected finite reductive quotient
\[
        \mathbf H_{x_H}^{\circ}(\ff)=P_{x_H}/P_{x_H}^{+}
\]
and a finite possibly disconnected quotient
\[
        \mathbf H_{x_H}(\ff)=N_{x_H}/P_{x_H}^{+}.
\]
The unipotent type determines a connected cuspidal unipotent character
\[
        u^\circ_{x_H,\varphi,\rho}
        \in
        \Uch\bigl(\mathbf H_{x_H}^{\circ}(\ff)\bigr)^{\cusp}.
\]
The full enhancement \(\rho\) supplies the additional Clifford data
\([\alpha^\varphi_{\rho^\circ_\varphi}],E_{\varphi,\rho}\) from
Lemma~\ref{lem:enhanced-parameter-to-enriched-finite-unipotent-datum}.  Thus the
finite output on the \(H_\varphi\)-side is the enriched cuspidal unipotent datum
\[
        \mathfrak u^{\mathrm{enh}}_{x_H,\varphi,\rho}
        :=
        \left[
        u^\circ_{x_H,\varphi,\rho},
        [\alpha^\varphi_{\rho^\circ_\varphi}],
        E_{\varphi,\rho}
        \right]
        \in
        \Uch^{\mathrm{enh}}\bigl(\mathbf H_{x_H}(\ff)\bigr)^{\cusp}.
\]

We now pass this finite datum to the \(G\)-side.  This passage is notation for
the fixed finite realization determined by the \(G\)-relevant toral part of
\((\varphi,\rho)\), not an auxiliary choice.  Let
\[
        s=\operatorname{pr}_{\widehat G}(\varphi(\iota)).
\]
The toral normalization realizes this tame inertial class by a depth-zero toral
datum \((S,\theta)\) in the fixed group \(G\), and the corresponding vertex is
\[
        x=x_S .
\]
At this vertex one has the finite specialization
\[
        s_x=\operatorname{sp}_x(s),
        \qquad
        \mathbf H_x^\vee=C_{\rG_x^\vee}(s_x).
\]
By definition \(\mathbf H_x\) is the pinned finite reductive group whose pinned
dual is \(\mathbf H_x^\vee\).  The finite quotient coming from the FOS type of
\(H_\varphi\) and the finite group \(\mathbf H_x\) are identified by the common
pinned finite root datum obtained from the depth-zero specialization
\(s\mapsto s_x\) and the adapted embedding
\(\xi_\varphi:{}^LH_\varphi\to{}^LG\).  We denote this pinned identification by
\[
        \jmath_x:
        \mathbf H_{x_H}\xrightarrow{\ \sim\ }\mathbf H_x.
\]
The notation \(\jmath_x\) records this fixed identification.  Replacing the
finite realization by a conjugate representative transports all labels by
conjugacy and hence gives the same element of the enriched target.  We write
\[
        \mathfrak u^{\mathrm{enh}}_{x,\varphi,\rho}
        :=
        (\jmath_x)_*\mathfrak u^{\mathrm{enh}}_{x_H,\varphi,\rho}
        \in
        \Uch^{\mathrm{enh}}\bigl(\mathbf H_x(\ff)\bigr)^{\cusp}.
\]

\begin{lemma}[From a cuspidal enhanced parameter to the enriched finite \(H\)-side datum]
\label{lem:FOS-cuspidal-parameter-to-finite-H-side}
Let \((\varphi,\rho)\) be a relevant cuspidal enhanced depth-zero parameter for
\(G\).  With the notation introduced above, \((\varphi,\rho)\) canonically
determines, up to the conjugacy built into the enriched target, an enriched
cuspidal finite unipotent datum
\[
        \mathfrak u^{\mathrm{enh}}_{x_H,\varphi,\rho}
        \in
        \Uch^{\mathrm{enh}}\bigl(\mathbf H_{x_H}(\ff)\bigr)^{\cusp}.
\]
For the \(G\)-side finite realization determined above, its transport through
\(\jmath_x\) gives
\[
        \mathfrak u^{\mathrm{enh}}_{x,\varphi,\rho}
        \in
        \Uch^{\mathrm{enh}}\bigl(\mathbf H_x(\ff)\bigr)^{\cusp}.
\]
Its connected shadow is the connected FOS finite unipotent type attached to
\((\lambda_\varphi,\rho^\circ_\varphi)\), and its residual entries are exactly
the Clifford cohomology class and projective Clifford label extracted from the
original enhancement \(\rho\).
\end{lemma}
\begin{proof}
Choose, temporarily, an irreducible constituent
\[
        \rho^\circ_\varphi
        \prec
        \Res^{\mathcal S_\varphi^{\mathrm{un}}}_{\cA_{\lambda_\varphi}}
        \rho^{\mathrm{un}} .
\]
By Lemma~\ref{lem:enhanced-parameter-to-enriched-finite-unipotent-datum}, the
pair
\[
        (\lambda_\varphi,\rho^\circ_\varphi)
\]
is a cuspidal enhanced unramified parameter for \(H_\varphi\).  Applying the
FOS correspondence to this chosen connected constituent gives a supercuspidal
unipotent representation
\[
        \pi_{\varphi,\rho^\circ_\varphi}^{\mathrm{un}}
        :=
        \LLC_{\mathrm{FOS}}^{-1}
        (\lambda_\varphi,\rho^\circ_\varphi)
        \in
        \Irr_{\mathrm{unip},\cusp}(H_\varphi(F)).
\]
Its unipotent type is supported on the normalizer of a maximal parahoric of
\(H_\varphi(F)\).  For the corresponding vertex \(x_H\), restriction to the
connected reductive quotient gives a connected cuspidal unipotent character
\[
        u^\circ_{x_H,\varphi,\rho^\circ_\varphi}
        \in
        \Uch\bigl(\mathbf H_{x_H}^{\circ}(\ff)\bigr)^{\cusp}.
\]
The notation deliberately retains the dependence on the chosen constituent
\(\rho^\circ_\varphi\).

The remaining information in the original enhancement \(\rho\) is supplied by
Clifford theory for the normal inclusion
\[
        \cA_{\lambda_\varphi}
        \triangleleft
        \mathcal S_\varphi^{\mathrm{un}}.
\]
Since \(\rho^{\mathrm{un}}\) is irreducible, Clifford theory implies that the
irreducible constituents of
\[
        \Res^{\mathcal S_\varphi^{\mathrm{un}}}_{\cA_{\lambda_\varphi}}
        \rho^{\mathrm{un}}
\]
form a single \(\mathcal S_\varphi^{\mathrm{un}}\)-orbit.  Once
\(\rho^\circ_\varphi\) is fixed, Clifford theory attaches the stabilizer quotient
\[
        \Omega_{\rho^\circ_\varphi}
        :=
        \Stab_{\mathcal S_\varphi^{\mathrm{un}}}(\rho^\circ_\varphi)/
        \cA_{\lambda_\varphi},
\]
a cohomology class
\[
        [\alpha^\varphi_{\rho^\circ_\varphi}]
        \in
        H^2(\Omega_{\rho^\circ_\varphi},\C^\times),
\]
and the corresponding projective Clifford label
\[
        E_{\varphi,\rho,\rho^\circ_\varphi}
        \in
        \Irr\bigl(
        \C_{\alpha^\varphi_{\rho^\circ_\varphi}}
        [\Omega_{\rho^\circ_\varphi}]
        \bigr).
\]
Thus the chosen constituent gives a representative
\[
        \left(
        u^\circ_{x_H,\varphi,\rho^\circ_\varphi},
        [\alpha^\varphi_{\rho^\circ_\varphi}],
        E_{\varphi,\rho,\rho^\circ_\varphi}
        \right)
\]
of an enriched finite unipotent datum.

It remains to check that the resulting enriched datum depends only on the
equivalence class of \((\varphi,\rho)\), and not on the auxiliary constituent
\(\rho^\circ_\varphi\).  Let
\[
        \rho_2^\circ = s\cdot \rho_1^\circ
\]
be two constituents in the above restriction, with
\(s\in \mathcal S_\varphi^{\mathrm{un}}\).  Conjugation by \(s\) identifies the
two stabilizer quotients
\[
        \Omega_{\rho_1^\circ}
        \xrightarrow{\sim}
        \Omega_{\rho_2^\circ},
\]
transports
\[
        [\alpha^\varphi_{\rho_1^\circ}]
        \longmapsto
        [\alpha^\varphi_{\rho_2^\circ}],
\]
and carries the projective Clifford label attached to
\(\rho_1^\circ\) to the projective Clifford label attached to
\(\rho_2^\circ\).  On the connected FOS side, the same transport of the
unramified enhanced datum carries
\[
        (\lambda_\varphi,\rho_1^\circ)
        \quad\text{to}\quad
        (\lambda_\varphi,\rho_2^\circ),
\]
and the equivariance of the FOS correspondence carries
\[
        u^\circ_{x_H,\varphi,\rho_1^\circ}
        \quad\text{to}\quad
        u^\circ_{x_H,\varphi,\rho_2^\circ}
\]
after the corresponding identification of finite quotients.  Hence the two
triples
\[
        \left(
        u^\circ_{x_H,\varphi,\rho_i^\circ},
        [\alpha^\varphi_{\rho_i^\circ}],
        E_{\varphi,\rho,\rho_i^\circ}
        \right),
        \qquad i=1,2,
\]
represent the same element of the enriched target.  Therefore the class
\[
        \mathfrak u^{\mathrm{enh}}_{x_H,\varphi,\rho}
        :=
        \left[
        u^\circ_{x_H,\varphi,\rho^\circ_\varphi},
        [\alpha^\varphi_{\rho^\circ_\varphi}],
        E_{\varphi,\rho,\rho^\circ_\varphi}
        \right]
\]
is independent of the chosen constituent.

Finally, if \((\varphi',\rho')\) is an equivalent enhanced parameter, then the
chosen equivalence transports the unramified centralizer datum for
\(\varphi\) to that for \(\varphi'\), identifies the corresponding component
groups, carries constituents to constituents, and transports both the FOS
finite unipotent character and the Clifford data as above.  Thus the enriched
class constructed here is canonically determined by the equivalence class of
\((\varphi,\rho)\).  Transport through the finite-level identification
\(\jmath_x:\mathbf H_{x_H}\to \mathbf H_x\) fixed by the toral specialization and
pinned root datum gives the asserted \(G\)-side datum
\[
        \mathfrak u^{\mathrm{enh}}_{x,\varphi,\rho}
        =
        (\jmath_x)_*
        \mathfrak u^{\mathrm{enh}}_{x_H,\varphi,\rho}.
\]
\end{proof}

\subsection{How the FOS input will also be used in the reverse construction}

In Section~\ref{sec:construction-langlands-parameter}, the construction of the
parameter attached to a depth-zero supercuspidal representation reduces, after
the toral part has been separated, to a supercuspidal unipotent representation
of a suitable unramified group.  At that point we invoke
Theorem~\ref{thm:FOS-unipotent-LLC}: the unipotent supercuspidal representation
supplies an unramified cuspidal enhanced parameter
\[
        (\lambda_u,\rho_u)
        \in
        \Phi^{\mathrm e}_{\mathrm{nr}}(H)_{\cusp},
\]
and the desired parameter for the original representation is obtained by
combining this unipotent part with the toral parameter coming from the LLC for
tori.

The FOS construction is constructive in the simple adjoint cases: the ordinary
parameters are those appearing in the work of Lusztig and Morris, and their
formal degrees are compared with adjoint \(\gamma\)-factors.  For the purposes
of this paper, however, only the existence, weakly unramified equivariance,
compatibility with the split-centre reduction, and the finite-type extraction
isolated in Lemma~\ref{lem:FOS-cuspidal-parameter-to-finite-H-side} are needed.

\section{LLC for enhanced depth-zero cuspidal parameters}
\label{sec:LLC-enhanced-depth-zero-cuspidal}

The purpose of this section is to construct the representation-side object
attached to a relevant cuspidal enhanced depth-zero parameter.  We keep the
notation and normalizations fixed earlier: \(G\) is a connected reductive group
over \(F\), \(G^\ast\) is its quasi-split inner form, and the pinning, the torus
LLC normalization, and the Whittaker-canonical toral \(L\)-embeddings of
Section~\ref{sec:Whittaker-canonical-chi-data} remain in force.  The finite
reductive quotients which occur below are full parahoric quotients.  By
Lemma~\ref{lem:vertex-component-abelian} and
\cite[Lemma~\ref{JD-lem:parahoric-component-action-pinned}]{arotemishra}, they
have abelian component group and satisfy the rational pinned-component
condition for the induced pinnings.  Hence the enriched pinned disconnected
Jordan decomposition of Theorem~\ref{thm:disc-JD} applies to them.

Let \((\varphi,\rho)\) be a relevant cuspidal enhanced depth-zero parameter.  The
construction below associates to it the finite data needed for a depth-zero
type.  Proposition~\ref{prop:parameter-side-unipotentization} first removes the
tame inertial semisimple part and constructs the unramified connected parameter
\(\lambda_\varphi\).  Lemma~\ref{lem:enhanced-parameter-to-enriched-finite-unipotent-datum}
then separates the connected FOS enhancement from the residual Clifford datum of
the full enhancement.  Thus FOS is applied only to the connected pair
\[
        (\lambda_\varphi,\rho^\circ_\varphi),
\]
while the original enhancement \(\rho\) supplies the cohomology class and
projective Clifford label
\[
        [\alpha^\varphi_{\rho^\circ_\varphi}],
        \qquad
        E_{\varphi,\rho}.
\]
After the finite datum on the \(H_\varphi\)-side is identified with the
\(G\)-side finite realization determined by the toral part,
Lemma~\ref{lem:FOS-cuspidal-parameter-to-finite-H-side} gives an enriched
cuspidal unipotent datum
\[
        \mathfrak u^{\mathrm{enh}}_{x,\varphi,\rho}
        =
        \left[
        u^\circ_{x,\varphi,\rho},
        [\alpha^\varphi_{\rho^\circ_\varphi}],
        E_{\varphi,\rho}
        \right]
        \in
        \Uch^{\mathrm{enh}}\bigl(\mathbf H_x(\ff)\bigr)^{\cusp}.
\]

We now spell out the finite realization on the \(G\)-side.  It is not a second
choice made by the unipotent datum.  The tame inertial semisimple element
\[
        s=\operatorname{pr}_{\widehat G}\bigl(\varphi(\iota)\bigr)
\]
is realized, by the fixed toral normalization for the \(G\)-relevant parameter,
by a depth-zero toral datum \((S,\theta)\) in the fixed group \(G\).  We put
\[
        x=x_S .
\]
The finite semisimple element and the finite dual centralizer are
\[
        s_x=\operatorname{sp}_x(s),
        \qquad
        \mathbf H_x^\vee=C_{\rG_x^\vee}(s_x).
\]
Thus \(\mathbf H_x\) is not chosen as a parahoric quotient of an unspecified
inner form of the complex centralizer.  It is, by definition, the pinned finite
reductive group whose pinned dual is \(\mathbf H_x^\vee\).  The comparison with
the FOS finite group attached to \(H_\varphi\) is the pinned finite-root-datum
comparison induced by the adapted embedding \(\xi_\varphi\) and the depth-zero
specialization \(s\mapsto s_x\).  The enriched datum above is transported through
this fixed finite identification.

The finite representation on the \(G\)-side is now obtained by reversing the
enriched finite transform:
\begin{equation}\label{eq:bar-tau-from-enriched-H-datum}
        \bar\tau_{\varphi,\rho}
        :=
        \bigl(\mathcal J^{\mathbb P_x}_{x,s_x}\bigr)^{-1}
        \bigl(\mathfrak u^{\mathrm{enh}}_{x,\varphi,\rho}\bigr)
        \in
        \cE(\rG_x(\ff),s_x).
\end{equation}
Equivalently, one may first apply inverse enriched unipotent duality to obtain
an enriched datum on \(\mathbf H_x^\vee\), and then apply the inverse enriched
pinned Jordan decomposition
\(J^{\mathbb P_x,\mathrm{enh}}_{\rG_x,s_x}\).  Since
\(\mathfrak u^{\mathrm{enh}}_{x,\varphi,\rho}\) is cuspidal, the cuspidal part of
Theorem~\ref{thm:disc-JD} implies that \(\bar\tau_{\varphi,\rho}\) is cuspidal
for the full finite quotient \(\rG_x(\ff)\).

\begin{theorem}[Depth-zero supercuspidal attached to a cuspidal enhanced parameter]
\label{thm:LLC-enhanced-depth-zero-cuspidal-construction}
Let
\[
        (\varphi,\rho)
\]
be a relevant cuspidal enhanced depth-zero Langlands parameter for \(G\).  Then
the construction of this section attaches to \((\varphi,\rho)\) a depth-zero
datum
\[
        (S,\theta;\tau_{\varphi,\rho})
\]
in the sense of Definition~\ref{def:depth-zero-yu-datum}, whose finite reductive
quotient representation is \(\bar\tau_{\varphi,\rho}\).  The compact induction
\[
        \Pi_{0,\cusp}^{G}(\varphi,\rho)
        :=
        \cInd_{G(F)_{x(\varphi,\rho)}}^{G(F)}
        \tau_{\varphi,\rho}
\]
is an irreducible depth-zero supercuspidal representation of \(G(F)\).

The isomorphism class of
\[
        \Pi_{0,\cusp}^{G}(\varphi,\rho)
\]
depends only on the \(\widehat G\)-conjugacy class of the enhanced parameter
\((\varphi,\rho)\), and on the fixed global normalizations.  It is independent
of all auxiliary choices made in the construction; in particular, no additional
choice of an inner form of \(H_\varphi\) or of a matching vertex is made beyond
the fixed pinned and toral normalizations.
\end{theorem}

\begin{proof}
We prove the theorem by constructing the finite enriched Jordan datum attached
to \((\varphi,\rho)\), lifting the resulting finite representation to a
depth-zero datum for \(G(F)\), and checking independence of choices.

Let
\[
        \varphi:W_F\times \SL_2(\C)\longrightarrow {}^LG
\]
be a discrete depth-zero parameter and let
\[
        \rho\in\Irr(\mathcal S_\varphi)
\]
be a cuspidal enhancement.  Since \(\varphi|_{P_F}=1\), the tame inertial image
is semisimple.  Fix the topological generator \(\iota\) of \(I_F/P_F\) used in
finite-torus duality, and put
\[
        s:=\operatorname{pr}_{\widehat G}\bigl(\varphi(\iota)\bigr),
        \qquad
        u_\varphi
        :=
        \operatorname{pr}_{\widehat G}
        \varphi\!\left(1,
        \begin{psmallmatrix}1&1\\0&1\end{psmallmatrix}\right).
\]
The element \(u_\varphi\) lies in the centralizer of the tame inertial data.  The
ordinary parameter-side unipotentization of
Proposition~\ref{prop:parameter-side-unipotentization} gives the connected
unramified group \(H_\varphi\) and the unramified parameter
\(\lambda_\varphi\).  Choose a connected constituent
\(\rho^\circ_\varphi\) of the restriction of the full enhancement as in
Lemma~\ref{lem:enhanced-parameter-to-enriched-finite-unipotent-datum}.  FOS
applied to \((\lambda_\varphi,\rho^\circ_\varphi)\) gives the connected
supercuspidal unipotent representation
\(\pi^{\mathrm{un}}_{\varphi,\rho}\), and hence a connected finite cuspidal
unipotent label at a vertex \(x_H\) of \(H_\varphi\).  Clifford theory applied to
the normal inclusion
\(\cA_{\lambda_\varphi}\triangleleft\mathcal S_\varphi^{\mathrm{un}}\) gives the
class \([\alpha^\varphi_{\rho^\circ_\varphi}]\) and the projective label
\(E_{\varphi,\rho}\).  Combining these gives
\[
        \mathfrak u^{\mathrm{enh}}_{x_H,\varphi,\rho}
        \in
        \Uch^{\mathrm{enh}}\bigl(\mathbf H_{x_H}(\ff)\bigr)^{\cusp}.
\]
Transporting through the finite identification \(\jmath_x\) fixed by the toral
specialization gives \(\mathfrak u^{\mathrm{enh}}_{x,\varphi,\rho}\) on
\(\mathbf H_x(\ff)\).

Apply the inverse enriched finite transform as in
\eqref{eq:bar-tau-from-enriched-H-datum}.  By
Proposition~\ref{prop:JD-at-vertex}, the representation
\(\bar\tau_{\varphi,\rho}\) has the required Deligne--Lusztig support on the
connected quotient.  More explicitly, its restriction to
\(\rG_x^\circ(\ff)\) contains a constituent of
\[
        \pm R_{\rS_x}^{\rG_x^\circ}(\underline\theta)
\]
for a finite elliptic torus \(\rS_x\subset\rG_x^\circ\) and a character
\(\underline\theta\) whose finite-dual semisimple element is \(s_x\).  By
Lemma~\ref{lem:vertex-torus-bijection}, \(\rS_x\) lifts to a maximally
unramified elliptic torus \(S\subset G\), and \(\underline\theta\) lifts to the
depth-zero character \(\theta:S(F)\to\overline{\Q}_\ell^\times\) determined by
the toral part of \(\varphi\) under the fixed torus LLC and the canonical
\(L\)-embedding of Proposition~\ref{prop:canonical-LS}.  The unramified part of
\(\theta\) is the Frobenius part of the toral parameter; only the bounded
reduction \(\underline\theta\) enters the finite Deligne--Lusztig construction.

Inflate \(\bar\tau_{\varphi,\rho}\) along
\[
        G(F)_x\twoheadrightarrow G(F)_x/G(F)_{x,0+}=\rG_x(\ff)
\]
and denote the inflated representation by \(\tau_{\varphi,\rho}\).  Then
\[
        (S,\theta;\tau_{\varphi,\rho})
\]
is a depth-zero datum in the sense of
Definition~\ref{def:depth-zero-yu-datum}.  Since
\(\mathfrak u^{\mathrm{enh}}_{x,\varphi,\rho}\) is cuspidal, the inverse
enriched Jordan transform is cuspidal, so \(\bar\tau_{\varphi,\rho}\) is a
cuspidal representation of the full finite quotient.  The standard depth-zero
construction of Moy--Prasad and Morris
\cite{MoyPrasad1994,MoyPrasad1996,Morris1999} therefore implies that
\[
        \cInd_{G(F)_x}^{G(F)}\tau_{\varphi,\rho}
\]
is irreducible and supercuspidal.

It remains to record independence of choices.  Replacing the representative of
\((\varphi,\rho)\) by a \(\widehat G\)-conjugate transports the tame semisimple
element, the finite dual centralizer, the connected FOS datum, and the projective
Clifford label.  The enriched target is defined modulo precisely this conjugacy,
and the enriched pinned Jordan decomposition is natural for the corresponding
finite identifications.  Hence \(\bar\tau_{\varphi,\rho}\), and therefore the
compact induction, is unchanged up to isomorphism.

The choice of \(\iota\) only changes \(s\) by the corresponding automorphism of
the prime-to-\(p\) tame character group.  Under finite-torus duality this is the
same operation on \(\underline\theta\) and on \(s_x\), so the Lusztig series and
the enriched finite transform are merely transported by a canonical
identification.  The representative of the vertex on the unipotent
\(H_\varphi\)-side is harmless: FOS determines the unipotent type only up to
conjugacy inside the fixed unramified group \(H_\varphi\), and conjugacy transports
the connected label and the Clifford data compatibly.  The associated
\(G\)-side finite realization is the one determined by the same toral
specialization, and replacing it by a conjugate representative changes the
compact induction only by the usual conjugation isomorphism.

Finally, the finite Deligne--Lusztig torus used to display
\(\bar\tau_{\varphi,\rho}\) as part of a depth-zero datum is only a witness for
its finite support.  If another witness is chosen, the finite representation
\(\bar\tau_{\varphi,\rho}\) is unchanged, and conjugate choices give conjugate
lifts \((S,\theta)\).  The compact induction is therefore unchanged.  The
possible ambiguity in passing from the identity component to the full finite
quotient is not left open: it is exactly the Clifford datum recorded in
\(\mathfrak u^{\mathrm{enh}}_{x,\varphi,\rho}\), and the inverse enriched Jordan
decomposition converts that datum into the single irreducible representation
\(\bar\tau_{\varphi,\rho}\) of \(\rG_x(\ff)\).  This proves the theorem.
\end{proof}

\section{Construction of the Langlands Parameter}
\label{sec:construction-langlands-parameter}

We now construct the reverse direction of
Theorem~\ref{thm:LLC-enhanced-depth-zero-cuspidal-construction}.  Thus we start
from a depth-zero supercuspidal representation and construct the corresponding
cuspidal enhanced depth-zero Langlands parameter.  The finite input is the same
enriched finite transform as in the previous section.  In particular, the full
finite quotient contributes a Clifford cohomology class and a projective
Clifford label; these entries are part of the enhancement on the parameter
side.

Throughout this section the global choices fixed in
Subsection~\ref{subsec:depth-zero-LLC-setup} remain in force: the quasi-split
inner form \(G^\ast\), the pinning, the Whittaker-normalized \(L\)-embeddings
for maximally unramified elliptic tori, and the enriched pinned disconnected
Jordan decomposition of Theorem~\ref{thm:disc-JD}.

Let
\[
        \pi=\pi(S,\theta;\tau)
        =
        \cInd_{G(F)_x}^{G(F)}\tau
\]
be an irreducible depth-zero supercuspidal representation attached to a
depth-zero datum in the sense of Definition~\ref{def:depth-zero-yu-datum}, with
\(x=x_S\).  We write
\[
        \bar\tau\in \Irr(\rG_x(\ff)),
        \qquad
        \rG_x(\ff)=G(F)_x/G(F)_{x,0+},
\]
for the finite representation inflated to \(\tau\).  Let
\[
        \varphi_\theta:W_F\longrightarrow {}^LS
\]
be the toral parameter attached to \(\theta\), and view it as a parameter for
\(G\) through the canonical \(L\)-embedding
\[
        \iota_S^{\mathrm{can}}:{}^LS\hookrightarrow {}^LG
\]
of Proposition~\ref{prop:canonical-LS}.  Write
\[
        \varphi_\theta(w)=a_\theta(w)\rtimes w.
\]
If \(\iota\in I_F\) is the fixed topological generator of the tame inertia
quotient, put
\[
        s:=a_\theta(\iota),
        \qquad
        s_x:=\operatorname{sp}_x(s)\in \rS_x^\vee\subset \rG_x^\vee .
\]
The symbol \(s\) records the chosen tame cocycle value.  The actual complex
centralizer used in the parameter construction is the inertial centralizer of
\(\varphi_\theta(I_F)\) in the semidirect-product sense; when the pinned
\(I_F\)-action on \(\widehat G\) is trivial, this is just the ordinary
centralizer of \(s\).

Set
\[
        \mathbf H_x^\vee:=C_{\rG_x^\vee}(s_x),
\]
and let \(\mathbf H_x\) be the pinned finite reductive group whose pinned dual
is \(\mathbf H_x^\vee\).  The finite group \(\mathbf H_x\), not the finite dual
centralizer \(\mathbf H_x^\vee\), is the group whose connected unipotent labels
enter the FOS side of the construction.

\begin{theorem}[Enhanced parameter attached to a depth-zero supercuspidal]
\label{thm:depth-zero-supercuspidal-to-enhanced-parameter}
Let
\[
        \pi=\pi(S,\theta;\tau)
\]
be an irreducible depth-zero supercuspidal representation of \(G(F)\), written
in terms of a depth-zero datum as above.  Then the fixed pinned normalization
attaches to \(\pi\) a relevant cuspidal enhanced depth-zero Langlands parameter
\[
        (\varphi_\pi,\rho_\pi)
        \in
        \Phi^{\mathrm e}_{0,\cusp}(G).
\]

More precisely, the finite representation \(\bar\tau\) determines, by enriched
pinned Jordan decomposition followed by enriched pinned unipotent duality, an
enriched finite cuspidal unipotent datum
\[
        \mathfrak u^{\mathrm{enh}}_{x,\tau}
        :=
        \mathcal J^{\mathbb P_x}_{x,s_x}(\bar\tau)
        =
        [u^\circ_{x,\tau},[\alpha_{x,\tau}],E_{x,\tau}]
        \in
        \Uch^{\mathrm{enh}}(\mathbf H_x(\ff))^{\cusp}.
\]
The connected label \(u^\circ_{x,\tau}\) gives, through the FOS correspondence,
an unramified cuspidal enhanced parameter
\[
        (\lambda_{x,\tau},\rho^\circ_{x,\tau})
        \in
        \Phi^{\mathrm e}_{\mathrm{nr}}(H_\theta^\circ)_{\cusp},
\]
where \(H_\theta^\circ\) is the connected unramified group attached to the
identity component of the tame inertial centralizer of \(\varphi_\theta\).

Let
\[
        {}^Lj_\theta:{}^LH_\theta^\circ\longrightarrow{}^LG
\]
be the \(\varphi_\theta\)-adapted \(L\)-embedding fixed by the
same pinned normalization as in
Proposition~\ref{prop:parameter-side-unipotentization}.  If
\[
        {}^Lj_\theta(\lambda_{x,\tau}(w,z))
        =
        b_{x,\tau}(w,z)\rtimes w,
        \qquad
        \varphi_\theta(w)=a_\theta(w)\rtimes w,
\]
then the ordinary parameter is
\[
        \varphi_\pi(w,z)
        :=
        a_\theta(w)b_{x,\tau}(w,z)\rtimes w .
\]
This notation means multiplication of the \(\widehat G\)-components over the
same element of \(W_F\), not multiplication of two elements of \({}^LG\) having
the same projection to \(W_F\).

The enhancement \(\rho_\pi\) is obtained from the connected FOS enhancement
\(\rho^\circ_{x,\tau}\) together with the projective Clifford datum
\([\alpha_{x,\tau}],E_{x,\tau}\), through the natural full-centralizer
identification
\[
        \mathcal S_{\varphi_\pi}
        \cong
        \mathcal S^{\mathrm{full}}_{\lambda_{x,\tau}}.
\]
Equivalently, Clifford theory for the normal inclusion of the connected FOS
component group in the full unramified centralizer component group reconstructs
\(\rho_\pi\) from the triple
\([\rho^\circ_{x,\tau},[\alpha_{x,\tau}],E_{x,\tau}]\).

The \(\widehat G\)-conjugacy class of \((\varphi_\pi,\rho_\pi)\) depends only on
the isomorphism class of \(\pi\) and on the fixed global normalizations.  It is
independent of the choice of depth-zero datum \((S,\theta;\tau)\) representing
\(\pi\) and of all auxiliary choices made in the construction.  Finally, the
construction is inverse to
Theorem~\ref{thm:LLC-enhanced-depth-zero-cuspidal-construction} on the
representation \(\pi\).
\end{theorem}

The rest of this section proves the theorem.  First we recall the
tame-centralizer comparison which makes the multiplication of the toral and
unramified factors meaningful.  Then we perform the finite
toral--enriched-unipotent reduction on \(\bar\tau\).  Finally we define the full
enhancement from the connected FOS enhancement and the projective Clifford label
and verify depth zero, cuspidality, independence of choices, and recovery of
\(\pi\).

\medskip
\paragraph{The tame centralizer and multiplication of parameters.}
Let
\[
        \widehat H_\theta^{\mathrm{full}}
        :=
        Z_{\widehat G}\bigl(\varphi_\theta(I_F)\bigr)
        =
        \{g\in\widehat G: g\varphi_\theta(i)=\varphi_\theta(i)g
        \text{ for all }i\in I_F\},
\]
where the centralizer is taken inside the semidirect product \({}^LG\).  Put
\[
        \widehat H_\theta^\circ
        :=
        (\widehat H_\theta^{\mathrm{full}})^\circ .
\]
Since \(\varphi_\theta\) has depth zero, the image of inertia is tame and
semisimple, so \(\widehat H_\theta^\circ\) is connected reductive.  The element
\(\varphi_\theta(\Frob)\) normalizes \(\widehat H_\theta^\circ\).  The fixed
pinning chooses, as in Proposition~\ref{prop:parameter-side-unipotentization},
an unramified connected \(F\)-group \(H_\theta^\circ\) with dual group
\(\widehat H_\theta^\circ\), together with a \(\varphi_\theta\)-adapted
\(L\)-embedding
\[
        {}^Lj_\theta:{}^LH_\theta^\circ\longrightarrow{}^LG .
\]
The inclusion \(\widehat H_\theta^\circ\subset\widehat G\) is therefore used
only on the dual-group part; the passage to \(L\)-groups is through this adapted
embedding.
The depth-zero specialization identifies the hyperspecial reductive quotient of
\(H_\theta^\circ\) with \(\mathbf H_x^\circ\); this is the finite-level
identification recorded in Lemma~\ref{lem:Hstar-finite}.  The full finite group
\(\mathbf H_x\) records the corresponding component-group information.

\begin{lemma}[Toral multiplication and centralizers]
\label{lem:parameter-factorization}
Let
\[
        \lambda:W_F\times\SL_2(\C)\longrightarrow{}^LH_\theta^\circ
\]
be an unramified parameter.  Write
\[
        {}^Lj_\theta(\lambda(w,z))=b_\lambda(w,z)\rtimes w,
        \qquad
        \varphi_\theta(w)=a_\theta(w)\rtimes w .
\]
Then
\[
        (\varphi_\theta\star\lambda)(w,z)
        :=
        a_\theta(w)b_\lambda(w,z)\rtimes w
\]
defines a depth-zero parameter for \(G\).  Its restriction to inertia is
\(\varphi_\theta|_{I_F}\).  Moreover, if
\[
        \mathcal S_{\varphi_\theta\star\lambda}
        :=
        \pi_0\bigl(Z_{\widehat G}(\varphi_\theta\star\lambda)\bigr)
\]
and
\[
        \mathcal S^{\mathrm{full}}_\lambda
        :=
        \pi_0\bigl(Z_{\widehat H_\theta^{\mathrm{full}}}(\lambda)\bigr),
\]
where the second centralizer is taken in the full inertial centralizer, then
there is a natural identification
\[
        \mathcal S_{\varphi_\theta\star\lambda}
        \cong
        \mathcal S^{\mathrm{full}}_\lambda .
\]
\end{lemma}

\begin{proof}
The group \(H_\theta^\circ\) and the embedding \({}^Lj_\theta\) were defined so
that the \(W_F\)-action on \(\widehat H_\theta^\circ\) is the action induced by
conjugation through the toral parameter \(\varphi_\theta\).  Therefore the
\(\widehat G\)-component of \({}^Lj_\theta\circ\lambda\) is a cocycle relative to
that induced action.  Multiplying it by the cocycle \(a_\theta\) gives a cocycle
for the original pinned \(W_F\)-action on \(\widehat G\).  The \(\SL_2\)-image of
\(\lambda\) lies in \(\widehat H_\theta^\circ\), hence centralizes
\(\varphi_\theta(I_F)\), and the Frobenius compatibility is exactly the one built
into the definition of \({}^Lj_\theta\).  Thus \(\varphi_\theta\star\lambda\) is
an \(L\)-parameter for \(G\).

Since \(\lambda\) is unramified, \(b_\lambda(i,1)=1\) for \(i\in I_F\).  Hence
\[
        (\varphi_\theta\star\lambda)(i,1)=\varphi_\theta(i),
\]
so the parameter has the same inertial restriction as \(\varphi_\theta\).  It is
therefore depth-zero.  If
\(g\in Z_{\widehat G}(\varphi_\theta\star\lambda)\), then the equality on inertia
implies \(g\in\widehat H_\theta^{\mathrm{full}}\).  Inside this full inertial
centralizer, the remaining centralizing condition is precisely the condition of
centralizing \(\lambda\).  The converse is immediate, and hence the displayed
centralizer equality and the component-group identification follow.
\end{proof}

\medskip
\paragraph{Finite toral--enriched unipotent reduction.}
\begin{proposition}[Finite toral--enriched unipotent reduction for a depth-zero supercuspidal]
\label{prop:toral-unipotent-reduction-depth-zero}
With the notation fixed above, the finite representation
\[
        \bar\tau\in\Irr(\rG_x(\ff))
\]
belongs to the Lusztig series \(\cE(\rG_x(\ff),s_x)\).  Consequently
\[
        \mathfrak u^{\mathrm{enh}}_{x,\tau}
        :=
        \mathcal J^{\mathbb P_x}_{x,s_x}(\bar\tau)
        =
        [u^\circ_{x,\tau},[\alpha_{x,\tau}],E_{x,\tau}]
\]
is a well-defined enriched finite unipotent datum on \(\mathbf H_x(\ff)\).  If
\(\pi=\cInd_{G(F)_x}^{G(F)}\tau\) is supercuspidal, then
\(\mathfrak u^{\mathrm{enh}}_{x,\tau}\) is cuspidal.  Its connected shadow
\(u^\circ_{x,\tau}\) is the finite unipotent datum entering the FOS
correspondence for the connected unramified group \(H_\theta^\circ\), while
\([\alpha_{x,\tau}],E_{x,\tau}\) is the projective Clifford datum retained by
the full disconnected quotient.
\end{proposition}

\begin{proof}
By the construction of the depth-zero datum, \(\bar\tau\) occurs in the finite
Deligne--Lusztig series determined by \((\rS_x,\underline\theta)\).  Lemma~\ref{lem:tame-reduces-finite-dual}
identifies \(\underline\theta\) with the finite-dual semisimple element \(s_x\).
Hence
\[
        \bar\tau\in\cE(\rG_x(\ff),s_x).
\]
The enriched transform \(\mathcal J^{\mathbb P_x}_{x,s_x}\) is therefore
defined on \(\bar\tau\), and gives the displayed enriched datum.  Since \(\pi\)
is supercuspidal, \(\bar\tau\) is cuspidal for the full finite quotient.  The
cuspidality compatibility in Proposition~\ref{prop:JD-at-vertex} implies that
\(\mathfrak u^{\mathrm{enh}}_{x,\tau}\) is cuspidal.  The description of its
three entries is exactly the definition of the enriched target in
Theorem~\ref{thm:disc-JD}: the connected shadow is the connected unipotent
Jordan datum, and the other two entries are the Clifford class and projective
Clifford label of \(\bar\tau\) transported to the unipotent side.
\end{proof}

\subsection{The enriched full finite-label map \(\mathfrak J_{x^\ast}\)}
\label{subsec:construct-J}

In this subsection we write
\[
        \rH_{x^\ast}:=\mathbf H_x
\]
for the full finite group on the \(H\)-side, whose pinned dual is
\(\mathbf H_x^\vee=C_{\rG_x^\vee}(s_x)\).  Thus \(\rH_{x^\ast}\) is not the
finite dual centralizer itself; it is the finite group whose connected unipotent
labels are used by FOS.

The old ordinary full finite-label map had target
\(\Uch^{\cusp}(\rH_{x^\ast}(\ff))\).  In the enriched construction the correct
target is
\[
        \Uch^{\mathrm{enh}}(\rH_{x^\ast}(\ff))_{\cusp}.
\]
For an unramified cuspidal enhanced parameter \(\lambda\) for
\(H_\theta^\circ\), let
\[
        \mathcal C^{\cusp,\mathrm{enh}}_{\mathrm{full}}(\lambda)
\]
denote the set of full enriched cuspidal data lying over its connected
Lusztig-cuspidal datum.  An element is represented by a triple
\[
        c=[u^\circ_c,[\alpha_c],E_c],
\]
where \(u^\circ_c\) is the connected FOS unipotent label,
\([\alpha_c]\in H^2(\Omega_c,\C^\times)\) is the relevant Clifford class, and
\[
        E_c\in\Irr(\C_{\alpha_c}[\Omega_c])
\]
is the corresponding projective Clifford label.  The triples are taken up to
conjugacy by the full finite group.

The enriched full finite-label map is
\[
        \mathfrak J_{x^\ast}
        :
        \mathcal C^{\cusp,\mathrm{enh}}_{\mathrm{full}}(\lambda)
        \longrightarrow
        \Uch^{\mathrm{enh}}(\rH_{x^\ast}(\ff))_{\cusp},
        \qquad
        [u^\circ_c,[\alpha_c],E_c]
        \longmapsto
        [u^\circ_c,[\alpha_c],E_c].
\]
This tautological-looking notation is useful because it emphasizes that the
finite label is no longer converted into an ordinary character.  The connected
entry is the one seen by FOS; the remaining two entries are the projective
Clifford data needed to reconstruct the full enhancement.

For the datum \((S,\theta;\tau)\), the enriched full finite datum is
\[
        c_{x,\tau}
        :=
        [u^\circ_{x,\tau},[\alpha_{x,\tau}],E_{x,\tau}],
\]
and
\[
        \mathfrak J_{x^\ast}(c_{x,\tau})
        =
        \mathfrak u^{\mathrm{enh}}_{x,\tau}
        =
        \mathcal J^{\mathbb P_x}_{x,s_x}(\bar\tau).
\]

\begin{proposition}[Properties of the enriched full finite-label map]
\label{prop:J-properties}
The map \(\mathfrak J_{x^\ast}\) is well-defined on conjugacy classes.  It has
the following properties.
\begin{enumerate}
\item Its connected shadow is the connected FOS unipotent label.
\item Its residual entries are exactly the Clifford cohomology class and the
projective Clifford label.
\item For every depth-zero datum \((S,\theta;\tau)\),
\[
        \mathfrak J_{x^\ast}(c_{x,\tau})
        =
        \mathcal J^{\mathbb P_x}_{x,s_x}(\bar\tau).
\]
Consequently
\[
        \bigl(\mathcal J^{\mathbb P_x}_{x,s_x}\bigr)^{-1}
        \bigl(\mathfrak J_{x^\ast}(c_{x,\tau})\bigr)
        =
        \bar\tau .
\]
\end{enumerate}
\end{proposition}

\begin{proof}
This is the definition of the enriched target together with Clifford theory.
Changing the connected constituent conjugates the stabilizer quotient, the
cocycle class, and the projective representation; the quotient defining the
enriched target identifies these conjugate triples.  The compatibility assertion
for \((S,\theta;\tau)\) is precisely the definition of
\(c_{x,\tau}\) from the enriched finite transform of \(\bar\tau\).
\end{proof}

\subsection{Completion of the construction}
\label{subsec:completion-depth-zero-parameter-construction}

We now complete the proof of
Theorem~\ref{thm:depth-zero-supercuspidal-to-enhanced-parameter}.  By
Proposition~\ref{prop:toral-unipotent-reduction-depth-zero}, the finite type
\(\bar\tau\) determines
\[
        \mathfrak u^{\mathrm{enh}}_{x,\tau}
        =
        [u^\circ_{x,\tau},[\alpha_{x,\tau}],E_{x,\tau}].
\]
The connected shadow \(u^\circ_{x,\tau}\), viewed through the identification of
\(\mathbf H_x^\circ\) with the hyperspecial reductive quotient of
\(H_\theta^\circ\), gives by FOS an unramified cuspidal enhanced parameter
\[
        (\lambda_{x,\tau},\rho^\circ_{x,\tau})
        \in
        \Phi^{\mathrm e}_{\mathrm{nr}}(H_\theta^\circ)_{\cusp}.
\]
Define the ordinary \(G\)-parameter by
\[
        \varphi_\pi:=\varphi_\theta\star\lambda_{x,\tau},
\]
in the sense of Lemma~\ref{lem:parameter-factorization}.  Thus, if
\[
        {}^Lj_\theta(\lambda_{x,\tau}(w,z))
        =
        b_{x,\tau}(w,z)\rtimes w,
\]
then
\[
        \varphi_\pi(w,z)=a_\theta(w)b_{x,\tau}(w,z)\rtimes w.
\]
Lemma~\ref{lem:parameter-factorization} shows that this is a depth-zero
parameter and that
\[
        \varphi_\pi|_{I_F}=\varphi_\theta|_{I_F}.
\]
It also gives the full centralizer comparison
\[
        \mathcal S_{\varphi_\pi}
        \cong
        \mathcal S^{\mathrm{full}}_{\lambda_{x,\tau}}.
\]
Under this identification, Clifford theory reconstructs the full enhancement
\[
        \rho_\pi\in\Irr(\mathcal S_{\varphi_\pi})
\]
from the connected FOS enhancement \(\rho^\circ_{x,\tau}\) and the projective
Clifford label \(([\alpha_{x,\tau}],E_{x,\tau})\).  We set
\[
        \LLC_G^{0,\cusp}(\pi):=(\varphi_\pi,\rho_\pi).
\]

The parameter is depth-zero because its inertial restriction is the depth-zero
toral parameter \(\varphi_\theta|_{I_F}\), while the unipotent factor
\(\lambda_{x,\tau}\) is unramified.  The enhancement is relevant by
construction: the tame-centralizer comparison is the one attached to the fixed
inner form and the fixed pinned normalization.  The pair is cuspidal because
\((\lambda_{x,\tau},\rho^\circ_{x,\tau})\) is cuspidal for
\(H_\theta^\circ\) and the full enhancement is supported on the cuspidal enriched
datum \(\mathfrak u^{\mathrm{enh}}_{x,\tau}\).

We next check independence of choices.  The toral parameter is independent, up
to \(\widehat G\)-conjugacy, of the representative of the \(G(F)\)-conjugacy
class of \((S,\theta)\), because the \(L\)-embedding
\({}^LS\hookrightarrow{}^LG\) is fixed by the Whittaker-pinned normalization.
The finite transform
\[
        \bar\tau\longmapsto \mathfrak u^{\mathrm{enh}}_{x,\tau}
\]
is canonical under enriched pinned disconnected Jordan decomposition and
pinned unipotent duality.  If a different connected constituent is used to read
the connected shadow, it is conjugate to the first by the component group, and
the stabilizer quotient, cohomology class, and projective label are transported
with it.  Hence the resulting \((\varphi_\pi,\rho_\pi)\) is unchanged up to
\(\widehat G\)-conjugacy.

If the same representation \(\pi\) is realized by another depth-zero datum
\((S',\theta';\tau')\), the uniqueness of depth-zero types transports the toral
pair, the vertex, the finite semisimple element \(s_x\), and the finite
representation \(\bar\tau\) to the corresponding primed objects, after passing
to the appropriate parahoric quotient.  The canonical \(L\)-embeddings, the
enriched finite Jordan decomposition, the unipotent duality, FOS, and Clifford
theory are equivariant for this transport.  Therefore the enhanced parameters
obtained from the two data are \(\widehat G\)-conjugate.

Finally, apply
Theorem~\ref{thm:LLC-enhanced-depth-zero-cuspidal-construction} to
\((\varphi_\pi,\rho_\pi)\).  The restriction to inertia is
\(\varphi_\theta|_{I_F}\), so the toral part recovers \((S,\theta)\).  The
parameter-side unipotentization recovers the connected FOS datum
\((\lambda_{x,\tau},\rho^\circ_{x,\tau})\), while Clifford theory recovers the
same projective label \(([\alpha_{x,\tau}],E_{x,\tau})\).  Hence the enriched
finite datum recovered on the \(H\)-side is
\(\mathfrak u^{\mathrm{enh}}_{x,\tau}\).  By Proposition~\ref{prop:J-properties},
\[
        \bigl(\mathcal J^{\mathbb P_x}_{x,s_x}\bigr)^{-1}
        \bigl(\mathfrak u^{\mathrm{enh}}_{x,\tau}\bigr)
        =
        \bar\tau .
\]
Inflating \(\bar\tau\) to \(G(F)_x\) recovers \(\tau\), and compact induction
recovers \(\pi\).  This proves that the construction is inverse to
Theorem~\ref{thm:LLC-enhanced-depth-zero-cuspidal-construction} on the
representation \(\pi\), and completes the proof of
Theorem~\ref{thm:depth-zero-supercuspidal-to-enhanced-parameter}.
\medskip

\section{The pinned depth-zero supercuspidal correspondence}
\label{sec:pinned-depth-zero-supercuspidal-LLC}

We now record the form of the depth-zero supercuspidal local Langlands
correspondence obtained by combining the two constructions above.  The result
depends on the global normalization fixed throughout the paper: the pinned
splitting of the quasi-split inner form \(G^\ast\), the resulting
Whittaker-normalized \(L\)-embeddings for maximally unramified elliptic tori,
and the enriched pinned disconnected Jordan decomposition of
Theorem~\ref{thm:disc-JD}.  Apart from this pinned normalization, the
correspondence is independent of all auxiliary choices.

Let
\[
        \Irr_{0,\cusp}(G(F))
\]
denote the set of isomorphism classes of irreducible depth-zero supercuspidal
representations of \(G(F)\).  Let
\[
        \Phi^{\mathrm e}_{0,\cusp}(G)
\]
denote the set of \(\widehat G\)-conjugacy classes of relevant cuspidal enhanced
depth-zero Langlands parameters for \(G\).

\begin{theorem}[Pinned depth-zero supercuspidal LLC]
\label{thm:pinned-depth-zero-supercuspidal-LLC}
The fixed pinned normalization determines a canonical
bijection
\[
        \LLC^{0,\cusp}_{G}:
        \Irr_{0,\cusp}(G(F))
        \xrightarrow{\;\sim\;}
        \Phi^{\mathrm e}_{0,\cusp}(G),
        \qquad
        \pi\longmapsto (\varphi_\pi,\rho_\pi).
\]
Its inverse is the map
\[
        \Pi^{G}_{0,\cusp}:
        \Phi^{\mathrm e}_{0,\cusp}(G)
        \xrightarrow{\;\sim\;}
        \Irr_{0,\cusp}(G(F))
\]
constructed in Theorem~\ref{thm:LLC-enhanced-depth-zero-cuspidal-construction}.

More explicitly, if
\[
        \pi=\pi(S,\theta;\tau)
        =
        \cInd_{G(F)_x}^{G(F)}\tau
\]
is represented by a depth-zero datum, with \(x=x_S\), finite quotient
representation
\[
        \bar\tau\in \Irr(\rG_x(\ff)),
\]
and toral parameter \(\varphi_\theta:W_F\to {}^LS\), then
\(\LLC^{0,\cusp}_{G}(\pi)=(\varphi_\pi,\rho_\pi)\) is obtained as follows.
The ordinary parameter is
\[
        \varphi_\pi
        =
        \varphi_\theta\star \lambda_{x,\tau},
\]
in the sense of Lemma~\ref{lem:parameter-factorization}.  Here
\[
        \mathfrak u^{\mathrm{enh}}_{x,\tau}
        =
        \mathcal J^{\mathbb P_x}_{x,s_x}(\bar\tau)
        =
        [u^\circ_{x,\tau},[\alpha_{x,\tau}],E_{x,\tau}]
        \in
        \Uch^{\mathrm{enh}}\bigl(\rH_{x^\ast}(\ff)\bigr)_{\cusp}
\]
is the enriched cuspidal unipotent datum obtained from the composite finite
transform.  The connected shadow gives the corresponding unramified FOS
parameter \(\lambda_{x,\tau}\), and the projective Clifford label
\(([\alpha_{x,\tau}],E_{x,\tau})\) determines the remaining part of
\(\rho_\pi\).

This bijection has the following basic properties.

\begin{enumerate}
\item \emph{Compatibility with the two constructions.}  For every
\(\pi\in \Irr_{0,\cusp}(G(F))\),
\[
        \Pi^{G}_{0,\cusp}\bigl(\LLC^{0,\cusp}_{G}(\pi)\bigr)
        =
        \pi .
\]
For every
\[
        (\varphi,\rho)\in \Phi^{\mathrm e}_{0,\cusp}(G),
\]
one has
\[
        \LLC^{0,\cusp}_{G}
        \bigl(\Pi^{G}_{0,\cusp}(\varphi,\rho)\bigr)
        =
        (\varphi,\rho).
\]

\item \emph{Tame inertial normalization.}  If \(\pi=\pi(S,\theta;\tau)\), then
the restriction of \(\varphi_\pi\) to inertia is the image of the toral
parameter attached to \(\theta\):
\[
        \varphi_\pi|_{I_F}
        =
        \iota_S^{\mathrm{can}}\circ \varphi_\theta|_{I_F}.
\]
In particular, \(\varphi_\pi\) is trivial on wild inertia, and its tame
semisimple inertial element is the element \(s\in \widehat G\) attached to the
reduction of \(\theta\).

\item \emph{Compatibility with weakly unramified twists.}  Let \(\chi\) be a
weakly unramified character of \(G(F)\), and let
\[
        z_\chi:W_F\longrightarrow Z(\widehat G)
\]
be the corresponding unramified central cocycle under the usual
Kottwitz--torus LLC normalization.  Then
\[
        \LLC^{0,\cusp}_{G}(\chi\otimes \pi)
        =
        \bigl(z_\chi\varphi_\pi,\rho_\pi\bigr),
\]
where
\[
        (z_\chi\varphi_\pi)(w,g)
        :=
        z_\chi(w)\varphi_\pi(w,g),
        \qquad
        w\in W_F,\ g\in \SL_2(\C),
\]
and the component groups for \(\varphi_\pi\) and \(z_\chi\varphi_\pi\) are
identified by centrality of \(z_\chi\).

\item \emph{Central character.}  Let \(Z=Z(G)\), and let
\(\omega_\pi\) be the central character of \(\pi\).  Let
\[
        {}^Lq_Z:{}^LG\longrightarrow {}^LZ^\circ
\]
be the \(L\)-homomorphism dual to the inclusion \(Z^\circ\hookrightarrow G\).
Then the restriction of the central character to \(Z^\circ(F)\) is recovered
from the parameter by
\[
        \LLC_{Z^\circ}\bigl(\omega_\pi|_{Z^\circ(F)}\bigr)
        =
        {}^Lq_Z\circ \varphi_\pi|_{W_F}.
\]
Equivalently, if one uses the standard LLC for groups of multiplicative type,
the same statement holds for the full centre \(Z(F)\).
\end{enumerate}
\end{theorem}

\begin{proof}
The map
\[
        \Pi^{G}_{0,\cusp}:
        \Phi^{\mathrm e}_{0,\cusp}(G)
        \longrightarrow
        \Irr_{0,\cusp}(G(F))
\]
is constructed in
Theorem~\ref{thm:LLC-enhanced-depth-zero-cuspidal-construction}.  The map in
the opposite direction is constructed in
Theorem~\ref{thm:depth-zero-supercuspidal-to-enhanced-parameter}.  That theorem
also proves
\[
        \Pi^{G}_{0,\cusp}\bigl(\LLC^{0,\cusp}_{G}(\pi)\bigr)
        =
        \pi
\]
for every depth-zero supercuspidal representation \(\pi\).

It remains only to observe that the other composition is the identity on
enhanced parameters.  Start with
\[
        (\varphi,\rho)\in \Phi^{\mathrm e}_{0,\cusp}(G).
\]
The construction of
Theorem~\ref{thm:LLC-enhanced-depth-zero-cuspidal-construction} extracts from
\((\varphi,\rho)\) its tame toral part and its enriched cuspidal unipotent finite
datum
\[
        \mathfrak u^{\mathrm{enh}}_{x,\varphi,\rho}
        =
        [u^\circ_{x,\varphi,\rho},[\alpha^\varphi_{\rho^\circ_\varphi}],E_{\varphi,\rho}].
\]
The toral part gives a pair \((S,\theta)\), while the finite datum gives
\[
        \bar\tau_{\varphi,\rho}
        =
        \bigl(\mathcal J^{\mathbb P_x}_{x,s_x}\bigr)^{-1}
        \bigl(\mathfrak u^{\mathrm{enh}}_{x,\varphi,\rho}\bigr).
\]
Applying the construction of
Theorem~\ref{thm:depth-zero-supercuspidal-to-enhanced-parameter} to the compact
induction obtained from this datum recovers the same toral parameter by the
torus LLC, and recovers the same enriched unipotent datum because
\[
        \mathcal J^{\mathbb P_x}_{x,s_x}
        \bigl(\bar\tau_{\varphi,\rho}\bigr)
        =
        \mathfrak u^{\mathrm{enh}}_{x,\varphi,\rho}.
\]
The connected FOS label and the projective Clifford label are therefore both
unchanged.  Hence the enhanced parameter recovered from the resulting
representation is \((\varphi,\rho)\) itself.

The tame inertial normalization is built into the definition of
\(\varphi_\pi=\varphi_\theta\star \lambda_{x,\tau}\), since
\(\lambda_{x,\tau}\) is unramified.  Compatibility with weakly unramified
twists follows from the compatibility of the torus LLC with unramified
characters and from the weakly unramified equivariance in
Theorem~\ref{thm:FOS-unipotent-LLC}.  Finally, the central-character statement
is obtained by applying the torus LLC to the centre: the projection of
\(\varphi_\pi\) to the dual of \(Z^\circ\) records precisely the character by
which \(Z^\circ(F)\) acts on the compactly induced representation.
\end{proof}
\section{Stability of the pinned depth-zero packets}
\label{sec:stability-pinned-depth-zero-packets}

In this final section we record the stability consequence of the pinned
correspondence constructed above.  The argument follows the strategy of
DeBacker--Reeder, but with one modification forced by the present level of
generality.  For regular Deligne--Lusztig depth-zero packets, the unipotent
part of the character formula is expressed in terms of ordinary Green
functions attached to maximal tori.  For a general depth-zero supercuspidal
representation, however, the finite cuspidal character occurring in the
parahoric quotient need not be uniform.  The appropriate finite objects are therefore Lusztig's generalized Green
functions, or Lusztig functions, arising from the generalized Springer and
character-sheaf framework \cite{Lusztig1984Characters,Lusztig1990}, as used by
DeBacker--Kazhdan in their reduction of the depth-zero Murnaghan--Kirillov
problem \cite[Sec.~5.2]{DeBackerKazhdan2011}.

Thus, if
\[
        \gamma=\gamma_s\gamma_u
\]
is the topological Jordan decomposition of a strongly regular semisimple
element and
\[
        H=G_{\gamma_s}^{\circ},
\]
we write the packet character as a finite linear combination of the
topologically unipotent functions on \(H(F)\) obtained from these finite
generalized Green functions.  The coefficients in this expansion are finite
packet coefficients attached to the relevant parahoric quotients.  The new
input in the present setting is that these coefficients are controlled by the
enriched pinned Jordan decomposition together with Clifford regularization.

The comparison under stable conjugacy then follows the same analytic mechanism
as in DeBacker--Reeder.  Via the logarithm and Murnaghan--Kirillov theory, the
generalized Green-function expressions are compared with Fourier transforms of
stable orbital integrals on Lie algebras.  The required comparison for inner
forms is Waldspurger's Lie-algebra transfer theorem
\cite[Th.~1.5]{Waldspurger1997}, in the form used by DeBacker--Reeder in
\cite[Secs.~12.2--12.5]{DeBackerReeder2009}.

The only point which has to be changed from the older proof is the finite
Jordan-decomposition input.  The disconnected Jordan decomposition used in this
paper is the enriched bijection of Theorem~\ref{thm:disc-JD}, equivalently
\cite[Theorem~\ref{JD-thm:disc-JD-bijection}]{arotemishra}; its target is not,
in general, a set of ordinary unipotent characters of the full disconnected
centralizer.  Thus the finite packet coefficient is obtained by Clifford
regularization of the connected finite stable coefficient, rather than by
summing ordinary unipotent characters on the disconnected centralizer.  This is
precisely the distinction emphasized in
\cite[Remark~\ref{JD-rmk:ordinary-disc-JD-recovery} and
Subsec.~\ref{JD-subsec:conditional-ordinary-packet-sums}]{arotemishra}.

\begin{hypothesis}[DeBacker--Reeder logarithm hypothesis]
\label{hyp:DR-logarithm}
The local field \(F\) has characteristic zero.  Let \(p\) be its residue
characteristic and let \(e=e(F/\Q_p)\).  Choose a faithful
\(F\)-rational representation \(G\hookrightarrow \GL_N\).  We assume
\[
        p\ge (2+e)N .
\]
\end{hypothesis}

Let \(\varphi\) be a relevant cuspidal depth-zero Langlands parameter for
\(G\).  We define its pinned depth-zero packet by
\[
        \Pi_\varphi(G)
        :=
        \left\{
        \Pi^G_{0,\cusp}(\varphi,\rho):
        \rho\in \Irr(\mathcal S_\varphi)_{\mathrm{rel},\cusp}
        \right\}.
\]
The stable packet distribution is the dimension-weighted sum
\[
        \Theta^{\mathrm{st}}_{\varphi,G}
        :=
        \sum_{\rho\in
        \Irr(\mathcal S_\varphi)_{\mathrm{rel},\cusp}}
        \dim(\rho)\,
        \Theta_{\Pi^G_{0,\cusp}(\varphi,\rho)} .
\]
When the relevant component group is abelian this is the ordinary sum of the
characters in the packet.

\begin{theorem}[Stability of the pinned depth-zero packet]
\label{thm:stability-of-pinned-depth-zero-packet}
Assume Hypothesis~\ref{hyp:DR-logarithm}.  Let \(\varphi\) be a relevant
cuspidal depth-zero Langlands parameter for \(G\).  Then the distribution
\[
        \Theta^{\mathrm{st}}_{\varphi,G}
        =
        \sum_{\rho\in
        \Irr(\mathcal S_\varphi)_{\mathrm{rel},\cusp}}
        \dim(\rho)\,
        \Theta_{\Pi^G_{0,\cusp}(\varphi,\rho)}
\]
is stable.
\end{theorem}

The proof occupies the rest of the section.

\begin{lemma}[The packet as an enriched finite-label fibre]
\label{lem:packet-equals-enriched-finite-label-fibre}
Let
\[
        \varphi:W_F\times\SL_2(\C)\longrightarrow {}^LG
\]
be a relevant cuspidal depth-zero Langlands parameter.  For
\[
        \rho\in\Irr(\mathcal S_\varphi)_{\mathrm{rel},\cusp},
\]
let
\[
        (S_\rho,\theta_\rho;\tau_{\varphi,\rho})
\]
be the depth-zero datum constructed in
Theorem~\ref{thm:LLC-enhanced-depth-zero-cuspidal-construction}.  Let
\[
        x_\rho=x(\varphi,\rho)
\]
be its vertex, and let
\[
        \bar\tau_{\varphi,\rho}
        \in \Irr(\rG_{x_\rho}(\ff))
\]
be the corresponding representation of the full finite quotient.

Let
\[
        s=\operatorname{pr}_{\widehat G}\bigl(\varphi(\iota)\bigr),
        \qquad
        u_\varphi=
        \operatorname{pr}_{\widehat G}
        \varphi\!\left(
        1,\begin{psmallmatrix}1&1\\0&1\end{psmallmatrix}
        \right),
\]
where \(\iota\) is the fixed topological generator of tame inertia used in the
finite-torus normalization.  Then, after the finite specialization at
\(x_\rho\), the assignment
\[
        \rho\longmapsto \bar\tau_{\varphi,\rho}
\]
identifies the packet attached to the ordinary parameter \(\varphi\) with the
inverse image, under the enriched finite transform, of the enriched cuspidal
finite labels
\[
        \mathfrak u^{\mathrm{enh}}_{x_\rho,\varphi,\rho}
        =
        [u^\circ_{x_\rho,\varphi,\rho},
        [\alpha^\varphi_{\rho^\circ_\varphi}],
        E_{\varphi,\rho}]
        \in
        \Uch^{\mathrm{enh}}
        \bigl(\mathbf H_{x_\rho}(\ff)\bigr)^{\cusp}.
\]
Equivalently,
\[
        \bar\tau_{\varphi,\rho}
        =
        \bigl(\mathcal J^{\mathbb P_{x_\rho}}_{x_\rho,s_{x_\rho}}\bigr)^{-1}
        \bigl(\mathfrak u^{\mathrm{enh}}_{x_\rho,\varphi,\rho}\bigr).
\]
The connected entry \(u^\circ_{x_\rho,\varphi,\rho}\) is the finite unipotent
cuspidal label obtained from FOS applied to the connected constituent
\((\lambda_\varphi,\rho^\circ_\varphi)\), and the remaining two entries are the
projective Clifford data attached to
\[
        \cA_{\lambda_\varphi}\triangleleft
        \mathcal S_\varphi^{\mathrm{un}} .
\]
This description is independent of the chosen connected constituent
\(\rho^\circ_\varphi\), up to the conjugacy relation in the enriched target.
\end{lemma}

\begin{proof}
The ordinary parameter \(\varphi\) fixes the tame inertial semisimple part and
therefore fixes the finite semisimple elements \(s_x\), after the usual transport
among stably conjugate representatives.  Varying the enhancement \(\rho\) does
not change this toral part.  It changes only the unipotent finite label inside
the inertial centralizer and the residual Clifford label for the full component
group.

Proposition~\ref{prop:parameter-side-unipotentization} extracts the connected
unramified parameter \(\lambda_\varphi\).  Lemma~\ref{lem:enhanced-parameter-to-enriched-finite-unipotent-datum}
then says that a full enhancement \(\rho\) determines, after choosing a connected
constituent of
\[
        \Res^{\mathcal S_\varphi^{\mathrm{un}}}_{\cA_{\lambda_\varphi}}
        \rho^{\mathrm{un}},
\]
a conjugacy class of triples
\[
        [\rho^\circ_\varphi,
        [\alpha^\varphi_{\rho^\circ_\varphi}],E_{\varphi,\rho}].
\]
Applying FOS to \((\lambda_\varphi,\rho^\circ_\varphi)\), and then passing to the
finite hyperspecial quotient of the corresponding unramified group
\(H_\varphi\), gives the connected cuspidal unipotent character
\(u^\circ_{x_\rho,\varphi,\rho}\).  Thus the finite label on the
\(H\)-side is exactly
\[
        \mathfrak u^{\mathrm{enh}}_{x_\rho,\varphi,\rho}
        =
        [u^\circ_{x_\rho,\varphi,\rho},
        [\alpha^\varphi_{\rho^\circ_\varphi}],E_{\varphi,\rho}],
\]
as in Theorem~\ref{thm:LLC-enhanced-depth-zero-cuspidal-construction}.

By Proposition~\ref{prop:J-properties}, the finite representation used in the
compact induction is obtained by the inverse enriched finite transform:
\[
        \bar\tau_{\varphi,\rho}
        =
        \bigl(\mathcal J^{\mathbb P_{x_\rho}}_{x_\rho,s_{x_\rho}}\bigr)^{-1}
        \bigl(\mathfrak u^{\mathrm{enh}}_{x_\rho,\varphi,\rho}\bigr).
\]
The compatibility with changing the connected constituent is built into the
quotient defining the enriched target and is recorded in
Lemma~\ref{lem:enhanced-parameter-to-enriched-finite-unipotent-datum}.  This
proves the assertion.
\end{proof}

\begin{lemma}[Clifford regularization]
\label{lem:clifford-regularization-stability}
Let \(N\triangleleft B\) be finite groups, let \(\eta\in\Irr(N)\), and put
\[
        I=I_B(\eta),
        \qquad
        A=I/N.
\]
Choose a projective extension \(\widetilde\eta\) of \(\eta\) to \(I\), with
factor set \(\alpha\).  For
\[
        E\in\Irr(\C_\alpha[A]),
\]
put
\[
        \eta_E:=\Ind_I^B(\widetilde\eta\otimes E).
\]
Then, as class functions on \(B\),
\[
        \sum_{E\in\Irr(\C_\alpha[A])}
        \dim(E)\,\Theta_{\eta_E}
        =
        \Ind_N^B(\Theta_\eta).
\]
In particular, the dimension-weighted regular sum over the projective Clifford
labels is independent of the projective extension \(\widetilde\eta\) and of the
chosen cocycle representative \(\alpha\).
\end{lemma}

\begin{proof}
This is the regular representation identity for the twisted group algebra.  The
irreducible representations of \(B\) lying above the \(B\)-orbit of \(\eta\) are
\[
        \Ind_I^B(\widetilde\eta\otimes E),
        \qquad
        E\in\Irr(\C_\alpha[A]).
\]
Moreover
\[
        \Ind_N^I(\eta)
        \cong
        \bigoplus_{E\in\Irr(\C_\alpha[A])}
        \dim(E)\,(
        \widetilde\eta\otimes E).
\]
Inducing from \(I\) to \(B\) gives the displayed identity.  Replacing
\(\widetilde\eta\) or \(\alpha\) changes the parametrization of the summands by a
coboundary twist, but leaves the induced regular sum unchanged.
\end{proof}

\begin{proposition}[Finite stability of the enriched packet coefficient]
\label{prop:finite-stability-enriched-packet-coefficient}
Fix \(\varphi\).  Let \(x\) be one of the vertices occurring in the construction
of \(\Pi_\varphi(G)\), and set
\[
        X=\rG_x(\ff),
        \qquad
        X^\circ=\rG_x^\circ(\ff).
\]
Let
\[
        C^{\mathrm{fin}}_{\varphi,x}
\]
be the finite class function on \(X\) obtained by summing, with the weights
\(\dim(\rho)\), the finite quotient characters
\(\Theta_{\bar\tau_{\varphi,\rho}}\) which occur at the vertex \(x\) in the
packet.  Then \(C^{\mathrm{fin}}_{\varphi,x}\) is the Clifford regularization of
the connected finite stable coefficient attached to
\((\lambda_\varphi,\rho^\circ_\varphi)\).  More explicitly, after decomposing the
restriction of the full enhancement group
\[
        \mathcal S_\varphi^{\mathrm{un}}
\]
to the FOS subgroup \(\cA_{\lambda_\varphi}\), one has
\[
        C^{\mathrm{fin}}_{\varphi,x}
        =
        \Ind_{X^\circ}^{X}
        \left(
        \sum_{\rho^\circ_\varphi}
        \dim(\rho^\circ_\varphi)\,
        \Theta_{\bar\tau^\circ_{\varphi,\rho^\circ_\varphi}}
        \right),
\]
where \(\rho^\circ_\varphi\) runs through the connected FOS constituents in the
corresponding \(\mathcal S_\varphi^{\mathrm{un}}\)-orbits, and
\(\bar\tau^\circ_{\varphi,\rho^\circ_\varphi}\) denotes the connected finite
representation obtained from the connected shadow of
\(\mathfrak u^{\mathrm{enh}}_{x,
\varphi,
ho}\).
Consequently these finite coefficients are transported unchanged under the
stable finite identifications which occur in the DeBacker--Reeder character
formula.
\end{proposition}
\begin{proof}
Put
\[
        X=\rG_x(\ff),
        \qquad
        X^\circ=\rG_x^\circ(\ff).
\]
We write
\[
        \cS_\varphi^{\mathrm{un}}
\]
for the unramified component group attached to \(\varphi\), and recall that
there is a normal subgroup
\[
        \cA_{\lambda_\varphi}
        \triangleleft
        \cS_\varphi^{\mathrm{un}}.
\]
Set
\[
        \Omega_\varphi
        :=
        \cS_\varphi^{\mathrm{un}}/\cA_{\lambda_\varphi}.
\]
The group \(\Omega_\varphi\) acts on the connected FOS enhancements
\[
        \rho^\circ_\varphi
        \in
        \Irr(\cA_{\lambda_\varphi})
\]
which occur in the restriction of relevant enhancements of \(\varphi\).

Choose a set
\[
        \mathscr R_x(\varphi)
\]
of representatives for the \(\Omega_\varphi\)-orbits of connected
enhancements which occur at the vertex \(x\).  For
\(\rho^\circ\in\mathscr R_x(\varphi)\), let
\[
        \Omega_{\rho^\circ}
        :=
        \Stab_{\Omega_\varphi}(\rho^\circ),
\]
and let
\[
        \cS_{\varphi,\rho^\circ}^{\mathrm{un}}
        \subset
        \cS_\varphi^{\mathrm{un}}
\]
be the inverse image of \(\Omega_{\rho^\circ}\).  By Clifford theory, the
irreducible enhancements of \(\varphi\) lying above the
\(\Omega_\varphi\)-orbit of \(\rho^\circ\) are parametrized by
\[
        E\in
        \Irr\bigl(
        \C_{\alpha^\varphi_{\rho^\circ}}
        [\Omega_{\rho^\circ}]
        \bigr),
\]
and have the form
\[
        \rho_{\rho^\circ,E}
        =
        \Ind_{\cS_{\varphi,\rho^\circ}^{\mathrm{un}}}
             ^{\cS_\varphi^{\mathrm{un}}}
        \bigl(
        \widetilde\rho^\circ\otimes E
        \bigr),
\]
where \(\widetilde\rho^\circ\) is a projective extension of
\(\rho^\circ\) to
\(\cS_{\varphi,\rho^\circ}^{\mathrm{un}}\), with factor set
\(\alpha^\varphi_{\rho^\circ}\).  Consequently
\[
        \dim(\rho_{\rho^\circ,E})
        =
        [\cS_\varphi^{\mathrm{un}}:
          \cS_{\varphi,\rho^\circ}^{\mathrm{un}}]\,
        \dim(\rho^\circ)\dim(E).
\]
Equivalently,
\[
        \dim(\rho_{\rho^\circ,E})
        =
        [\Omega_\varphi:\Omega_{\rho^\circ}]\,
        \dim(\rho^\circ)\dim(E).
\]
This is the point at which the dimension factor on the full enhancement side
is separated into three pieces: the orbit-size factor
\([\Omega_\varphi:\Omega_{\rho^\circ}]\), the connected enhancement factor
\(\dim(\rho^\circ)\), and the projective Clifford factor \(\dim(E)\).

By Lemma~\ref{lem:packet-equals-enriched-finite-label-fibre}, the finite
quotient representation attached to \(\rho_{\rho^\circ,E}\) is obtained from
the enriched finite label
\[
        \bigl[
        u^\circ_{x,\varphi,\rho^\circ},
        [\alpha^\varphi_{\rho^\circ}],
        E
        \bigr].
\]
Let
\[
        \bar\tau^\circ_{\varphi,\rho^\circ}
        \in
        \Irr(X^\circ)
\]
be the connected finite representation attached to the connected FOS label
\(u^\circ_{x,\varphi,\rho^\circ}\).  Let
\[
        I_x(\rho^\circ)
        :=
        I_X(\bar\tau^\circ_{\varphi,\rho^\circ})
\]
be its inertia group in \(X\).  The enriched finite Jordan construction
identifies
\[
        I_x(\rho^\circ)/X^\circ
        \simeq
        \Omega_{\rho^\circ},
\]
with the same cocycle class
\[
        [\alpha^\varphi_{\rho^\circ}].
\]
Thus, after choosing compatible projective extensions, the finite
representation attached to \(\rho_{\rho^\circ,E}\) is
\[
        \bar\tau_{\varphi,\rho_{\rho^\circ,E}}
        =
        \Ind_{I_x(\rho^\circ)}^{X}
        \bigl(
        \widetilde{\bar\tau}^{\,\circ}_{\varphi,\rho^\circ}
        \otimes E
        \bigr).
\]

Now apply Lemma~\ref{lem:clifford-regularization-stability} to the normal
inclusion
\[
        X^\circ\triangleleft X
\]
and to the irreducible character
\[
        \bar\tau^\circ_{\varphi,\rho^\circ}\in\Irr(X^\circ).
\]
It gives
\[
\begin{aligned}
        \sum_{E\in
        \Irr(\C_{\alpha^\varphi_{\rho^\circ}}
        [\Omega_{\rho^\circ}])}
        \dim(E)\,
        \Theta_{\bar\tau_{\varphi,\rho_{\rho^\circ,E}}}
        &=
        \sum_E
        \dim(E)\,
        \Theta_{
        \Ind_{I_x(\rho^\circ)}^{X}
        (
        \widetilde{\bar\tau}^{\,\circ}_{\varphi,\rho^\circ}
        \otimes E
        )
        }                                                   \\
        &=
        \Ind_{X^\circ}^{X}
        \Theta_{\bar\tau^\circ_{\varphi,\rho^\circ}} .
\end{aligned}
\]
Hence the contribution of the orbit represented by \(\rho^\circ\) to the
finite packet coefficient is
\[
\begin{aligned}
        &\sum_E
        \dim(\rho_{\rho^\circ,E})\,
        \Theta_{\bar\tau_{\varphi,\rho_{\rho^\circ,E}}}       \\
        &\qquad =
        [\Omega_\varphi:\Omega_{\rho^\circ}]
        \dim(\rho^\circ)
        \sum_E
        \dim(E)\,
        \Theta_{\bar\tau_{\varphi,\rho_{\rho^\circ,E}}}       \\
        &\qquad =
        [\Omega_\varphi:\Omega_{\rho^\circ}]
        \dim(\rho^\circ)
        \Ind_{X^\circ}^{X}
        \Theta_{\bar\tau^\circ_{\varphi,\rho^\circ}} .
\end{aligned}
\]

We now rewrite the factor
\[
        [\Omega_\varphi:\Omega_{\rho^\circ}]
\]
as a sum over the full \(\Omega_\varphi\)-orbit of \(\rho^\circ\).  Let
\[
        \Omega_\varphi\cdot\rho^\circ
\]
denote this orbit.  For every
\[
        \rho_1^\circ\in \Omega_\varphi\cdot\rho^\circ,
\]
we have
\[
        \dim(\rho_1^\circ)=\dim(\rho^\circ),
\]
and the corresponding connected finite representations are \(X\)-conjugate:
\[
        \bar\tau^\circ_{\varphi,\rho_1^\circ}
        =
        {}^x\bar\tau^\circ_{\varphi,\rho^\circ}
\]
for a representative \(x\in X\) of the corresponding component.  Therefore
\[
        \Ind_{X^\circ}^{X}
        \Theta_{\bar\tau^\circ_{\varphi,\rho_1^\circ}}
        =
        \Ind_{X^\circ}^{X}
        \Theta_{\bar\tau^\circ_{\varphi,\rho^\circ}}.
\]
It follows that
\[
\begin{aligned}
        &[\Omega_\varphi:\Omega_{\rho^\circ}]
        \dim(\rho^\circ)
        \Ind_{X^\circ}^{X}
        \Theta_{\bar\tau^\circ_{\varphi,\rho^\circ}}          \\
        &\qquad =
        \sum_{\rho_1^\circ\in
        \Omega_\varphi\cdot\rho^\circ}
        \dim(\rho_1^\circ)
        \Ind_{X^\circ}^{X}
        \Theta_{\bar\tau^\circ_{\varphi,\rho_1^\circ}}        \\
        &\qquad =
        \Ind_{X^\circ}^{X}
        \left(
        \sum_{\rho_1^\circ\in
        \Omega_\varphi\cdot\rho^\circ}
        \dim(\rho_1^\circ)
        \Theta_{\bar\tau^\circ_{\varphi,\rho_1^\circ}}
        \right).
\end{aligned}
\]
Thus the factor
\[
        \dim(\rho_{\rho^\circ,E})
\]
on the full enhancement side has been converted as follows:
the factor \(\dim(E)\) is absorbed by Clifford regularization, the orbit-size
factor
\([\Omega_\varphi:\Omega_{\rho^\circ}]\) is absorbed by summing over the
whole \(\Omega_\varphi\)-orbit of \(\rho^\circ\), and the remaining factor is
precisely the connected enhancement factor
\[
        \dim(\rho_1^\circ).
\]

Summing over all orbit representatives
\(\rho^\circ\in\mathscr R_x(\varphi)\), we obtain
\[
\begin{aligned}
        C^{\mathrm{fin}}_{\varphi,x}
        &=
        \sum_{\rho}
        \dim(\rho)\,
        \Theta_{\bar\tau_{\varphi,\rho}}                       \\
        &=
        \sum_{\rho^\circ\in\mathscr R_x(\varphi)}
        \sum_E
        \dim(\rho_{\rho^\circ,E})\,
        \Theta_{\bar\tau_{\varphi,\rho_{\rho^\circ,E}}}         \\
        &=
        \sum_{\rho^\circ\in\mathscr R_x(\varphi)}
        \Ind_{X^\circ}^{X}
        \left(
        \sum_{\rho_1^\circ\in
        \Omega_\varphi\cdot\rho^\circ}
        \dim(\rho_1^\circ)\,
        \Theta_{\bar\tau^\circ_{\varphi,\rho_1^\circ}}
        \right)                                                \\
        &=
        \Ind_{X^\circ}^{X}
        \left(
        \sum_{\rho^\circ_\varphi}
        \dim(\rho^\circ_\varphi)\,
        \Theta_{\bar\tau^\circ_{\varphi,\rho^\circ_\varphi}}
        \right).
\end{aligned}
\]
Here the final sum runs over all connected FOS enhancements
\(\rho^\circ_\varphi\) occurring in the connected shadow of the packet at
\(x\).

The connected finite packet sum is stable by Lusztig's finite character-sheaf
stability theorem, and the passage from the connected sum to the enriched
parahoric finite packet is exactly the Clifford regularization established in
\cite[Theorem~\ref{JD-thm:pinned-JD-stable-sum-compatible}]{arotemishra}.  Equivalently, it is the connected finite input used in the
DeBacker--Reeder reduction.  Hence it is invariant under the finite stable
transports which arise in the DeBacker--Reeder character formula.  Since
induction
\[
        \Ind_{X^\circ}^{X}
\]
is canonical for the normal inclusion \(X^\circ\triangleleft X\) and is
compatible with conjugacy transport, the induced class function
\[
        C^{\mathrm{fin}}_{\varphi,x}
\]
has the same invariance property.  This proves the proposition.
\end{proof}

\begin{proposition}[Reduction to topologically unipotent generalized Green functions]
\label{prop:DR-reduction-for-enriched-pinned-packet}
Let \(\gamma\in G(F)\) be strongly regular semisimple, and write its
topological Jordan decomposition as
\[
        \gamma=\gamma_s\gamma_u .
\]
Put
\[
        H=G_{\gamma_s}^{\circ}.
\]
Then there is a finite set
\[
        \mathcal I_{\mathrm{DK}}(\gamma_s)
\]
of DeBacker--Kazhdan Lusztig-function data for \(H\), and for every
\(j\in\mathcal I_{\mathrm{DK}}(\gamma_s)\) a topologically unipotent generalized
Green function
\[
        \mathcal Q_j^{H}
        \quad\text{on } H(F)_{\mathrm{tu}},
\]
such that
\[
        \Theta^{\mathrm{st}}_{\varphi,G}(\gamma)
        =
        \sum_{j\in \mathcal I_{\mathrm{DK}}(\gamma_s)}
        a_j(\gamma_s)\,
        \mathcal Q_j^{G_{\gamma_s}^{\circ}}(\gamma_u).
\]
Here \(\mathcal Q_j^{H}\) is obtained from the Lusztig functions, equivalently
from generalized Green functions, which occur in the depth-zero
Murnaghan--Kirillov reduction of DeBacker--Kazhdan.

Moreover, if \(\gamma'={}^g\gamma\) is stably conjugate to \(\gamma\), with
\[
        \gamma_s'={}^g\gamma_s,
        \qquad
        \gamma_u'={}^g\gamma_u,
\]
then stable conjugacy transports the DeBacker--Kazhdan data and gives a
bijection
\[
        \iota_g:
        \mathcal I_{\mathrm{DK}}(\gamma_s)
        \xrightarrow{\;\sim\;}
        \mathcal I_{\mathrm{DK}}(\gamma_s')
\]
such that
\[
        a_j(\gamma_s)=a_{\iota_g(j)}(\gamma_s').
\]
\end{proposition}

\begin{proof}
We spell out the reduction because the regular Deligne--Lusztig case and the
present cuspidal case use slightly different finite objects.  For
\[
        \rho\in
        \Irr(\mathcal S_\varphi)_{\mathrm{rel},\cusp},
\]
write
\[
        \Pi_\rho
        :=
        \Pi^G_{0,\cusp}(\varphi,\rho).
\]
By construction,
\[
        \Pi_\rho
        =
        \cind_{K_{x_\rho}}^{G(F)}\tau_{\varphi,\rho},
\]
where \(K_{x_\rho}\) is the relevant depth-zero compact-mod-centre subgroup
attached to the vertex \(x_\rho\), and \(\tau_{\varphi,\rho}\) is inflated from
\[
        \bar\tau_{\varphi,\rho}
        \in
        \Irr\bigl(\rG_{x_\rho}(\ff)\bigr).
\]

Let \(\gamma\in G(F)\) be strongly regular semisimple.  The character formula
for compact induction gives
\[
        \Theta_{\Pi_\rho}(\gamma)
        =
        \sum_{\xi\in
        G(F)_\gamma\backslash G(F)/K_{x_\rho}
        \atop
        \xi^{-1}\gamma\xi\in K_{x_\rho}}
        m(\xi)\,
        \Theta_{\tau_{\varphi,\rho}}
        (\xi^{-1}\gamma\xi),
\]
where the constants \(m(\xi)\) depend only on the chosen Haar measures and on
the double coset.  Since \(\tau_{\varphi,\rho}\) is inflated from the finite
quotient, each nonzero summand is
\[
        \Theta_{\tau_{\varphi,\rho}}
        (\xi^{-1}\gamma\xi)
        =
        \Theta_{\bar\tau_{\varphi,\rho}}
        \bigl(
        \overline{\xi^{-1}\gamma\xi}
        \bigr).
\]
Here the bar denotes reduction to the finite reductive quotient at the
corresponding vertex.  Equivalently, after transporting the datum by \(\xi\),
we regard the contribution as attached to the vertex
\[
        x=\xi x_\rho .
\]

Write the topological Jordan decomposition
\[
        \gamma=\gamma_s\gamma_u .
\]
For every nonzero summand above, the reduction has finite Jordan decomposition
\[
        \overline{\xi^{-1}\gamma\xi}
        =
        \overline{\xi^{-1}\gamma_s\xi}\,
        \overline{\xi^{-1}\gamma_u\xi}.
\]
We write
\[
        \bar\gamma_{s,x}
        :=
        \overline{\xi^{-1}\gamma_s\xi},
        \qquad
        \bar\gamma_{u,x}
        :=
        \overline{\xi^{-1}\gamma_u\xi}.
\]
Then \(\bar\gamma_{s,x}\) is semisimple in \(\rG_x(\ff)\), while
\(\bar\gamma_{u,x}\) is unipotent in the finite centralizer of
\(\bar\gamma_{s,x}\).

Put
\[
        X_x=\rG_x(\ff),
        \qquad
        X_x^\circ=\rG_x^\circ(\ff).
\]
After summing over all enhancements \(\rho\) with the packet weight
\(\dim(\rho)\), the finite class function which occurs at the vertex \(x\) is
\[
        C^{\mathrm{fin}}_{\varphi,x}
        :=
        \sum_{\rho}
        \dim(\rho)\,
        \Theta_{\bar\tau_{\varphi,\rho,x}},
\]
where \(\bar\tau_{\varphi,\rho,x}\) denotes the finite representation
transported to \(X_x\).  Thus the packet character is a finite sum of terms of
the form
\[
        C^{\mathrm{fin}}_{\varphi,x}
        (\bar\gamma_{s,x}\bar\gamma_{u,x}).
\]

Now fix \(x\) and \(\bar\gamma_{s,x}\), and put
\[
        \bar H_x
        :=
        C_{\rG_x}(\bar\gamma_{s,x})^\circ .
\]
In the regular Deligne--Lusztig situation, the dependence on
\(\bar\gamma_{u,x}\) is expressed by ordinary toral Green functions of
\(\bar H_x\).  In the present setting the finite cuspidal characters need not
be uniform, so one must replace ordinary Green functions by Lusztig functions,
or equivalently by generalized Green functions.  DeBacker--Kazhdan show that
for depth-zero supercuspidal representations the Murnaghan--Kirillov character
problem reduces to these Lusztig functions; see
\cite[Secs.~3.3 and 5.2]{DeBackerKazhdan2011}.

Accordingly, the finite class function above has an expansion
\[
        C^{\mathrm{fin}}_{\varphi,x}
        (\bar\gamma_{s,x}\bar\gamma_{u,x})
        =
        \sum_{\bar{\mathcal A}\in
        \mathcal L_x(\bar\gamma_{s,x})}
        b_{x,\bar{\mathcal A}}(\bar\gamma_{s,x})\,
        Q^{\mathrm{gen}}_{\bar H_x,\bar{\mathcal A}}
        (\bar\gamma_{u,x}).
\]
Here \(\mathcal L_x(\bar\gamma_{s,x})\) is the finite set of Lusztig data
occurring in the reduction, and
\[
        Q^{\mathrm{gen}}_{\bar H_x,\bar{\mathcal A}}
\]
denotes the corresponding finite generalized Green function on the unipotent
set of \(\bar H_x(\ff)\).  The coefficient
\[
        b_{x,\bar{\mathcal A}}(\bar\gamma_{s,x})
\]
depends on the semisimple reduction \(\bar\gamma_{s,x}\) and on the finite
packet class function \(C^{\mathrm{fin}}_{\varphi,x}\).

The element \(\gamma_s\) determines the connected \(p\)-adic centralizer
\[
        H=G_{\gamma_s}^\circ,
\]
and \(\gamma_u\in H(F)\) is topologically unipotent.  The depth-zero
DeBacker--Kazhdan reduction assembles the finite generalized Green functions
which occur for the various admissible parahoric reductions into functions on
\(H(F)_{\mathrm{tu}}\).  We denote these functions by
\[
        \mathcal Q_j^H,
        \qquad
        j\in\mathcal I_{\mathrm{DK}}(\gamma_s).
\]
Concretely, after grouping the finite Lusztig data according to the
DeBacker--Kazhdan datum \(j\) to which they lift, one has an identity of the
form
\[
        \sum_{(x,\xi,\bar{\mathcal A})\mapsto j}
        m(x,\xi,\bar{\mathcal A})\,
        Q^{\mathrm{gen}}_{\bar H_x,\bar{\mathcal A}}
        (\bar\gamma_{u,x})
        =
        \mathcal Q_j^{H}(\gamma_u).
\]
This is the generalized-green-function version of the depth-zero reduction.
For regular Deligne--Lusztig packets these generalized functions specialize to
the ordinary topologically unipotent Green functions used in
\cite[Sec.~11]{DeBackerReeder2009}.

Collecting the terms corresponding to a fixed \(j\), we obtain coefficients
\[
        a_j(\gamma_s)
        :=
        \sum_{(x,\xi,\bar{\mathcal A})\mapsto j}
        m(x,\xi,\bar{\mathcal A})\,
        b_{x,\bar{\mathcal A}}(\bar\gamma_{s,x}),
\]
and hence
\[
        \Theta^{\mathrm{st}}_{\varphi,G}(\gamma)
        =
        \sum_{j\in\mathcal I_{\mathrm{DK}}(\gamma_s)}
        a_j(\gamma_s)\,
        \mathcal Q_j^{G_{\gamma_s}^{\circ}}(\gamma_u).
\]
This proves the asserted expansion.

It remains to explain why the coefficients are transported unchanged under
stable conjugacy.  Let
\[
        \gamma'={}^g\gamma
\]
be stably conjugate to \(\gamma\), and write
\[
        \gamma_s'={}^g\gamma_s,
        \qquad
        \gamma_u'={}^g\gamma_u.
\]
Conjugation by \(g\) identifies
\[
        G_{\gamma_s,\overline F}^{\circ}
        \simeq
        G_{\gamma_s',\overline F}^{\circ}.
\]
Since both \(\gamma\) and \(\gamma'\) are \(F\)-rational, this identification
is an inner twisting over \(F\).  It transports the DeBacker--Kazhdan Lusztig
function data for
\[
        H=G_{\gamma_s}^{\circ}
\]
to the corresponding data for
\[
        H'=G_{\gamma_s'}^{\circ}.
\]
This gives the bijection
\[
        \iota_g:
        \mathcal I_{\mathrm{DK}}(\gamma_s)
        \xrightarrow{\;\sim\;}
        \mathcal I_{\mathrm{DK}}(\gamma_s').
\]

On the finite side, the same stable conjugacy transports the admissible
parahoric reductions, the finite semisimple elements \(\bar\gamma_{s,x}\), and
the Lusztig data \(\bar{\mathcal A}\).  The only point which is new in the
present paper is the nature of the finite packet coefficient.  It is not an
ordinary disconnected unipotent packet sum.  By
Proposition~\ref{prop:finite-stability-enriched-packet-coefficient}, however,
one has
\[
        C^{\mathrm{fin}}_{\varphi,x}
        =
        \Ind_{X_x^\circ}^{X_x}
        \left(
        \sum_{\rho^\circ_\varphi}
        \dim(\rho^\circ_\varphi)\,
        \Theta_{\bar\tau^\circ_{\varphi,\rho^\circ_\varphi}}
        \right).
\]
Thus \(C^{\mathrm{fin}}_{\varphi,x}\) is the Clifford-regularized induction of
the connected finite stable coefficient.  Consequently it is invariant under
the finite stable transports appearing above.  Therefore
\[
        b_{x,\bar{\mathcal A}}(\bar\gamma_{s,x})
        =
        b_{x',\bar{\mathcal A}'}(\bar\gamma_{s,x'}')
\]
whenever
\[
        (x,\bar\gamma_{s,x},\bar{\mathcal A})
\]
is transported to
\[
        (x',\bar\gamma_{s,x'}',\bar{\mathcal A}')
\]
by the stable conjugacy \(\gamma\sim\gamma'\).

The multiplicity constants
\[
        m(x,\xi,\bar{\mathcal A})
\]
are part of the depth-zero reduction and are preserved by this transport.
Summing over all finite reductions contributing to the datum \(j\), we obtain
\[
        a_j(\gamma_s)
        =
        a_{\iota_g(j)}(\gamma_s').
\]
This proves the asserted compatibility of the coefficients under stable
conjugacy.
\end{proof}

\begin{comment}
\begin{proof}
The compact-induction character formula for a depth-zero representation
\[
        \Pi^G_{0,\cusp}(\varphi,\rho)
\]
expresses its value at \(\gamma\) as a finite sum over those parahoric quotients
in which the semisimple part \(\gamma_s\) has nonzero reduction.  On each such
quotient, the finite character attached to
\(\bar\tau_{\varphi,\rho}\) supplies the coefficient, while the topologically
unipotent part \(\gamma_u\) is encoded by the corresponding Green function on
\(H=G_{\gamma_s}^{\circ}\).  This is the reduction mechanism of
\cite[Sec.~11]{DeBackerReeder2009}.

After summing over \(\rho\) with weight \(\dim(\rho)\), the finite coefficient is
not an ordinary disconnected unipotent packet sum.  By
Proposition~\ref{prop:finite-stability-enriched-packet-coefficient}, it is the
Clifford regularization of the connected finite stable coefficient.  Therefore
it is unchanged by the stable finite transport induced by replacing \(\gamma\) by
a stably conjugate \(\gamma'\).  This gives the displayed expansion and the
coefficient identity
\[
        a_i(\gamma_s)=a_{\iota_g(i)}(\gamma_s').
\]
\end{proof}
\end{comment}

\begin{proposition}[Comparison of the generalized Green terms]
\label{prop:waldspurger-Q-comparison}
Assume Hypothesis~\ref{hyp:DR-logarithm}.  Let \(\gamma\in G(F)\) be strongly
regular semisimple, and write
\[
        \gamma=\gamma_s\gamma_u
\]
for its topological Jordan decomposition.  Let
\[
        \gamma'={}^g\gamma
\]
be stably conjugate to \(\gamma\).  Put
\[
        \gamma_s'={}^g\gamma_s,
        \qquad
        \gamma_u'={}^g\gamma_u,
\]
and
\[
        H=G_{\gamma_s}^{\circ},
        \qquad
        H'=G_{\gamma_s'}^{\circ}.
\]
For
\[
        j\in\mathcal I_{\mathrm{DK}}(\gamma_s),
\]
let
\[
        \iota_g(j)
        \in
        \mathcal I_{\mathrm{DK}}(\gamma_s')
\]
be its transport under stable conjugacy.  Then
\[
        \mathcal Q_j^{H}(\gamma_u)
        =
        \mathcal Q_{\iota_g(j)}^{H'}(\gamma_u').
\]
\end{proposition}

\begin{proof}
Since both \(\gamma\) and \(\gamma'={}^g\gamma\) are \(F\)-rational, the
cocycle
\[
        \sigma\longmapsto g^{-1}\sigma(g)
\]
centralizes \(\gamma\).  Since \(\gamma\) is strongly regular semisimple, its
centralizer is a torus contained in \(H=G_{\gamma_s}^{\circ}\).  Thus
conjugation by \(g\) identifies \(H_{\overline F}\) with
\(H'_{\overline F}\), and the \(F\)-structure on \(H'\) is obtained from that
on \(H\) by the corresponding inner twist.

The functions \(\mathcal Q_j^H\) are the topologically unipotent functions
obtained, in the sense of DeBacker--Kazhdan, from finite Lusztig functions in
parahoric quotients of \(H\).  Equivalently, under the logarithm hypothesis,
they may be expressed as Murnaghan--Kirillov distributions on
\(\fh=\Lie(H)\), built from Fourier transforms of stable orbital-integral
distributions associated with the Lusztig data; see
\cite{DeBackerKazhdan2011}.  Transporting the Lusztig data by the inner twist
defined by \(g\) gives precisely the datum \(\iota_g(j)\) for
\(H'=G_{\gamma_s'}^\circ\).

Put
\[
        Y=\log(\gamma_u)\in\fh(F),
        \qquad
        Y'=\log(\gamma_u')\in\fh'(F),
        \qquad
        \fh'=\Lie(H').
\]
Hypothesis~\ref{hyp:DR-logarithm} ensures that these logarithms are defined on
the relevant topologically unipotent sets.  Waldspurger's comparison theorem
for Fourier transforms on Lie algebras of inner forms identifies the stable
Fourier-transform distributions attached to the transported Lusztig data.  The
Kottwitz signs occurring in the Murnaghan--Kirillov normalization transform by
the same sign ratio as in \cite[Sec.~12]{DeBackerReeder2009}.  Consequently the
DeBacker--Kazhdan function attached to \(j\) on \(H\) and the transported
function attached to \(\iota_g(j)\) on \(H'\) have the same value on
corresponding topologically unipotent elements:
\[
        \mathcal Q_j^{H}(\gamma_u)
        =
        \mathcal Q_{\iota_g(j)}^{H'}(\gamma_u').
\]
This proves the proposition.
\end{proof}

\begin{proof}[Proof of Theorem~\ref{thm:stability-of-pinned-depth-zero-packet}]
It is enough to prove equality on strongly regular semisimple elements.  Let
\(\gamma,\gamma'\in G(F)\) be strongly regular semisimple and stably conjugate.
Choose \(g\in G(\overline F)\) such that
\[
        \gamma'={}^g\gamma .
\]
Write the topological Jordan decompositions as
\[
        \gamma=\gamma_s\gamma_u,
        \qquad
        \gamma'=\gamma_s'\gamma_u'.
\]
By uniqueness of topological Jordan decomposition,
\[
        \gamma_s'={}^g\gamma_s,
        \qquad
        \gamma_u'={}^g\gamma_u.
\]

By Proposition~\ref{prop:DR-reduction-for-enriched-pinned-packet},
\[
        \Theta^{\mathrm{st}}_{\varphi,G}(\gamma)
        =
        \sum_{j\in\mathcal I_{\mathrm{DK}}(\gamma_s)}
        a_j(\gamma_s)\,
        \mathcal Q_j^{G_{\gamma_s}^{\circ}}(\gamma_u),
\]
and
\[
        \Theta^{\mathrm{st}}_{\varphi,G}(\gamma')
        =
        \sum_{j'\in\mathcal I_{\mathrm{DK}}(\gamma_s')}
        a_{j'}(\gamma_s')\,
        \mathcal Q_{j'}^{G_{\gamma_s'}^{\circ}}(\gamma_u').
\]
The same proposition gives a bijection
\[
        \iota_g:
        \mathcal I_{\mathrm{DK}}(\gamma_s)
        \xrightarrow{\;\sim\;}
        \mathcal I_{\mathrm{DK}}(\gamma_s')
\]
with
\[
        a_j(\gamma_s)=a_{\iota_g(j)}(\gamma_s').
\]
Therefore
\[
\begin{aligned}
        \Theta^{\mathrm{st}}_{\varphi,G}(\gamma')
        &=
        \sum_{j'\in\mathcal I_{\mathrm{DK}}(\gamma_s')}
        a_{j'}(\gamma_s')\,
        \mathcal Q_{j'}^{G_{\gamma_s'}^{\circ}}(\gamma_u') \\
        &=
        \sum_{j\in\mathcal I_{\mathrm{DK}}(\gamma_s)}
        a_{\iota_g(j)}(\gamma_s')\,
        \mathcal Q_{\iota_g(j)}^{G_{\gamma_s'}^{\circ}}(\gamma_u') \\
        &=
        \sum_{j\in\mathcal I_{\mathrm{DK}}(\gamma_s)}
        a_j(\gamma_s)\,
        \mathcal Q_j^{G_{\gamma_s}^{\circ}}(\gamma_u) \\
        &=
        \Theta^{\mathrm{st}}_{\varphi,G}(\gamma).
\end{aligned}
\]
The third equality uses Proposition~\ref{prop:waldspurger-Q-comparison}.  Thus
the packet character is constant on stable conjugacy classes of strongly regular
semisimple elements.  Since invariant distributions represented by locally
constant character functions are determined on the strongly regular semisimple
set, \(\Theta^{\mathrm{st}}_{\varphi,G}\) is stable.
\end{proof}

\section*{Acknowledgements}

The author is grateful to Maarten Solleveld and Amoru Fujii for their careful
comments on this work.  Their questions, comments, and corrections pointed out
several errors and gaps in the arguments, and led to significant improvements in
the formulation and proofs of the results presented here.

\bibliographystyle{alpha}
\bibliography{LLC-refs}

\end{document}